\documentclass{scrartcl}

\usepackage[normalem]{ulem}

\usepackage{amsmath, bbm}
\usepackage{amssymb}
\usepackage{amsfonts}
\usepackage{amsthm}
\usepackage{amstext}
\usepackage{dsfont}
\usepackage{nicefrac}
\usepackage{amscd}

\usepackage{url}

\usepackage[ansinew]{inputenc}

\usepackage[all,knot,cmtip]{xy}
\xyoption{arc}
\usepackage{graphicx}
\usepackage{color}

\usepackage{hyperref}
\usepackage{enumerate}

\usepackage{authblk}

\numberwithin{equation}{section}

\newcommand*{\alg}{\mathcal}

\theoremstyle{plain}
\newtheorem{lem}{Lemma}[section]
\newtheorem{prop}[lem]{Proposition}

\newtheorem{thm}[lem]{Theorem}

\theoremstyle{definition}
\newtheorem{df}[lem]{Definition}
\newtheorem{rem}[lem]{Remark}
\newtheorem{ex}[lem]{Example}
\newtheorem{cond}[lem]{Condition}
\newtheorem{obs}[lem]{Observation}

\newcommand*{\N}{\mathbb N}
\newcommand*{\Z}{\mathbb Z}
\newcommand*{\R}{\mathbb R}
\newcommand*{\C}{\mathbb C}
\newcommand*{\Q}{\mathbb Q}
\newcommand*{\one}{\mathbbm 1}
\newcommand{\id}{\ensuremath{\textnormal{id}}}		
\newcommand*{\ot}{\otimes}
\newcommand{\algA}{\ensuremath{\mathcal{A}}}		
\newcommand{\algB}{\ensuremath{\mathcal{B}}}		
\newcommand*{\otA}{{\ot_{A_r}}}

\newcommand*{\eps}{\varepsilon}

\newcommand*{\Cat}{\mathcal C}
\newcommand*{\CatD}{\mathcal D}

\newcommand*{\E}{\mathcal E}

\newcommand*{\Op}{\mathcal O}

\newcommand*{\h}{\mathfrak h}

\newcommand*{\bF}{{b}}		
\newcommand*{\Fc}{{F^+}}		
\newcommand*{\Fa}{{F^-}}		
\newcommand*{\Fca}{{F^\pm}}		
\newcommand*{\tFc}{{\tilde F^+}}		
\newcommand*{\tFa}{{\tilde F^-}}		

\newcommand*{\B}{\alg B}

\newcommand*{\mb}{\mathbf}

\newcommand{\Hom}{\textnormal{Hom}}
\newcommand{\cok}{\textnormal{cok}}
\newcommand*{\F}{\mathcal F}		
\newcommand*{\fcG}{\mathcal G}		
\newcommand*{\End}{\textnormal{End}}		
\newcommand*{\Mod}{\textnormal{-Mod}}		
\newcommand*{\BMod}[1][]{\textnormal{-Mod}_{#1}\textnormal{-}}		

\newcommand*{\Vect}{\mathbf{Vect}}	
\newcommand*{\lsl}{\mathfrak{sl}}
\newcommand*{\ev}{\textnormal{ev}}		
\newcommand*{\Gr}{K_0}		
\newcommand*{\DZ}{\mathcal Z}		
\newcommand*{\YD}[1]{\textnormal{-}\mathcal{YD}_{#1}\textnormal{-}}		
\newcommand*{\Rep}{\textnormal{Rep}}		
\newcommand*{\Repfd}{\textnormal{Rep}^{\textnormal{fd}}}		
\newcommand*{\Repss}{\textnormal{Rep}_{\textnormal{ss}}}		
\newcommand*{\RepAss}[1]{\textnormal{Rep}_{#1\textnormal{-ss}}}		
\newcommand*{\simp}[1]{\C_{#1}}		
\newcommand*{\tpmorph}{T}		
\newcommand*{\talg}{{\mathcal T}}	
\newcommand*{\aind}{\alpha}		
\newcommand*{\fxh}{{h}}		
\newcommand*{\fcAlgPtUnc}{{\mathcal I}}		
\newcommand*{\algBim}{{\mathcal B}}		
\newcommand*{\ptfE}{E}		
\newcommand*{\algPtBim}{{\mathcal B_{\ptfE}}}	
\newcommand*{\fcAB}{{\mathcal I}}		
\newcommand*{\wzconst}{t}		
\newcommand*{\Sym}{\textnormal{Sym}}

\begin{document}

\thispagestyle{empty}

\mbox{}
\vskip 3em
\begin{center}\LARGE
Integrable perturbations of conformal field theories \\
and Yetter-Drinfeld modules
\end{center}

\vskip 2em
\begin{center}
{\large 
David B\"ucher ~~ and ~~
Ingo Runkel
\\[1em]
{\small\sf
davidbuecher@gmail.com ~~,~~
ingo.runkel@uni-hamburg.de}
\\[1em]
Fachbereich Mathematik, Universit\"at Hamburg\\
Bundesstra\ss e 55, 20146 Hamburg, Germany
}
\end{center}

\vskip 3em

\begin{abstract}
In this paper we relate a problem in representation theory -- the study of Yetter-Drinfeld modules over certain braided Hopf algebras
-- to a problem in two-dimensional quantum field theory, namely the identification of integrable perturbations of a conformal field theory. A prescription that parallels Lusztig's construction
allows one to read off the quantum group governing the integrable 
symmetry. As an example, we illustrate how the quantum group for the loop algebra of $\lsl(2)$ appears in the integrable structure of the perturbed uncompactified and compactified free boson.
\end{abstract}

\newpage

\tableofcontents

\newpage

\section{Introduction and Summary}

\medskip

There are three main methods to investigate integrable field theories: factorised scattering, lattice discretisations, and perturbations around a conformal field theory (CFT). We will be concerned with the latter, building on the approach of \cite{BLZ96,BLZ97a,BLZ97b,BLZ99} and its formulation and generalisation in terms of perturbed defects in \cite{Ru07,MR09,Ru10}.

The key ingredient in our construction are one-dimensional objects in the two-dimen\-sional CFT called {\em defect lines}. At such a defect line, the fields of the CFT may have singularities or discontinuities. For the particular type of defect we are interested in, the so-called {\em topological defects}, the fields may be discontinuous but not singular and the stress tensor remains continuous across the defect line. Consider the CFT on a cylinder of circumference $L$ and a circular topological defect labelled $X$ winding around the cylinder. This produces an operator $\mathcal{O}(X)$ on the space of states $\mathcal{H}$ of the CFT on a circle, called the {\em defect operator}. Since the defect $X$ is topological, the operator $\mathcal{O}(X)$ commutes with the conformal Hamiltonian $H_0$,
\begin{align}
  \big[ \, H_0 \,,\, \mathcal{O}(X) \, \big] = 0
  \qquad , \quad H_0 = \tfrac{2 \pi}{L}\big(L_0 + \overline L_0 - \tfrac{c+\overline c}{24}\big) \ .
\end{align}

We want to find a systematic way to simultaneously perturb the conformal Hamiltonian and the topological defect such that their commutator still vanishes.
Fix a bulk field $\Phi$ in the space of bulk fields $\mathcal{F}$. The perturbed Hamiltonian is 
\begin{align}
  H_\text{pert}(\Phi) ~=~ H_0 + \int_0^L \!\! \Phi(x) \, dx \ .
\end{align}
Next consider a field $\psi$ which lives on the defect line $X$, i.e.\ is an element in the space of defect fields $\mathcal{F}_{XX}$ on the defect $X$ -- this is discussed in more detail in Section \ref{sec:cyl+def-field}. Inserting the perturbing term $\exp( \int_0^L \psi(x) dx )$ on the defect line $X$ results in the perturbed defect operator $\mathcal{O}(X,\psi)$. 
This operator is defined by formally expanding the exponential into a sum of ordered integrals of defect fields, see Section \ref{ssec:pert-def-and-cat}. 

We want to find pairs $\Phi$, $\psi$ such that $[H_\text{pert}(\Phi),\mathcal{O}(X,\psi)]=0$. 
This problem has two parts. The first part is analytic and concerns the definition of the perturbed theory, convergence of the individual integrals in the series defining $\mathcal{O}(X,\psi)$, convergence of the series itself, the domain and codomain of the resulting operator, etc. These points will not be addressed here; in the special case of the free boson, we will return to these questions in 
\cite{InProgr}. 
In the present paper, we will concentrate on the second and more algebraic part, which we outline next.

\medskip

The starting point are two categories,
\begin{itemize}
\item[$\Cat$ -] the category of representations of the chiral algebra, which we take to be braided monoidal, as is the case for example in rational CFT and for the free boson,
\item[$\CatD$ -] the monoidal category formed by topological defects, or, more precisely, by defects which are transparent to all fields from the holomorphic and anti-holomorphic copy of the chiral algebra.
\end{itemize}
The space of bulk fields $\mathcal{F}$ and the spaces of defect fields $\mathcal{F}_{XX}$, $X \in \CatD$, carry an action of the holomorphic and anti-holomorphic copy of the chiral algebra and can thus be considered as objects in $\Cat \boxtimes \Cat^\mathrm{rev}$, where ``$\boxtimes$'' is the Deligne-product and ``rev'' indicates that the braiding in the second copy of $\Cat$ is inverted. 

The perturbing fields we consider are described as follows.
Choose two representations of the chiral algebra, $F^\pm_\Cat \in \Cat$, and fix a vector in each, $\psi^\pm \in F^\pm_\Cat$. The bulk field $\Phi$ is defined in terms of an intertwiner $\hat b : F^+_\Cat \otimes_{\mathbb{C}} F^-_\Cat \to \mathcal{F}$ to be
\begin{align} \label{eq:intro-phi-def}
  \Phi = \hat b(\psi^+ \otimes \psi^-) \ .
\end{align}
The defect fields, on the other hand, are parametrised by intertwiners $\hat m : F^+_\Cat \otimes_{\mathbb{C}} \one \, \oplus \, \one \otimes_{\mathbb{C}} F^-_\Cat \to \mathcal{F}_{XX}$, 
where $\one$ is the vacuum representation. We set
\begin{align} \label{eq:intro-psiX-def}
  \psi = \hat m(\psi^+  \otimes \Omega + \Omega \otimes \psi^-) \ ,
\end{align}
where $\Omega \in \one$ is the vacuum vector. One could say that ``the bulk field is split into its holomorphic and anti-holomorphic part, and these two parts serve as perturbing defect fields''.

As explained in Section \ref{ssec:pert-def-and-cat}, for rational CFTs there is a braided monoidal functor $\alpha$ from $\Cat \boxtimes \Cat^\mathrm{rev}$ to the monoidal centre $\DZ(\CatD)$ of $\CatD$ \cite{FFRS06}. Writing $U$ for the forgetful functor $\DZ(\CatD) \to \CatD$, the composition $U\alpha$ has the property that for all $R,S \in \Cat$, the space of intertwiners $R \otimes_{\mathbb{C}} S \to \mathcal{F}_{XX}$ is naturally isomorphic to the space of morphisms $U\alpha(R \otimes_{\mathbb{C}} S) \ot_\CatD X \to X$ 
in $\CatD$. For bulk fields one uses $X=\one$, the tensor unit in $\CatD$  describing the trivial defect, since the space of bulk fields satisfies 	$\mathcal{F} = \mathcal{F}_{\one\one}$. 
Let us abbreviate
\begin{align}
  F^+ = \alpha(F^+_\Cat \otimes_{\mathbb{C}} \one) \quad , \quad
  F^- = \alpha(\one \otimes_{\mathbb{C}} F^-_\Cat) \ ,
\end{align}
both of which are objects in $\DZ(\CatD)$. 
Then a morphism 
\begin{align}\label{eq:pairing_b}
	b : U(F^+) \otimes_{\CatD} U(F^-) \to \one 
\end{align}
describes a bulk field via the above natural isomorphism and \eqref{eq:intro-phi-def}, and a morphism $m : U(F^+ \oplus F^-) \otimes_{\CatD} X \to X$ describes a defect field via \eqref{eq:intro-psiX-def}. Since $\Phi$ is determined by $b$, we can write $H_\text{pert}(b)$ instead of $H_\text{pert}(\Phi)$. Similarly, we will write $\mathcal{O}(X,m)$ instead of 	$\mathcal{O}(X,\psi)$. 

The first main result of this paper is the observation that $H_\text{pert}(b)$ and $\mathcal{O}(X,m)$ commute (assuming existence of the operators) if $b$ and $m$ satisfy a simple compatibility condition, which we call the {\em commutation condition}:
\begin{align}
  \raisebox{-0.5\height}{\includegraphics[height=.11\textwidth]{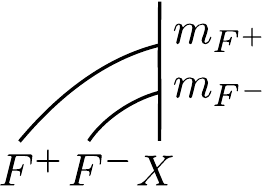}} 
  \quad - 
  \raisebox{-0.5\height}{\includegraphics[height=.11\textwidth]{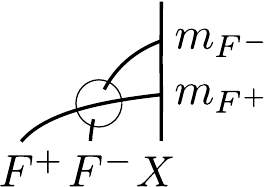}} 
  \quad = \ 
  \raisebox{-0.5\height}{\includegraphics[height=.11\textwidth]{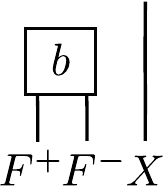}} 
  \quad - \raisebox{-0.5\height}{\includegraphics[height=.11\textwidth]{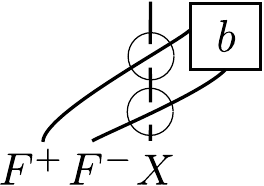}} \quad .
\end{align}
Here we use string diagram notation for morphisms in $\CatD$, the encircled crossings stand for the half-braiding of $F^\pm \in \DZ(\CatD)$, and $m_{F^\pm}$ denotes the restriction of $m : (F^+ \oplus F^-) \otimes_{\CatD} X \to X$ to $F^\pm \otimes_{\CatD} X$. The commutation condition is explained in detail in Sections \ref{ssec:PertDefNonlocalConsCharges} and \ref{sec:comm_cond_YD}. It was first found in \cite{Ru10} for minimal model CFTs; the above formulation generalises to a larger class of models which in particular includes unitary rational CFTs.

The desired systematic approach to finding joint perturbations of $H_0$ and $X$ preserving the vanishing commutator can thus be phrased as: find pairs $b : U(F^+) \otimes_{\CatD} U(F^-) \to \one$ and $(X,m)$ with $m : U(F^+ \oplus F^-) \otimes_{\CatD} X \to X$ 
satisfying the commutation condition. 
Of course, in order to have an integrable perturbation we need to have a whole family of such pairs $(X,m)$ for a fixed $b$, such that their defect operators mutually commute. To formulate this in terms of the category $\CatD$, we first endow the category of pairs $(X,m)$ with a tensor product (Section \ref{sec:tensor_algebras}) and call the resulting monoidal category $\CatD_{F^+ \oplus F^-}$.  
The  tensor product in $\CatD_{F^+ \oplus F^-}$ is constructed to be compatible with the composition of defect operators, 
\begin{align}
  \mathcal{O}(X,m) \, \mathcal{O}(Y,n) = \mathcal{O}\big( (X,m) \otimes (Y,n) \big) \ ,
\end{align}
(cf.\ Section \ref{ssec:tensor-prod-of-pert-def}) 
always assuming that the operators involved exist. We argue in Section \ref{ssec:op+K0} that the defect operator $\mathcal{O}(X,m)$ only depends on the class $[X,m]$ of $(X,m)$ in the Grothendieck ring $K_0(\CatD_{F^+ \oplus F^-})$. Thus we need to find commuting families in $K_0(\CatD_{F^+ \oplus F^-})$. Indeed, it may happen that $(X,m) \ot (Y,n) \ncong (Y,n) \ot (X,m)$ but still $[X,m] \,[Y,n] = [Y,n] \,[X,m]$. 

The category $\CatD_{F^+ \oplus F^-}$ was first constructed in the Cardy-case of rational CFTs and for $F^-=0$ in \cite{Ru07,MR09}; here we generalise it to a larger class of CFTs and include the simultaneous treatment of holomorphic and anti-holomorphic defect fields, which is necessary to treat perturbed CFTs.

\medskip

At this point we can forget about the category $\Cat$ and about the underlying CFT and start directly from $\CatD$, together with objects $F^\pm \in \DZ(\CatD)$ and a morphism $b : U(F^+) \otimes_{\CatD} U(F^-) \to \one$. The problem we need to solve is
\begin{quote}
Find families $\{ (X_u,m_u) \}_u$ in $\CatD_{F^+ \oplus F^-}$ which are as large as possible, such that $b$ and $m_u$ satisfy the commutation condition for all $u$, and $[X_u,m_u] \,[X_v,m_v] = [X_v,m_v] \,[X_u,m_u]$ for all $u,v$.
\end{quote}
This specific formulation of the problem is motivated directly from the CFT application. In the following, we will cast it into a more standard mathematical form. 

The reformulation relies on the existence of tensor algebras in $\CatD$. A sufficient condition for this to be the case is that countable direct sums, compatible with the tensor product, exist in $\CatD$, see Condition \ref{cond:councop}. The latter assumption is not directly met by the categories $\CatD$ one associates to rational CFTs (those are finitely semisimple, i.e.\ their objects are finite direct sums of simple objects). However, in this case we will implicitly replace $\CatD$ with a completed category containing such infinite sums. For the following discussion, we will thus assume that tensor algebras exist in $\CatD$.

We show in Section~\ref{sec:tensor_algebras} that $\CatD_{F^+ \oplus F^-}$ is nothing but the category of modules over $\mathcal{T}(F^+ \oplus F^-)$ in $\CatD$. Since $F^\pm \in \DZ(\CatD)$, the tensor algebra $\mathcal{T}(F^+ \oplus F^-)$ is a Hopf algebra and so the category of $\mathcal{T}(F^+ \oplus F^-)$ modules is monoidal and in fact monoidally isomorphic to $\CatD_{F^+ \oplus F^-}$ (Proposition \ref{prop:tensormod}). 

The map $b$ can be extended to a Hopf pairing $\rho(b) : \mathcal{T}(F^+) \otimes \mathcal{T}(F^-) \to \one$. We can now consider the Yetter-Drinfeld modules for $\mathcal{T}(F^+),\mathcal{T}(F^-),\rho(b)$ in $\CatD$. 
Since the standard formulation of Yetter-Drinfeld modules is done in braided categories \cite{Be95}, for these considerations we will restrict ourselves to the case that $\CatD$ is braided (see, however, Remark \ref{rem:D-not-braided}).
Yetter-Drinfeld modules are $\mathcal{T}(F^+)$-left modules and $\mathcal{T}(F^-)$-right modules such that the left and right actions satisfy a compatibility condition in terms of $\rho(b)$, see Section \ref{sec:comm_cond_YD}. The second main result of this paper is the insight that giving a pair $(X,m)$ satisfying the commutation condition with $b$ is equivalent to giving $X$ the structure of a Yetter-Drinfeld module for $\mathcal{T}(F^+),\mathcal{T}(F^-),\rho(b)$, see Theorem \ref{thm:YD-com-rel-mon}. This correspondence is compatible with the tensor product. We can therefore rephrase the above problem as follows

\begin{quote}
Find families $\{ M_u \}_u$ in the category $\talg(\Fc)\YD{\rho}\talg(\Fa)$ of Yetter-Drinfeld modules, which are as large as possible, such that for all $u,v$, $[M_u] \,[M_v] = [M_v] \,[M_u]$ holds in the Grothendieck ring of $\talg(\Fc)\YD{\rho}\talg(\Fa)$.
\end{quote}

The role of  the quantum double of the algebra of non-local conserved currents (whose modules are Yetter-Drinfeld modules of that algebra) in integrable quantum field theory was also emphasised in \cite{BL93}.

\medskip

We will look at the above question in two closely related examples: the uncompactified and the compactified free boson and their perturbations to sin(h)-Gordon theories (Sections \ref{sec:uncompboson} and \ref{sec:compboson}). This allows us to recover some results of \cite{BLZ96,BLZ97a,BLZ97b} in our formalism. 

For the category $\Cat$ we take the semisimple part of the category of representations of the Heisenberg vertex operator algebra (i.e.\ the $U(1)$-current algebra); this is braided-equivalent to the monoidal category of $\C$-graded vector spaces with a deformed braiding, see Section \ref{ssec:unc_input-data}. The representations $F^\pm_{\Cat}$ are chosen to be both given by 
\begin{align}
  F^+_{\Cat} = F^-_{\Cat} = \C_\omega \oplus \C_{-\omega} \ ,
\end{align}  
where $\omega \in \C^\times$ (the $U(1)$-charge of the perturbing field) and $\C_\omega$ denotes the one-dimensional vector space of grade $\omega$. The highest weight vectors in the corresponding representations of the Heisenberg VOA have conformal weight $h_\text{conf} = \frac12 \omega^2$.

The approach we take is to find a fully faithful exact monoidal functor from the representation category of an appropriate quantum group into the category $\talg(\Fc)\YD{\rho}\talg(\Fa)$ of Yetter-Drinfeld modules. This allows one to transfer results on the product in the Grothen\-dieck ring from quantum groups to $\talg(\Fc)\YD{\rho}\talg(\Fa)$. In fact, if, following Lusztig \cite{Lus94}, one divides the tensor algebras $\talg(\Fc)$ and $\talg(\Fa)$ by the kernel of the Hopf pairing induced by the morphism $b$ in \eqref{eq:pairing_b}, then Yetter-Drinfeld modules of the quotients are in one-to-one correspondence with representations of the relevant quantum group (cf.\ Remark \ref{rem:Lusztig-f}). 

For the uncompactified free boson, we will need the quantum group $\tilde U_\hbar(L\lsl_2)$, where $L\lsl_2$ denotes the loop algebra of $\lsl_2$. It has generators $h$, $e_i^\pm$, $i=0,1$, and the brackets 	$[e_i^+,e_j^-]$ involve $e^{\hbar h}$ for some constant $\hbar \in \C^\times$.
In Theorem \ref{thm:exrepcatfuncmon_uncpt1} we show that the category of $\tilde U_\hbar(L\lsl_2)$-modules on which $h$ acts semisimply fully embeds into $\talg(\Fc)\YD{\rho}\talg(\Fa)$, provided that we choose (the sign is a convention)
\begin{align}
  \hbar = - \tfrac12 i \pi \omega^2 \ .
\end{align}

In the compactified case we have an extra parameter, the compactification radius $r$ (which one may in fact choose complex). The charge $\omega$ of the perturbing field is constrained to satisfy
\begin{align}
  \omega = r^{-1} \, t \quad , \quad t \in \Z \ .
\end{align}
We will restrict ourselves to $t \in 2\Z$ as the treatment simplifies.
The relevant quantum group is now $U_q(L\lsl_2)$, which has generators $e_i^\pm$, $i=0,1$, and $k, k^{-1}$. 
The latter are related to the generator $h$ of $\tilde U_\hbar(L\lsl_2)$ via $k \mapsto e^{\hbar h}$. 
Analogously to the uncompactified case, the category of $U_q(L\lsl_2)$-representations on which $k$ acts semisimply fully embeds into $\talg(\Fc)\YD{\rho}\talg(\Fa)$, provided we choose (in our convention) $t=-2$ and
\begin{align}
  q = \exp\!\big({-}\tfrac12 i \pi \omega^2) \ ,
\end{align}
see Theorem \ref{thm:exrepcatfuncmon_cpt1}.
One important difference between $\tilde U_\hbar(L\lsl_2)$ and $U_q(L\lsl_2)$ is that the latter allows for so-called cyclic representations if $q$ is a root of unity. The corresponding non-local conserved charges behave similar to $Q$-operators, see Section \ref{ssec:comp-TQ_cycl}.

\medskip

This paper is organised as follows. In Section \ref{sec:tensor_algebras} we give the definition of the category $\CatD_{F^+ \oplus F^-}$ and relate it to representations of the tensor algebra. Section \ref{sec:PertDef2D_FT} contains the discussion of CFT with defects: the relation between $\CatD_{F^+ \oplus F^-}$ and defect perturbations is explained, and the commutation condition for compatible bulk and defect perturbations is derived. In Section \ref{sec:comm_cond_YD} the commutation condition is reformulated via Yetter-Drinfeld modules. The free boson example is treated in Sections \ref{sec:uncompboson} and \ref{sec:compboson}.

\subsection*{Acknowledgements}

We thank Anton Alekseev, Ed Corrigan, Alexei Davydov, 
Christian Kassel, Shan Majid, Dimitrios Manolopoulos, 
Carlo Meneghelli, Aliosha Semikhatov 
and J\"org Teschner for helpful suggestions and discussions. 
DB is funded by the SFB 676 ``Particles, Strings, and the Early Universe'' 
of the German Research Foundation (DFG).

\section{Tensor algebras and their modules}
\label{sec:tensor_algebras}

In the present section we recall the notion of a tensor algebra $\talg(F)$ for an object $F$ in a monoidal category $\CatD$ and of modules over $\talg(F)$. We also remind the reader that if $F$ is equipped with a lift to the monoidal centre $\DZ(\CatD)$ of $\CatD$, the tensor algebra $\talg(F)$ carries a Hopf algebra structure. In this case the category of $\talg(F)$-modules is itself monoidal. 
Notions from categorical algebra such as the monoidal centre, (Hopf) algebras in monoidal categories, their modules, etc., are briefly collected in Appendix \ref{appx:alg-etc-in-braided-cats}. 

Some conventions that we will use in the following are: The smallest natural number is $0$, i.e.\ $\N = \{0,1,2,\ldots\}$. 
Vector spaces, algebras, etc., will be over $\C$ (unless when speaking about algebras in more general monoidal categories) and $\Vect$ denotes the category of (not necessarily finite dimensional) complex vector spaces. 
For the graphical representation of morphisms in braided categories 
we use notation of e.g.\ \cite[Ch.\,2.3]{BaKi00}. In particular, our pictures are 
read from bottom to top. For example, a (generic) morphism $f : U \to V$ and the 
braiding $c_{U,V} : U \ot V \to V \ot U$ are represented by
\begin{align}
 \raisebox{-0.5\height}{\includegraphics[height=0.18\textwidth]{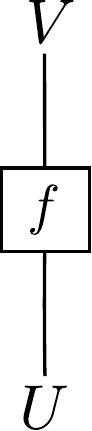}} \qquad \text{and} \qquad 
 \raisebox{-0.5\height}{\includegraphics[height=0.18\textwidth]{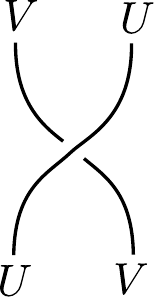}}. 
\end{align}
For an object $V$ of the monoidal centre the half-braiding 
$\varphi_{V,X}: V \ot X \to X \ot V$ will be represented as
\begin{align}
\raisebox{-0.5\height}{\includegraphics[height=0.18\textwidth]{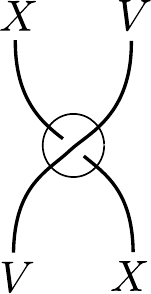}} .
\end{align}

To exhibit the relation to perturbed defect operators -- to be described in Section \ref{sec:PertDef2D_FT} -- it is helpful to reformulate the category of $\talg(F)$-modules and the tensor product on it in terms of the generator $F$ alone. This reformulation does not rely on the existence of the tensor algebra, which can be an advantage in applications. When working with tensor algebras, we will assume the following condition:

\begin{cond} \label{cond:councop}
 $\CatD$ is an abelian monoidal category containing countable direct sums (coproducts) 
 preserved by the tensor product in both variables.
\end{cond}

Fix an abelian monoidal category $(\CatD,\ot,\alpha,\one,\lambda,\rho)$. The {\em tensor algebra} of an object $F \in \CatD$ is an 
algebra $\talg(F)$ in $\CatD$ and a morphism $\iota: F \to \talg(F)$ such that the following universal property 
holds: For each algebra $A$ in $\CatD$ and each morphism 
$f|_F: F \to A$, there exists a unique morphism $f: \talg(F) \to A$ of algebras 
making the following diagram in $\CatD$ commute:
\begin{align} \xymatrix{ F \ar[r]^\iota \ar[dr]_{f|_F} & \talg(F) \ar@{-->}[d]^{\exists !\, f} \\ &A } \end{align}
Under Condition \ref{cond:councop} the tensor algebra $\talg(F)$ exists.
It is unique up to unique isomorphism and we will choose the realisation
\begin{align} \label{eq:talg}
\textstyle \talg(F) = \big( \bigoplus_{n \in \N} F^{\ot n},\, \mu,\, \eta\, \big) 
\end{align}
with unit $\eta := \iota_0$ and multiplication
$\mu := \sum_{n,m \in \N} \iota_{n+m} \circ \alpha_{n,m} \circ (\pi_n \ot \pi_m)$. 
Here,
\begin{align}
\textstyle
\iota_n: F^{\ot n} \to \bigoplus_{n \in \N} F^{\ot n}
\quad , \qquad
\pi_n: \bigoplus_{n \in \N} F^{\ot n} \to F^{\ot n}
\label{eq:iota-pi-tensoralg}
\end{align}
are the canonical embeddings and projections, respectively, and $\alpha_{n,m}$ is a composition of associators moving the brackets from $F^{\ot n} \otimes F^{\ot m}$ to $F^{\ot(n+m)}$. In the notation $F^{\ot n}$ it is understood that the brackets are placed as $((\ldots(F\ot F) \ot \ldots \ot F) \ot F) \ot F$, 
see \cite[Ch.\,VII.3, Thm.\,2]{McL98} for details.

\medskip

If $\Cat$ is a braided monoidal category with braiding $c$ and $(A,\mu_A,\eta_A)$ and $(B,\mu_B,\eta_B)$ are algebras in $\Cat$, 
	then we endow $A \ot B \in \Cat$ with the algebra structure\footnote{ \label{fn:noassoc}Here 
	we do not explicitly write out all associators, unit isomorphisms and 
	brackets between tensor products. This is done to make the expressions 
	easier to read.} 
	\begin{align} \label{eq:prod-on-tensor-of-alg}
		\mu_{A \ot B} = (\mu_A \ot \mu_B) \circ (\id_A \ot c_{B,A} \ot \id_B) ~~, \qquad 
		\eta_{A \ot B} = \eta_A \ot \eta_B \ . 
	\end{align}

Recall the notation $\DZ(\CatD)$ for the monoidal centre of $\CatD$ and suppose that $F \in \DZ(\CatD)$. Then (assuming Condition \ref{cond:councop}) also $\talg(F) \in \DZ(\CatD)$ and $\talg(F)$ carries a Hopf algebra structure. 
The coproduct 
$\Delta: \talg(F) \to \talg(F) \ot \talg(F)$, counit $\varepsilon: \talg(F) \to \one$ 
and antipode $S: \talg(F) \to \talg(F)^{op}$ are obtained from the morphisms 
$\Delta|_F := \iota_1 \ot \eta + \eta \ot \iota_1$, $\epsilon|_F := 0$ 
and $S|_F := -\id_F$ by making use of the universal property (see e.g.\ \cite[Sect.\,2]{Schau96}). The definition of the coproduct implicitly uses convention \eqref{eq:prod-on-tensor-of-alg} for the algebra structure on the tensor product of two algebras.

\medskip

In the remainder of this section, we will give a description of 
the category of $\talg(F)$-modules which involves just the generator $F$ and not the tensor algebra. We write an object $V \in \DZ(\CatD)$ as $V = (\tilde V, \varphi_{V,-})$, where $\tilde V$ is the underlying object in $\CatD$ and $\varphi_{V,-}$ is the half-braiding.

\begin{df} \label{df:catdf}
 Let $F = (\tilde F, \varphi_{F,-}) \in \DZ(\CatD)$. 
 The category $\CatD_F$ is given by
\begin{itemize}
 \item objects: pairs $(X,m)$ consisting of an object $X \in \CatD$ and 
	a morphism $m \in \Hom_{\CatD}(\tilde F \ot X,X)$.
 \item morphisms $(X,m) \to (Y,n)$: morphisms $f \in \Hom_{\CatD}(X,Y)$ 
	such that the diagram 
	\begin{align} \xymatrixcolsep{3pc}
		\xymatrix{ \tilde F \ot X  \ar[r]^-{\id_{\tilde F} \ot f} \ar[d]_{m} 
			&\tilde F \ot Y \ar[d]^{n} \\
			X 	\ar[r]^{f} 	&Y  } \label{eq:DF-morph-diag} \end{align}
	commutes. 
	Composition of morphisms and identity morphisms are those of $\CatD$.
\end{itemize}
\end{df}

We define a tensor product $\ot$ on $\CatD_F$ by $(X,m) \ot (Y,n) := (X \ot Y, \tpmorph(m,n))$, 
 where\footnote{ The 
graphical notation does not capture associators and unit isomorphisms. 
When translating a string diagram into a morphism of the braided category, 
by Mac Lane's coherence theorem there is only one way to add these natural 
isomorphisms. Alternatively, one can understand the diagrams as the process 
a) pass to an equivalent strict category \cite[Ch. XI.3, Thm 1]{McL98}; 
b) write out morphism for string diagram; 
c) transport back to original category. \label{fn:graphics}}
 \begin{align} 
  	\tpmorph(m,n) ~&=~ \left(m \ot \id_Y + (\id_X \ot n) \circ \alpha_{X,\tilde F,Y} 
		\circ (\varphi_{F,X} \ot \id_Y)\right) \circ \alpha_{\tilde F,X,Y}^{-1}, 
		\nonumber\\
&= 
	 \raisebox{-0.5\height}{\includegraphics[height=0.12\textwidth]{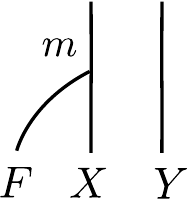}} \
	 + \raisebox{-0.5\height}{\includegraphics[height=0.12\textwidth]{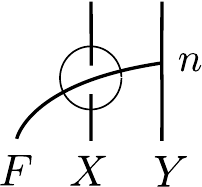}} 
\label{eq:Tpmorph}
 \end{align}
 On morphisms, $\ot$ is just the tensor product of $\CatD$: $f \ot_{\CatD_F} g := f \ot_\CatD g$. The associators and unit isomorphisms are also those from $\CatD$. It is easy to see that if $f,g$ satisfy \eqref{eq:DF-morph-diag}, so does $f \ot g$.

\begin{rem} 
The category $\CatD_F$ is a generalisation of the situation treated in \cite{MR09}. Specifically, suppose that $\CatD$ is braided with braiding $c$. Then there are two braided 
monoidal functors 
\begin{align} \label{eq:Functors=i+,i-}
 \iota^+: \, \CatD \to \DZ(\CatD) ~ , ~~ F \mapsto (F,c_{F,-}) \quad , \qquad 
 \iota^-: \, \CatD \to \DZ(\CatD) ~ , ~~ F \mapsto (F,c_{-,F}^{-1}) \ ,
\end{align}
defined to be identities on hom-sets. The setting of \cite{MR09} amounts to choosing $F = \iota^+(\tilde F) = (\tilde F,c_{\tilde F,-})$ for some $\tilde F$ in $\CatD$. Many properties of $\CatD_F$ carry over from \cite{MR09} to the present more general setting.
\end{rem}

In the following we will list some properties of $\CatD_F$ and then establish the equivalence with modules over the tensor algebra.

\begin{lem} \label{lem:exact_iff} 
 Suppose that the functor $\tilde F \ot (-): \CatD \to \CatD$ is right-exact. 
 Then a complex $(X,m) \stackrel{f}{\to} (X',m') \stackrel{f'}{\to} (X'',m'')$
 is exact at $(X',m')$ in $\CatD_F$ if and only if 
 $X \stackrel{f}{\to} X' \stackrel{f'}{\to} X''$ is exact at $X'$ in $\CatD$. 
\end{lem}
\begin{proof}
 The proof of \cite[Lem.\,2.4]{MR09} also works in the present setting and under the weaker assumption 
 that only $\tilde F \ot (-)$ is right-exact. (In \cite[Lem.\,2.4]{MR09} it is assumed that $\otimes$ is itself right-exact; this enters the proof of \cite[Lem.\,A.2\,(ii)]{MR09} which however only involves $\tilde F \otimes -$).
\end{proof}

\begin{prop} \label{prop:D_F-abelian-etc} 
 Let $F = (\tilde F, \varphi_{F,-})$ be in $\DZ(\CatD)$. Then
\begin{enumerate}
 \item $(\CatD_F,\ot,\alpha,\one,\lambda,\rho)$ is a monoidal category. 
 \item If $\tilde F \ot (-)$ is right-exact then $\CatD_F$ is abelian. 
 \item If the tensor product in $\CatD$ is right-exact, then so is that of $\CatD_F$.
 \item Suppose $\tilde F \ot (-)$ is right-exact. If the tensor product in $\CatD$ is left-exact, then so is that of $\CatD_F$.
\end{enumerate}
\end{prop}
\begin{proof}
 It is straightforward to check that the associators $\alpha$ and unit isomorphisms 
 $\lambda,\rho$ are isomorphisms in $\CatD_F$ and turn $\CatD_F$ 
 into a monoidal category. Abelianness and the exactness properties follow along the same lines as Theorem 2.3 and 2.8 in \cite{MR09}.
\end{proof}

Via the forgetful functor $U: \DZ(\CatD) \to \CatD$ (which forgets the half-braiding of the monoidal centre), an algebra $A$ in $\DZ(\CatD)$ becomes an algebra in $\CatD$. For brevity we will denote the category of $U(A)$-modules in $\CatD$ by $A\Mod_\CatD$ instead of $U(A)\Mod_\CatD$.
	
\begin{prop} \label{prop:tensormod}
Suppose that $\CatD$ satisfies Condition \ref{cond:councop} and that $F \in \DZ(\CatD)$. Then $\CatD_F \cong \talg(F)\Mod_\CatD$
 as monoidal categories.
\end{prop}

\begin{proof}
We will make use of the canonical embeddings and projections $\iota_k$, $\pi_k$ as in \eqref{eq:iota-pi-tensoralg}.
Let $\fcG: \talg(F)\Mod_\CatD \to \CatD_F$ be the functor given on objects by 
 \begin{align}
 	\fcG: (X,m) \mapsto (X, m \circ (\iota_1 \ot \id_X)) \ , 
 \end{align}
 and which acts as the identity on morphisms. 
 
 \medskip\noindent
 {\em $\fcG$ is an equivalence:} Actually, $\fcG$ is even an isomorphism of categories. To construct its inverse, we use the following universal property of $\talg(X)$:
For each object $X$ in $\CatD$ and each morphism 
 $m|_F: F \ot X \to X$, there exists a unique left-action $m: \talg(F) \ot X \to X$ 
 making the following diagram commute in $\CatD$: 
\begin{align} \label{eq:ta_univprop_2}
	\xymatrix{ F \ot X \ar[rr]^{\iota_1 \ot \id_X} \ar[drr]_{m|_F} && \talg(F) \ot X \ar@{-->}[d]^{\exists !\, m} \\ &&X } 
\end{align}
 The action $m$ is given by
 $m := \sum_{k \in \N} m_k \circ (\pi_k \ot \id_X)$, where the $m_k: F^{\ot k} \ot X \to X$ are defined recursively as${}^{\ref{fn:noassoc}}$ 
 $m_0 := \id_X$, $m_{k+1} := m|_F \circ (\id_F \ot m_k)$. 
 For morphisms $f: (X,m|_F) \to (Y,n|_F)$ in $\CatD_F$ we need to check 
 that $f$ is also a morphism for the associated $\talg(F)$ modules. 
 To this end, it suffices to show that $f \circ m_k = n_k\circ(\id_{F^{\ot k}} \ot f)$ 
 holds for all $k \in \N$. This is obvious for $k = 0$. 
 For larger $k$ we proceed inductively, using the above 
 inductive representation of $m$ and $n$, namely 
 $f \circ m_{k+1} = f \circ m|_F \circ (\id_F \ot m_k) 
 	= n|_F \circ (\id_F \ot (f \circ m_k)) \stackrel{\text{(ind. hyp.)}}= 
 	n|_F \circ (\id_F \ot n_k) \circ (\id_{F^{\ot (k+1)}} \ot f) 
 	= n_{k+1} \circ (\id_{F^{\ot (k+1)}} \ot f)$. 

The uniqueness assertion of the universal property implies that the above construction gives the inverse to $\fcG$ on objects. 

 \medskip\noindent
 {\em $\fcG$ is monoidal:}
 If $(X,m)$, $(Y,n)$ are two $\talg(F)$-modules, then we have
 \begin{gather}
 	\raisebox{-0.5\height}{\includegraphics[scale=0.5]{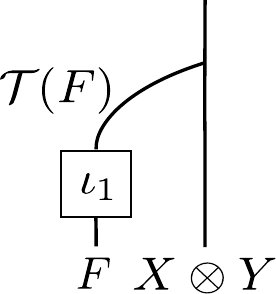}} 
 	= \ \raisebox{-0.5\height}{\includegraphics[scale=0.5]{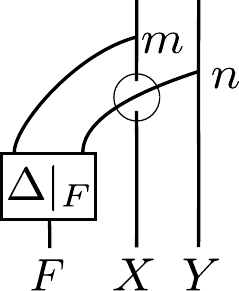}} 
 	= \ \raisebox{-0.5\height}{\includegraphics[scale=0.5]{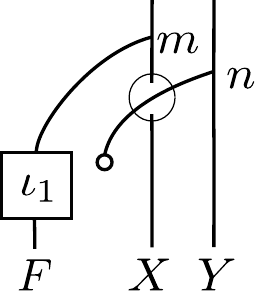}} 
 	+ \ \raisebox{-0.5\height}{\includegraphics[scale=0.5]{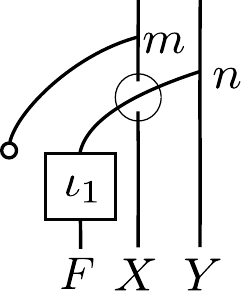}} 
 	= T(m|_F,n|_F) \ .
 \end{gather}
 In words, if one restricts the action via the coproduct of $\talg(F)$ on $X \ot Y$ to $F$, 
 one obtains the morphism $\tpmorph(-,-)$ as defined in \eqref{eq:Tpmorph}. This shows that $\fcG$ becomes monoidal with structure maps $\fcG_2 : \fcG(X,m) \otimes \fcG(Y,n) \to \fcG((X,m)\otimes (Y,n))$ and 
 $\fcG_0 : (\one,0) \to \fcG(\one,\pi_0 \ot \id_\one)$ 
 both given by the identity. 
\end{proof}

\section{Perturbed defects in two-dimensional field theory}
\label{sec:PertDef2D_FT}

In this section we will indicate how the algebraic constructions presented above and those to follow in Section \ref{sec:comm_cond_YD} are related to so-called non-local integrals of motion in two-dimensional conformal field theory and perturbations thereof. This is based on the original works \cite{BLZ96,BLZ97b,BLZ99} and their reinterpretation and generalisation in terms of perturbed defects in \cite{Ru07,MR09,Ru10}. As mentioned in the introduction, the application to CFT is the motivation to develop the present formalism.

As this section is quite long, a reader only interested in the mathematical definitions and theorems might as well skip ahead to Section \ref{sec:comm_cond_YD} without missing any of the mathematical content.

\medskip

A conformal field theory with defects assigns correlators (i.e.\ multilinear functions on products of vector spaces called ``spaces of states/fields'') to two-dimensional Riemann surfaces with parametrised boundaries and/or marked points. The two dimensional surfaces can be decorated by embedded oriented 1-manifolds -- the ``defect lines'' -- whose connected components are labelled by a set of ``defect conditions''.
There is a straightforward extension of the axiomatic description of CFT (as e.g.\ in \cite{Vafa87,SegalDef}) to CFT in the presence of defect lines, see \cite{RuSu08,DKR11} for details. Examples can be constructed in rational CFT
via its description in terms of three-dimensional topological field theory as developed in \cite{FRS02,FRS05,FFRS06}. 
The TFT construction assumes certain monodromy and factorisation properties of the conformal blocks which have been proved only in genus 0 and 1 (see \cite{Huang:2003cq}); our construction only uses the genus 0 part.
We will not review the general formalism but will just pick out the small part of it that we will need.

\subsection{Cylinders with defect circles and defect fields} \label{sec:cyl+def-field}

\begin{figure}[bt]

\begin{align*}
\begin{array}{llll}
\text{(a)}
& 
&\text{(c)}
& C_{\varepsilon_1,\varepsilon_2,\varepsilon_3}^{X_1,\ldots,X_m; Y_1,\ldots,Y_n} \\
& C_\varepsilon := && \qquad	(x_1,\ldots,x_m; y_1,\ldots,y_n) := \\
& \raisebox{-0.5\height}{\includegraphics[scale=.48]{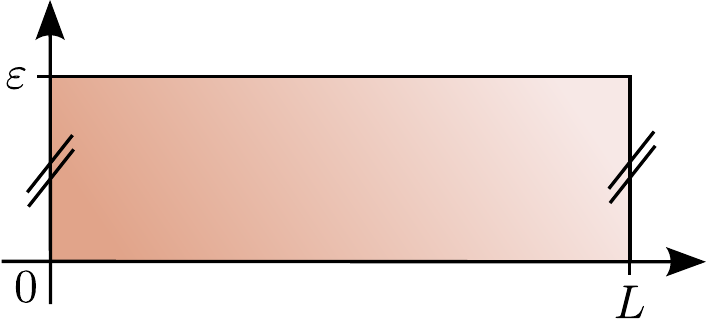}}
&& \raisebox{-0.5\height}{\includegraphics[scale=.48]{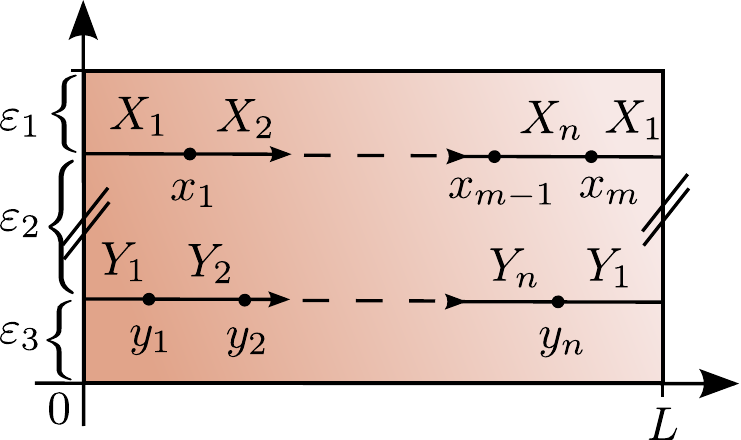}} \\ 
\text{(b)}
& C_{\varepsilon_1,\varepsilon_2}^{X_1,\ldots,X_n}(x_1,\ldots,x_n) := &&\\
& \raisebox{-0.5\height}{\includegraphics[scale=.48]{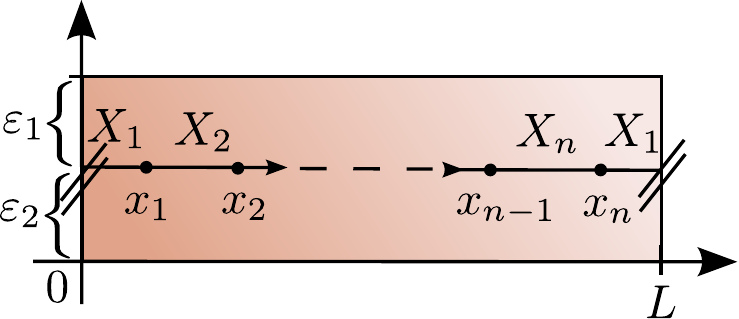}}
\end{array}
\end{align*}

\caption{Cylinders of circumference $L$ with none, one or two defect circles; the vertical sides are identified. The $X_i$ and $Y_i$ are defect conditions, the $x_i$ and $y_i$ are insertion points of defect changing fields.}
\label{fig:defect-cylinders}
\end{figure}

Three types of surfaces are going to appear below, namely cylinders of circumference $L$ with none, one or two defect circles inserted, see Figure \ref{fig:defect-cylinders}. To explain these diagrams and what the defect CFT assigns to them we need to introduce the space of bulk states, defect conditions and defect changing fields.

\medskip\noindent
{\em The space of bulk states}: a topological vector space $\mathcal{H}$ which we take to be a Hilbert space for concreteness. We require that $\mathcal{H}$ contains a dense subspace $\mathcal{F}$, the {\em space of fields}, which carries an action of $\mathrm{Vir} \oplus \mathrm{Vir}$, where $\mathrm{Vir}$ denotes the Virasoro algebra. We denote the generators of the first copy of $\mathrm{Vir}$ by $L_m,C$ and those of the second copy by $\overline L_m,\overline C$. We require that $\mathcal{F}$ is a direct sum of common eigenspaces $\mathcal{F}^{(h,\bar h)}$ of $L_0$ and $\overline L_0$ of non-negative eigenvalues,
\begin{equation}
  \mathcal{F} = \bigoplus_{h,\bar h\in\mathbb{R}_{\geq 0}} \mathcal{F}^{(h,\bar h)} \ .
\end{equation}
We assume that the eigenspaces $\mathcal{F}^{(\Delta)} = \bigoplus_{h+\bar h = \Delta} \mathcal{F}^{(h,\bar h)}$ of the generator $L_0+\overline L_0$ of scale transformations are finite-dimensional, and that $C,\overline C$ act by multiplication by constants $c, \bar c$ (called the left and right central charge). 

The requirements listed above are not the most general setting for conformal field theory. They exclude, for example Liouville and other non-compact CFTs, or logarithmic and other non-unitary CFTs. But they allow for all unitary rational CFTs and are sufficient to illustrate the origin of the formalism developed in the rest of this section. 

The {\em Hamiltionian} of the CFT on a cylinder of circumference $L$ is
\begin{equation}
  H_0 = \frac{2 \pi}{L}\Big( L_0 + \overline L_0 - \frac{c+\bar c}{24} \Big)
\end{equation}
and it acts on $\mathcal{F}$. By assumption, $e^{-\eps H_0}$ extends to a bounded linear operator on $\mathcal{H}$ for all $\eps>0$. This is the operator the CFT assignes to the cylinder $C_\eps$ in Figure \ref{fig:defect-cylinders}\,a):
\begin{equation}
  \mathcal{O}(C_\eps) = e^{-\eps H_0} \ .
\end{equation}

\medskip\noindent
{\em The set of defect conditions}: 
The defect circles in Figure \ref{fig:defect-cylinders}\,b),\,c) carry marked points, which we will call {\em field insertion points}. The intervals between two such insertion points are labelled by elements from a set of defect conditions.
We will only consider defect conditions that correspond to {\em topological defects}, which are defects that are transparent to the stress-energy tensor.
It turns out that a CFT with defects equips this set of defect conditions with the extra structure of a monoidal category (more precisely, one passes to lists of defect conditions, see \cite[Sect.\,2.4]{DKR11} for details).

In the CFT context, the monoidal category $\CatD$ in Section \ref{sec:tensor_algebras} is 
the category of defect conditions. For this reason, we will denote the set of defect conditions by
\begin{equation}
  \mathrm{Obj}(\CatD)
\end{equation}
(and assume that $\mathrm{Obj}(\CatD)$ is a set). The unit object $\one \in \CatD$ corresponds to the ``invisible defect'' in the sense that a defect segment labelled by $\one$ can be omitted from the surface. Thus, the cylinder in Figure \ref{fig:defect-cylinders}\,a) is a special case of that in Figure \ref{fig:defect-cylinders}\,b) when choosing $n=0$ and $\one$ as defect condition, and Figure \ref{fig:defect-cylinders}\,b) in turn is a special case of c). 

\medskip\noindent
{\em Defect fields and defect changing fields}: 
For each pair of defect labels $X,Y$ there is a space of {\em defect changing fields} $\mathcal{F}_{XY}$. Fields from $\mathcal{F}_{XY}$ can be assigned (see below) to a field insertion point between a segment labelled $X$ and a segment labelled $Y$; the order is determined by the orientation of the defect circle. The spaces of defect changing fields are subject to the same conditions as the space $\mathcal{F}$ of bulk fields above. 
That the $\mathcal{F}_{XY}$ are again $\mathrm{Vir}\oplus\mathrm{Vir}$-modules follows as all defects are topological.
Elements of $\mathcal{F}_{XX}$ are also called {\em defect fields}. Defect fields on the invisible defect $\one$ are nothing but bulk fields: $\mathcal{F} \equiv \mathcal{F}_{\one\one}$.

To the cylinder $C_{\eps_1,\eps_2}^{X_1,\dots,X_n}(x_1,\dots,x_n)$ in Figure \ref{fig:defect-cylinders}\,b) the CFT assigns a multilinear map to the bounded operators on $\mathcal{H}$,
\begin{equation} \label{eq:corr-cylinder-npoint}
  \mathcal{O}(C_{\eps_1,\eps_2}^{X_1,\dots,X_n}(x_1,\dots,x_n)) 
  : \mathcal{F}_{X_1 X_2} \times \mathcal{F}_{X_2 X_3} \times \cdots \times \mathcal{F}_{X_n X_1}
  \longrightarrow B(\mathcal{H}) \ ,
\end{equation}
where the insertion points are ordered as $0 \le x_1 < x_2 < \cdots < x_n < L$ and $X_k$ is the defect label for the segment $(x_{k-1},x_k)$,  $k=2,\dots,n$, and for the segment between $x_1$ and $x_n$ in the case $k=1$.

The assignment of operators to cylinders is compatible with composition. For example, $\mathcal{O}(C_\eps) \circ \mathcal{O}(C_{\eps_1,\eps_2}^{X_1,\dots,X_n}(x_1,\dots,x_n)) = \mathcal{O}(C_{\eps_1+\eps,\eps_2}^{X_1,\dots,X_n}(x_1,\dots,x_n))$, and composing the two operators which the CFT assigns to two cylinders with a single defect circle as in Figure \ref{fig:defect-cylinders}\,b) gives the operator for a cylinder with two defect circles as in Figure \ref{fig:defect-cylinders}\,c).

In the case $n=0$, i.e.\ in the absence of insertion points on the defect circle, we can try to define an operator on $\mathcal{H}$ as the limit
\begin{equation}
  \mathcal{O}(X) = \lim_{\eps_1,\eps_2 \to 0}  \mathcal{O}(C_{\eps_1,\eps_2}^{X}) \ .
\end{equation}
This limit may not exist for general defect conditions, but for topological defects the operator $\mathcal{O}(X)$ is defined at least on $\mathcal{F}$, and in fact maps $\mathcal{F}$ to $\mathcal{F}$, and it intertwines the Virasoro actions:
\begin{equation}
  [L_m , \mathcal{O}(X)] = 0 = [\overline L_m , \mathcal{O}(X)] 
  \qquad
  \text{for all} ~~ m \in \mathbb{Z}  ~,~ X \in \CatD \ . 
\end{equation}

\begin{rem} \label{rem:top-def-in-RCFT}
Let us outline how the category $\CatD$ of defect conditions can be realised in rational CFT. Fix a rational vertex operator algebra $\mathcal{V}$; by ``rational'' we mean that the category of modules $\Cat := \mathrm{Rep}(\mathcal{V})$ is a modular category, see \cite{Hu08}. In particular, $\Cat$ is braided monoidal, $\mathbb{C}$-linear and semisimple. In the approach to correlators of rational CFT via three-dimensional topological field theory in \cite{FRS02,FRS05,FFRS06,FFS12}, the different CFTs with symmetry $\mathcal{V} \otimes_{\mathbb{C}} \mathcal{V}$ are described by algebras $A \in \Cat$, more precisely by special symmetric Frobenius algebras. The defects in the CFT described by $A$ which are transparent to all fields in $\mathcal{V} \otimes_{\mathbb{C}} \mathcal{V}$  (not just to the stress-energy tensor) are labelled by $A|A$-bimodules. In fact, in this situation one has
\begin{equation} \label{eq:D=A-mod-A-RCFT}
  \CatD \cong A\text{-Mod$_\Cat$-}A 
\end{equation}
as $\mathbb{C}$-linear monoidal categories.
\end{rem}

\subsection{Perturbed defects and the category $\CatD_F$}
\label{ssec:pert-def-and-cat}

In quantum field theory, a standard method to modify a given model is ``to add new terms to the action''. If expressed via correlators, this amounts to inserting a term $e^{-\Delta S}$, where $\Delta S$ is some expression in terms of the fields and the exponential is understood as a formal sum whose precise definition depends on the situation. Here we are interested in the case that $-\Delta S = \lambda \int_0^L \psi(x) dx$, where $\psi$ is a defect field on some defect $X$ on a cylinder of circumference $L$. Let is explain this in more detail.

\medskip

Pick a defect label $X \in \CatD$ and a defect field $\psi \in \mathcal{F}_{XX}$. In expression \eqref{eq:corr-cylinder-npoint}, choose all $X_k$ to be $X$ and use $\psi$ in each of the $n$ arguments to define
\begin{equation} \label{eq:f-def-integrands}
f_{\eps_1,\eps_2;n}^X(\psi;x_1,\dots,x_n)
= \mathcal{O}(C_{\eps_1,\eps_2}^{X,\dots,X}(x_1,\dots,x_n))(\psi,\psi,\dots,\psi) ~ \in B(\mathcal{H}) \ .
\end{equation}
Next we integrate over the insertion points $x_1,\dots,x_n$ while preserving the ordering,
\begin{equation} \label{eq:e-def-integrals}
e_{\eps_1,\eps_2;n}^X(\psi)
= \int_0^L dx_1 \int_{x_1}^L dx_2 \cdots \int_{x_{n-1}}^L dx_n ~
f_{\eps_1,\eps_2;n}^X(\psi;x_1,\dots,x_n) \ .
\end{equation}
This integral may or may not exist, depending on the choice of $X$ and $\psi$. We assume that it exists, so that $e_{\eps_1,\eps_2;n}^X$ is again an operator on $\mathcal{H}$. 

We can now state what we mean by ``the defect $X$ perturbed by the defect field $\psi$'', or, equivalently, what we mean by ``inserting $\exp(\lambda \int \psi(x) dx)$ on a defect circle $X$''. Namely, take $\lambda$ to be a formal variable and consider the formal sum (a path ordered exponential)
\begin{equation} \label{eq:D-eps1eps2-def}
D_{\eps_1,\eps_2}^X(\psi;\lambda)
= \sum_{n=0}^\infty \lambda^n e_{\eps_1,\eps_2;n}^X(\psi) \ .
\end{equation}
The defect operator for the topological defect $X$ perturbed by the defect field $\psi$ is defined to be the limit of zero height of the surrounding cylinder:
\begin{equation} \label{eq:O(X,psi,lam)}
\mathcal{O}(X,\psi;\lambda) = \lim_{\eps_1,\eps_2 \to 0} D_{\eps_1,\eps_2}^X(\psi;\lambda) \ .
\end{equation}
Again, existence of the limit (as an operator on $\mathcal{H}$ with dense domain) will in general depend on the choice of $X$ and $\psi$.

The operators $\mathcal{O}(X,\psi;\lambda)$ are the main object of our interest. For which $X,\psi$ the $\mathcal{O}(X,\psi;\lambda)$ exist and what the radius of convergence in $\lambda$ is, are difficult questions. In the free boson example, they have been addressed in \cite[App.\,3]{BLZ99} and we will further investigate them in a forthcoming paper \cite{InProgr}. A general expectation would be that the integrals \eqref{eq:e-def-integrals} and the limit \eqref{eq:O(X,psi,lam)} exist if $\psi \in \mathcal{F}_{XX}^{(\Delta)}$ with $\Delta < \tfrac12$. The reason is that the leading singularity of the integrand \eqref{eq:f-def-integrands} in the limit of two coinciding points will generically be $(x_{i+1}-x_i)^{-2\Delta}$. However, one would also expect that some regulator exists to define the integrals for $\tfrac12 \le \Delta \le 1$. Defect fields of weight $\Delta<1$, $\Delta=1$, $\Delta>1$ are called relevant, marginal, and irrelevant, respectively. We do not know if there is a generic behaviour to expect regarding convergence of the power series in $\lambda$.

\medskip

\begin{table}
\begin{center}
\renewcommand{\arraystretch}{1.4}
\begin{tabular}{ccc|cc}
$\Cat$ & $\Cat \boxtimes \Cat^\mathrm{rev}$ & $\DZ(\Cat)$ & $\CatD$ & $\DZ(\CatD)$ 
\\
\hline
$F_\Cat$ & $F_{\Cat^2}$ & $F_{\DZ\Cat}$ & $\tilde F$ & $F$
\end{tabular}
\renewcommand{\arraystretch}{1}
\end{center}
\caption{Notation for related objects in the five categories that appear in the application to CFT. Starting from the right, $F$ is the object in the monoidal centre of the defect category $\CatD$. This is the datum used to define $\CatD_F$, see Definition \ref{df:catdf}. As an element of the monoidal centre, $F$ is a pair $(\tilde F, \varphi_{F,-})$ with $\tilde F \in \CatD$. Turning to the left, $\Cat$ stands for the (braided) category of representations of the underlying vertex operator algebra $\mathcal{V}$ and $F_\Cat$ is one such representation. Combined representations of the holomorphic and anti-holomorphic copy of the vertex operator algebra, i.e.\ of $\mathcal{V} \otimes_{\C} \mathcal{V}$, correspond to objects $F_{\Cat^2} \in \Cat \boxtimes \Cat^\mathrm{rev}$. 
The functors in \eqref{eq:Functors=i+,i-} allow one to map $\Cat \boxtimes \Cat^\mathrm{rev}$ to $\DZ(\Cat)$ 
 and so we can also think of objects $F_{\DZ\Cat}\in\DZ(\Cat)$ as $\mathcal{V} \otimes_{\C} \mathcal{V}$-modules. 
 We will write $F_{\DZ\Cat} = (F_\Cat, \varphi_{F_{\DZ\Cat},-})$, where $F_\Cat$ is the underlying object in $\Cat$. Finally, the functor $\alpha$ defined in \eqref{eq:alpha-def} maps $\DZ(\Cat)$ to $\DZ(\CatD)$, and we write $F$ for the image of $F_{\DZ\Cat}$ under that functor.}
\label{tab:F-notation}
\end{table}

To make the connection to the category $\CatD_F$, we will restrict ourselves to rational CFT in the setting described in Remark \ref{rem:top-def-in-RCFT}. 
Unfortunately, in the application to CFT we have to deal with five categories, each of which contain objects related to $F$. This is described below and it is summarised once more in Table \ref{tab:F-notation}.
We pick $F_{\Cat^2} \in  \mathrm{Rep}(\mathcal{V} \otimes_{\mathbb{C}} \mathcal{V}) = \Cat \boxtimes \Cat^\mathrm{rev}$, where $\Cat^\mathrm{rev}$ is $\Cat$ with inverse braiding. Via \eqref{eq:Functors=i+,i-} we have a functor $\Cat \boxtimes \Cat^\mathrm{rev} \to \DZ(\Cat)$; if $\Cat$ is modular, this functor is an equivalence \cite[Thm.\,7.10]{Mu03b}. We write $F_{\DZ\Cat}$ for the image of $F_{\Cat^2}$ in $\DZ(\Cat)$. 

To $F_{\DZ\Cat} = (F_\Cat, \varphi_{F_{\DZ\Cat},-}) \in \DZ(\Cat)$ we assign the $A|A$-bimodule $F$ with underlying object $A \ot F_\Cat$, with left action by multiplication and with right action by first using the half-braiding to move $A$ past $F_\Cat$ and then multiplying. In formulas,
 \begin{align}  
 	\aind(F_{\DZ\Cat}) = \big( \, A \ot F_\Cat \,,\, \mu \ot \id_{F_\Cat} \,,\, (\mu \ot \id_{F_\Cat}) \circ (\id_A \ot \varphi_{F_{\DZ\Cat},A})\,\big) \ . 
	\label{eq:alpha-def}
 \end{align}
 For $(X,m_l,m_r) \in A\BMod[\Cat] A$ let 
 $\varphi_{\aind(F_{\DZ\Cat}),X}: \aind(F_{\DZ\Cat}) \ot_A X \to X \ot_A \aind(F_{\DZ\Cat})$ 
 be the morphism in $A\BMod[\Cat] A$ induced by the morphism $A \ot F_\Cat \ot X \to X \ot A \ot F_\Cat$, 
 \begin{gather}
 	(\id_X \ot \eta \ot \id_{F_\Cat}) \circ (m_l \ot \id_{F_\Cat}) \circ (\id_A \ot \varphi_{F_{\DZ\Cat},X}) 
 		= \ \raisebox{-0.5\height}{\includegraphics[scale=0.5]{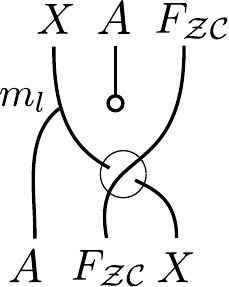}} \ .
 \end{gather}

One verifies that $\varphi_{\aind(F_{\DZ\Cat}),X}$ is an isomorphism in 
$A\BMod[\Cat] A$. In fact, we have:

\begin{prop} (cf.\ \cite[Prop.\,3.2]{Sch01}) \label{prop:alphaind}
	The assignment 
	\begin{align} 
		\aind: \; \DZ(\Cat) \to \DZ(A\BMod[\Cat] A) \quad , \qquad 
			F_{\DZ\Cat} \mapsto (\aind(F_{\DZ\Cat}), \varphi_{\aind(F_{\DZ\Cat}),-}) \ . 
	\end{align}
	defines a braided monoidal functor.
\end{prop}

\begin{rem}\label{rem:alpha-frs-relation}
	Denote the functors $\aind \circ \iota^\pm: \Cat \to \DZ(A\BMod[\Cat] A)$ by $\aind^\pm$.
	There are isomorphisms $\aind^\pm(F_\Cat) \ot_A X \cong F_\Cat \ot^\pm X$, 
	where $F_\Cat \ot^\pm X$ are as in \cite[Eqn. (2.17), (2.18)]{FRS05}. 
	The notation $\ot^\pm$ is used 
	extensively in \cite{FRS02} and \cite{FRS05} and this remark 
	may be useful to compare results. 
\end{rem}

Combining the equivalence $\Cat \boxtimes \Cat^\mathrm{rev} \to \DZ(\Cat)$ with $\alpha$ from Proposition \ref{prop:alphaind} gives a functor $\Cat \boxtimes \Cat^\mathrm{rev} \to \DZ(\CatD)$ which by abuse of notation we also call $\alpha$.
The space of $\mathcal{V} \otimes_{\mathbb{C}} \mathcal{V}$-intertwiners from $F_{\Cat^2}$ to the space of defect fields $\mathcal{F}_{XX}$ is given by (see \cite{FRS05,FFRS06} for details)
\begin{equation} \label{eq:Hom_VV-iso-Hom_AA}
  \Hom_{\mathcal{V} \otimes_{\mathbb{C}} \mathcal{V}}(F_{\Cat^2},\mathcal{F}_{XX})
  \cong \Hom_{A|A}(\alpha(F_{\Cat^2})\otimes_A X,X) \ ,
\end{equation}
which defines $\mathcal{F}_{XX}$ as the object representing the functor (in $F_{\Cat^2}$) on the right hand side.
Thus, given $F_{\Cat^2}$ and $\psi \in F_{\Cat^2}$, an object
\begin{equation} \label{eq:rel-DF-aux1}
  (X,m) \in \CatD_F
  \qquad ,~~\text{where}
  ~~ F = \alpha(F_{\Cat^2}) \in \DZ(\CatD) \ ,
\end{equation}
describes a defect field $\psi^m$ on the defect $X$, obtained by mapping $\psi \in F_{\Cat^2}$ to $\mathcal{F}_{XX}$ via the preimage of $m$ under the isomorphism \eqref{eq:Hom_VV-iso-Hom_AA}. We will abbreviate the operator in \eqref{eq:O(X,psi,lam)} as
\begin{equation} \label{eq:ODlam-abbrev}
  \mathcal{O}(D;\lambda) \equiv \mathcal{O}(X,\psi^m;\lambda) 
  \qquad ,~~\text{where}
  ~~ D = (X,m) \in \CatD \ ,
\end{equation}
and it is understood that $F_{\Cat^2}$ and $\psi \in F_{\Cat^2}$ are chosen once and for all.

The morphisms in $\CatD_F$ equally have a direct interpretation in terms of defect fields. To prepare this, consider the subspace of scale- and translation-invariant states $\mathcal{F}_{XY}^\mathrm{vac} \subset \mathcal{F}_{XY}$ ('vac' for vacuum), i.e.\footnote{
  Actually, the condition $L_{-1} \psi =0$ implies $L_m \psi =0$ for all $m \ge 0$ (and analogously for $\overline L_{-1}$) as in this case the vertex operator $Y(T,z)$, where $T$ is the Virasoro element, generates no negative powers of $z$ when acting on $\psi$.}
\begin{equation}
  \mathcal{F}_{XY}^\mathrm{vac} = \{ \psi \in  \mathcal{F}_{XY} | L_m \psi =0= \overline L_m \psi \text{ for } m=0,-1 \} \ .
\end{equation}
By construction of the morphism spaces of the category $\CatD$ (see \cite[Sect.\,2.4]{DKR11} for details), we have $\mathcal{F}_{XY}^\mathrm{vac} = \Hom_\CatD(X,Y)$. The composition in $\CatD$ then gives the operator product expansion of the corresponding fields (translation invariance implies the absence of short-distance singularities). Let now $f: (X,m) \to (Y,n)$ be a morphism in $\CatD_F$. Then we have the identity of operators
\begin{align}\label{eq:f-exchange-condition}
&\mathcal{O}\left(\raisebox{-.5\height}{\includegraphics[scale=.5]{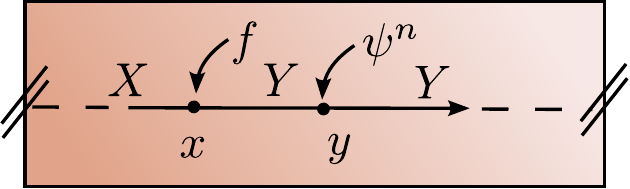}}\right)
	(\dots,f,\psi^n,\dots) \nonumber\\
&\qquad =
\mathcal{O}\left(\raisebox{-.5\height}{\includegraphics[scale=.5]{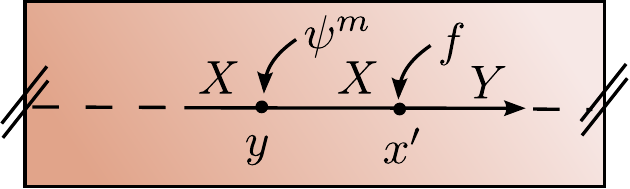}}\right)(\dots,\psi^m,f,\dots) \ .
\end{align}
Here, the pictures represent a cylinder as in \eqref{eq:corr-cylinder-npoint}, and the dotted part of the defect circle stands for any arrangement of field insertions common to both sides. In words, the defect changing field $f$ converts the topological defect $X$ perturbed by $\psi^m$ into $Y$ perturbed by $\psi^n$. The identity \eqref{eq:f-exchange-condition} is best understood in the TFT formulation, see \cite[Sect.\,3]{MR09}. 

\begin{table}
\begin{center}
\begin{tabular}{p{17em}|p{19em}}
quantities in $\CatD$ and $\CatD_F$
&	conformal field theory concept
\\
\hline
\\[-.5em]
object $X \in \CatD$
&	defect condition $X$ for which the corresponding defect is topological and transparent to $\mathcal{V} \otimes_{\mathbb{C}} \mathcal{V}$
\\[.5em]
object $F \in \DZ(\CatD)$
&	
obtained from the representation $F_{\Cat^2}$ in which the perturbing field $\psi$ lives via $F = \alpha(F_{\Cat^2})$, see \eqref{eq:rel-DF-aux1}
\\[.5em]
morphism $m \in \Hom_\CatD(F \otimes X,X)$
&
	determines the defect field 
	$\psi^m \in \mathcal{F}_{XX}$ as the preimage of $\psi$ under \eqref{eq:Hom_VV-iso-Hom_AA}
\\[.5em]
object $(X,m) \in \CatD_F$
&	determines a topological defect $X$ together with a choice of perturbing field $\psi^m$ on $X$ as in \eqref{eq:ODlam-abbrev}
\\[.5em]
morphism $f: (X,m) \to (Y,n)$ in $\CatD_F$ 	
&	a translation invariant field $f \in \mathcal{F}_{XX}^\mathrm{vac}$ which satisfies the
	exchange condition \eqref{eq:f-exchange-condition}
\end{tabular}
\end{center}
\caption{Relation between quantities in $\CatD$ and $\CatD_F$ and conformal field theory concepts. See also Table \ref{tab:F-notation}.}
\label{tab:DF-CFT-rel}
\end{table}

The relation between the quantities described above and conformal field theory concepts is summarised in the Table \ref{tab:DF-CFT-rel}.
While the above discussion is for rational $\mathcal{V}$, we believe that this construction applies to a wider class of vertex operator algebras. In a forthcoming paper \cite{InProgr} we will investigate the case of the free boson, i.e.\ of the Heisenberg vertex operator algebra, more closely.

\subsection{Composition of defect operators and the tensor product in $\CatD_F$}
\label{ssec:tensor-prod-of-pert-def}

As already mentioned above, the composition of two operators associated to a cylinder as in Figure \ref{fig:defect-cylinders}\,b) results in an operator associated to a cylinder as in Figure \ref{fig:defect-cylinders}\,c). As a formula, this reads
\begin{align}\label{eq:compos-def-aux1}
&  
\mathcal{O}(C_{\eps_1,\eps_2}^{X,\dots,X}(\vec x))(\phi_1,\dots,\phi_m) 
\circ
\mathcal{O}(C_{\eps_3,\eps_4}^{Y,\dots,Y}(\vec y))(\psi_1,\dots,\psi_n) 
\nonumber  \\
&  =
  \mathcal{O}(C_{\eps_1,\eps_2+\eps_3,\eps_4}^{X,\dots,X;Y,\dots,Y}(\vec x; \vec y))(\phi_1,\dots,\psi_n)
  \ , 
\end{align}
where $\vec x = (x_1,\dots,x_m)$, $\vec y = (y_1,\dots,y_n)$, and $\phi_i \in \mathcal{F}_{XX}$, $\psi_i \in \mathcal{F}_{YY}$, and where we chose fixed defect conditions $X,Y \in \CatD$ on each defect circle.

If $x_i \neq y_j$ for all $i,j$, we can take the limit $\eps_2+\eps_3 \to 0$ on the right hand side of \eqref{eq:compos-def-aux1}. The topological defects $X$ and $Y$ ``fuse'' into a new topological defect, which is $X \otimes Y$ by construction of the tensor product in the category $\CatD$ of topological defect conditions \cite[Sect.\,2.4]{DKR11}. 
Let $(X,m),(Y,n) \in \CatD_F$. In the above fusion procedure, a field $\psi^m$ inserted on the defect $X$ at a point $x$ turns into the field $\psi^{T(m,0)}$ inserted at point $x$ on $X \otimes Y$, with $T(-,-)$ as given in \eqref{eq:Tpmorph}. Analogously, a field $\psi^n$ inserted on the defect $Y$ at a point $y$ turns into the field $\psi^{T(0,n)}$ inserted at point $y$ on $X \otimes Y$. These rules are again best understood in the TFT formalism, and we refer to \cite[Sect.\,3]{MR09}. 

Using this, it is not hard to see that, if the multiple integrals implicit in the expression below all exist, the operators defined in \eqref{eq:D-eps1eps2-def} satisfy 
\begin{equation}
\lim_{\eps\to0} D_{\eps_1,\eps}^X(\psi^m;\lambda) \circ D_{\eps,\eps_2}^Y(\psi^n;\lambda)
= D^{X \otimes Y}_{\eps_1,\eps_2}(\psi^{T(m,n)};\lambda) \ .
\end{equation}
The argument for this is the same as in \cite[Thm.\,3.2]{MR09}. 
Thus, assuming that also the operators \eqref{eq:ODlam-abbrev} are well-defined for this choice of $(X,m)$ and $(Y,n)$,
\begin{equation}
  \mathcal{O}((X,m);\lambda) \circ   \mathcal{O}((Y,n);\lambda)
  =   \mathcal{O}((X \otimes Y, T(m,n));\lambda) \ .
\end{equation}
This illustrates that the tensor product in $\CatD_F$ is designed to be compatible with the composition of defect operators.

\subsection{Operators for perturbed defects and $K_0(\CatD_F)$}
\label{ssec:op+K0}

Suppose we are given an exact sequence 
\begin{equation}
0 \to (A,a) \xrightarrow{~f~} (B,b) \xrightarrow{~g~} (C,c) \to 0
\end{equation}
in $\CatD_F$. We claim that if the underlying sequence in $\CatD$ splits, the perturbed defect operators satisfy
\begin{equation} \label{eq:operator-sequence-sum}
  \mathcal{O}((B,b);\lambda) ~=~ 
  \mathcal{O}((A,a);\lambda) ~+~ 
  \mathcal{O}((C,c);\lambda)  \ ,
\end{equation}
where we assume the existence of the individual operators. In more detail, we ask that the underlying sequence $0 \to A \xrightarrow{~f~} B \xrightarrow{~g~} C \to 0$ in $\CatD$ splits, i.e.\ that we can find $\tilde f : B \to A$ and $\tilde g : C \to B$ such that $\tilde f \circ f = \id_A$, $g \circ \tilde g = \id_C$, and $f \circ \tilde f + \tilde g \circ g = \id_B$. Of course, this sequence splits automatically if $\CatD$ is semisimple, as is the case in rational CFT since then the category of bimodules in \eqref{eq:D=A-mod-A-RCFT} is semisimple. 

The reasoning behind \eqref{eq:operator-sequence-sum} is actually very simple (see \cite[Thm.\,3.2]{MR09}). We will verify it on the level of the integrands \eqref{eq:f-def-integrands}:
\begin{align}
f_{\eps_1,\eps_2;n}^B(\psi^b;x_1,\dots,x_n)
& \overset{(1)}= ~
\mathcal{O}(C_{\eps_1,\eps_2}^{B,A,B\dots,B}(z,z',x_1,\dots,x_n))(\tilde f, f, \psi^b,\dots,\psi^b)
\nonumber\\
&\hspace{5em}
+ \mathcal{O}(C_{\eps_1,\eps_2}^{B,C,B\dots,B}(z,z',x_1,\dots,x_n))(g,\tilde g, \psi^b,\dots,\psi^b)
\nonumber\\
& \overset{(2)}= ~
\mathcal{O}(C_{\eps_1,\eps_2}^{A,B,A\dots,A}(z,z',x_1,\dots,x_n))(f,\tilde f, \psi^a,\dots,\psi^a)
\nonumber\\
&\hspace{5em}
+ \mathcal{O}(C_{\eps_1,\eps_2}^{C,B,C\dots,C}(z,z',x_1,\dots,x_n))(\tilde g,g, \psi^c,\dots,\psi^c)
\nonumber\\
& \overset{(3)}= ~
f_{\eps_1,\eps_2;n}^A(\psi^a;x_1,\dots,x_n) + f_{\eps_1,\eps_2;n}^C(\psi^c;x_1,\dots,x_n) \ .
\end{align}
In step (1), $f \circ \tilde f + \tilde g \circ g = \id_B$ was used to rewrite the right hand side of \eqref{eq:f-def-integrands} as a sum of two terms. The key step is (2), where \eqref{eq:f-exchange-condition} is used to drag $f$ and $g$ around the defect circle (in opposite directions) past all the defect fields $\psi^b$. In step (3), $\tilde f \circ f = \id_A$, $g \circ \tilde g = \id_C$ is used to remove the additional insertions at $z$ and $z'$.

This shows that, at least on the semisimple part of $\CatD$ and provided the operators $\mathcal{O}(-;\lambda)$ are defined, the assignment $(X,m) \mapsto \mathcal{O}((X,m);\lambda)$ factors through $K_0(\CatD_F)$. 

\subsection{Perturbed defects as non-local conserved charges}
\label{ssec:PertDefNonlocalConsCharges}

We will now turn to the question whether the operators $\mathcal{O}((X,m);\lambda)$ can be interpreted as ``non-local conserved charges'', and if so, which Hamiltonian they are conserved charges for. The qualifier ``non-local'', incidentally, refers to the fact that defect fields are in general not local with respect to each other. 

Let us start with the perturbed Hamiltonian. Given a bulk field $\Phi \in \mathcal{F}$ and a constant $\mu$ we define
\begin{equation}
  H_\mathrm{pert}(\Phi; \mu) = H_0 + \mu \cdot H_\mathrm{int} \ ,
\end{equation}
where the new interaction is $\int_0^L \Phi(x) dx$. More specifically, in the notation of Section \ref{sec:cyl+def-field} we set
\begin{equation} \label{eq:HintDef}
  H_\mathrm{int} = \lim_{\eps_1,\eps_2 \to 0} 
  \int_0^L \mathcal{O}(C_{\eps_1,\eps_2}^{\one}(x))(\Phi) \, dx \ .
\end{equation}
Next, pick fields $\psi^+,\psi^- \in \mathcal{F}_{XX}$ such that $\psi^+$ is holomorphic in the sense that $\overline L_{-1} \psi^+ = 0$, and $\psi^-$ is anti-holomorphic, i.e.\ $L_{-1} \psi^- = 0$. Then if $\psi^+$ and $\psi^-$ are inserted at two neighbouring points $x,y$ on a defect circle, the limit $x \to y$ is regular. 

Suppose that the condition
\begin{align} \label{eq:comm-rel-defect-ops}
  &\lim_{\eps\to 0} \left( \Op\left(\raisebox{-.5\height}{\includegraphics[scale=.45]{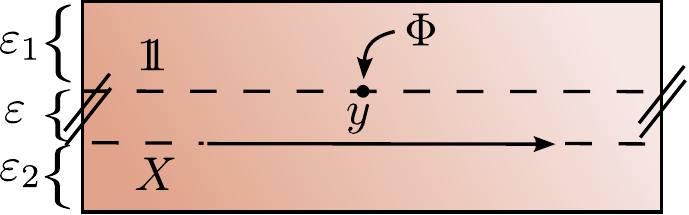}}\right) - \Op\left(\raisebox{-.5\height}{\includegraphics[scale=.45]{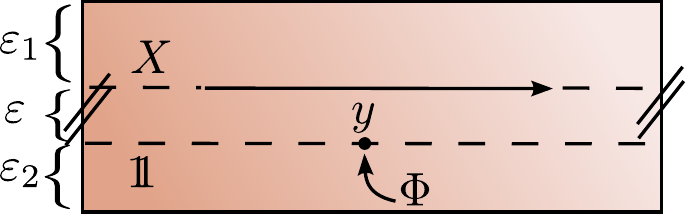}}\right) \right) \nonumber\\
  & = \lim_{\eps\to 0} \left( \Op\left( \raisebox{-.5\height}{\includegraphics[scale=.45]{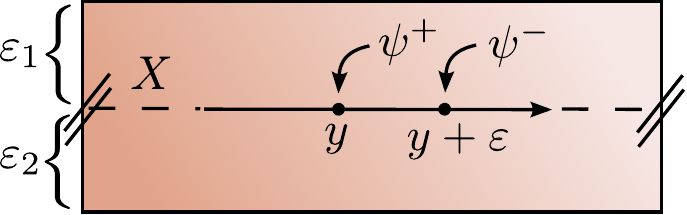}} \right) - \Op\left( \raisebox{-.5\height}{\includegraphics[scale=.45]{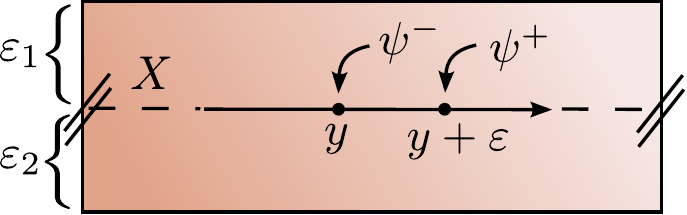}} \right) \right)
\end{align}
holds, where the omissions stand for any configuration of defect field insertions on $X$, common to all four cylinders, and we have not spelled out the arguments of $\mathcal{O}(\cdots)$. In formulas, the above identity reads
\begin{align}
\lim_{\eps\to 0} \Big(&
\mathcal{O}(C^{\one,X}_{\eps_1,\eps,\eps_2}(y;x_1,\dots,x_n))(\Phi;\psi_1,\dots,\psi_n)
\nonumber\\
&\qquad -  \mathcal{O}(C^{X,\one}_{\eps_1,\eps,\eps_2}(x_1,\dots,x_n;y))(\psi_1,\dots,\psi_n;\Phi) 
\Big)
\nonumber\\
=~\lim_{\eps\to 0} \Big(&
\mathcal{O}(C^{X}_{\eps_1,\eps_2}(x_1,\dots,y,y+\eps,\dots,x_n))(\psi_1,\dots,\psi^+,\psi^-,\dots,\psi_n)
\nonumber\\
&\qquad  - \mathcal{O}(C^{X}_{\eps_1,\eps_2}(x_1,\dots,y,y+\eps,\dots,x_n))(\psi_1,\dots,\psi^-,\psi^+,\dots,\psi_n) \Big)\ .
\end{align}
Here, $\psi_i \in \mathcal{F}_{XX}$ and $x_k<y<y+\eps<x_{k+1}$ for some $k$. We claim that this condition implies
\begin{equation} \label{eq:commutator-Hpert-O}
  \big[ \, H_\mathrm{pert}(\Phi; -2i\lambda^2) \, , \,  \mathcal{O}(X,\psi^+ + \psi^-;\lambda) \, \big] ~=~ 0 \ .
\end{equation}
In other words, $\mathcal{O}(X,\psi^++\psi^-;\lambda)$ is a conserved charge for a perturbed Hamiltonian. Note that this includes the special cases $\psi^+=0$ or $\psi^-=0$, where \eqref{eq:comm-rel-defect-ops} holds for $\Phi=0$. Thus $\mathcal{O}(X,\psi^+;\lambda)$ and $\mathcal{O}(X,\psi^-;\lambda)$ are conserved charges for the unperturbed Hamiltonian $H_0$. The argument leading to \eqref{eq:commutator-Hpert-O} is the same as given in \cite[Sect.\,3.4]{Ru10}, but as \cite{Ru10} is formulated for Virasoro minimal models and assumes that the perturbing fields have negative weight, in Appendix \ref{app:comm-der} we give the argument in the present setting without these restrictions.

\medskip

Finally, we would like a formulation of the commutation condition \eqref{eq:comm-rel-defect-ops} in terms of the category $\CatD_F$. To this end, we once more restrict our attention to rational CFT and the setting in Remark \ref{rem:top-def-in-RCFT}. 

Pick two representations $F^+_\Cat, F^-_\Cat \in \Cat \equiv \mathrm{Rep}(\mathcal{V})$, 
and pick vectors 
$\hat\psi_\pm \in F^\pm_\Cat$. Since the space of bulk fields is $\mathcal{F} \equiv \mathcal{F}_{\one\one}$, the isomorphism \eqref{eq:Hom_VV-iso-Hom_AA} becomes
\begin{equation} \label{eq:Hom_VV-iso-Hom_AA_bulk}
  \Hom_{\mathcal{V} \otimes_{\mathbb{C}} \mathcal{V}}(F^+_\Cat \otimes_{\mathbb{C}} F^-_\Cat,\mathcal{F})
  \cong \Hom_{A|A}(\alpha^+(F^+_\Cat) \otimes_A \alpha^-(F^-_\Cat) , A) \ ,
\end{equation}
where $\alpha^\pm = \alpha \circ \iota^\pm$ was introduced in Remark \ref{rem:alpha-frs-relation}. Writing $F^\pm \equiv \alpha^\pm(F^\pm_\Cat)$, we see that a morphism $b : F^+ \otimes F^- \to \one$ in $\CatD$ 
(rather than in $\DZ(\CatD)$ -- the forgetful functor is implicit)
defines a bulk field
\begin{equation}\label{eq:perturbing-bulk-field}
\Phi = (\hat\psi^+\otimes_{\mathbb{C}}\hat\psi^-)^b ~ \in ~ \mathcal{F}
\end{equation}
via the prescription after \eqref{eq:rel-DF-aux1}. For the defect perturbation we take $F = F^+ \oplus F^-$ and pick an object $(X,m) \in \CatD_F$. Let $\iota_{F^\pm} : F^\pm \to F$ be the canonical embeddings and set $m_{F^{\pm}} = m \circ (\iota_{F^\pm} \otimes \id_X)$. Then the triple
$\psi^+ := (\hat\psi^+)^{m_{F^+}}$, $\psi^- := (\hat\psi^-)^{m_{F^-}}$ and $\Phi$ satisfies condition \eqref{eq:comm-rel-defect-ops} if and only if
\begin{equation} \label{eq:commutation-relation-pics}
  \raisebox{-0.5\height}{\includegraphics[height=.11\textwidth]{img/eq_cc_1.pdf}} 
  \quad - 
  \raisebox{-0.5\height}{\includegraphics[height=.11\textwidth]{img/eq_cc_2.pdf}} 
  \quad = \ 
  \raisebox{-0.5\height}{\includegraphics[height=.11\textwidth]{img/eq_cc_3.pdf}} 
  \quad - \raisebox{-0.5\height}{\includegraphics[height=.11\textwidth]{img/eq_cc_4.pdf}} \quad .
\end{equation}
To find this condition one needs to employ the TFT formulation and adapt the argument given in \cite[Sect.\,4.3]{Ru10}; we omit the details.

\bigskip

This completes our overview of the application of the category $\CatD_F$ to the analysis of conserved charges of a perturbed conformal field theory. As already mentioned, this application is the main motivation for introducing $\CatD_F$. Clearly, the identity \eqref{eq:comm-rel-defect-ops} -- or its reformulation in \eqref{eq:commutation-relation-pics} -- is crucial for this application, as it guarantees that the operator $\mathcal{O}((X,m);\lambda)$ commutes with the perturbed Hamiltonian. In the next section we take a closer look at this condition.

\section{The commutation condition and its relation to Yetter-Drinfeld modules}
\label{sec:comm_cond_YD}

As in Section \ref{sec:tensor_algebras},
let $(\CatD,\ot,\alpha,\one,\lambda,\rho)$ be an abelian monoidal category. 
The following definition is motivated by \eqref{eq:commutation-relation-pics}. 

\begin{df} \label{df:commcond}
	Let $\Fc,\Fa \in \DZ(\CatD)$ and let
	$\bF: \Fc \ot \Fa \to \one$ be a morphism in $\CatD$. 
	An object $(X,m) \in \CatD_{\Fc \oplus \Fa}$ is said to \emph{satisfy the commutation condition  with $\bF$}
	if 
	\begin{align} \label{eq:compatible}
		&m_\Fc \circ (\id_{\Fc} \ot m_\Fa) \circ \alpha_{\Fc,\Fa,X} \nonumber\\ 
		& \hspace{2em} -~ m_\Fa \circ (\id_{\Fa} \ot m_\Fc) \circ 
				\alpha_{\Fc,\Fa,X} \circ (\varphi_{\Fc,\Fa} \ot \id_X) \nonumber\\
	 	&=~ \bF \ot \id_X ~-~ (\id_X \ot \bF) \circ \alpha_{X,\Fc,\Fa} 
	 			\circ \varphi_{\Fc \ot \Fa,X} \ ,
	\end{align}
where $m_\Fca := m \circ (\iota_\Fca \ot \id_X)$ and $\iota_\Fca : \Fca \to \Fc \oplus \Fa$ are the canonical embeddings.
\end{df}

The pictorial representation of condition  \eqref{eq:compatible} was already given in \eqref{eq:commutation-relation-pics}. The right hand side of \eqref{eq:compatible} is in general nonzero since $\bF$ is not required to be a morphism in $\DZ(\CatD)$. 

\medskip

We say that two objects $(V,\varphi_{V,-})$ and $(W,\varphi_{W,-})$ in $\DZ(\CatD)$ are {\em mutually transparent} if their half-braidings satisfy $\varphi_{W,V} \circ \varphi_{V,W} = \id_{V \otimes W}$.
If $F^\pm$ are mutually transparent, the commutation condition is compatible with the tensor product on $\CatD_F$:

\begin{prop} \label{prop:Delta}
	Suppose that $\Fc,\Fa \in \DZ(\CatD)$ are 
	mutually transparent. 
	If $(X,m)$, $(Y,n)$ $\in$ $\CatD_{\Fc \oplus \Fa}$ satisfy the commutation condition  with 
	$\bF: \Fc \ot \Fa \to \one$, then so does $(X,m) \ot (Y,n)$. 
\end{prop}
\begin{proof}
 Denote the left hand side of \eqref{eq:compatible} by $\delta(m)$ 
 (respectively $\delta(n)$ for the object $(Y,n)$). 
 With $T(m,n)$ as in \eqref{eq:Tpmorph} we have 
 \begin{align}
 	\delta(T(m,n)) &= ~~\phantom{-}
 	\raisebox{-.5\height}{\includegraphics[scale=0.50]{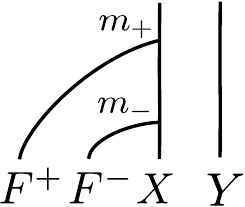}} 
 	+ \raisebox{-.5\height}{\includegraphics[scale=0.50]{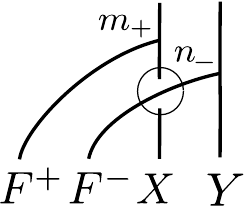}} 
 	\boxed{+ \raisebox{-.5\height}{\includegraphics[scale=0.50]{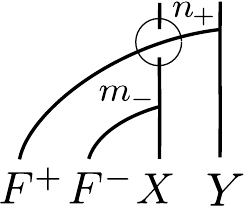}} } 
 	+ \raisebox{-.5\height}{\includegraphics[scale=0.50]{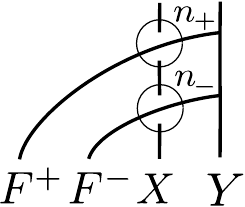}} \nonumber\\
 	&\phantom{=}~~ - \raisebox{-.5\height}{\includegraphics[scale=0.50]{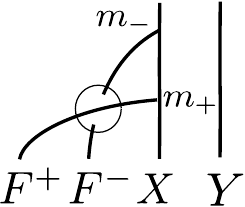}} 
 	- \raisebox{-.5\height}{\includegraphics[scale=0.50]{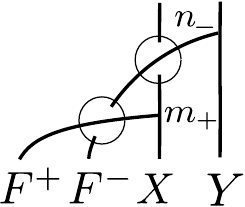}} 
 	\underbrace{\boxed{- \raisebox{-.5\height}{\includegraphics[scale=0.50]{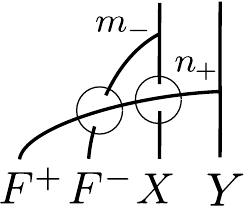}} }}_{=0} 
 	- \raisebox{-.5\height}{\includegraphics[scale=0.50]{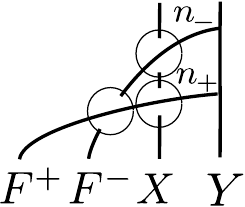}} \nonumber\\[1em]
 	&=~  T(\delta(m),\delta(n)) 
 	+ \underbrace{\raisebox{-.5\height}{\includegraphics[scale=0.50]{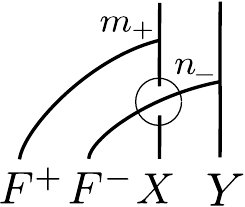}} 
 	- \raisebox{-.5\height}{\includegraphics[scale=0.50]{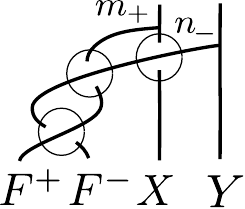}}}_{=0} \quad . 
 \end{align}
That the difference of the last two terms is zero follows from the assumption that $F^+$ and $F^-$ are mutually transparent.
Denote the right hand side of \eqref{eq:compatible} by $r_\bF(X)$. By assumption, $\delta(m)=r_\bF(X)$ and $\delta(n)=r_\bF(Y)$. 
	The proof is thus completed by observing that 
	\begin{align}
		\tpmorph(r_\bF(X),r_\bF(Y)) &= 
 		\raisebox{-0.5\height}{\includegraphics[scale=0.450]{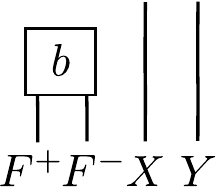}} \
 		\underbrace{- \raisebox{-0.5\height}{\includegraphics[scale=0.450]{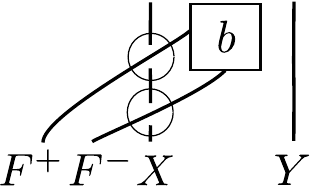}} \
 		+ \raisebox{-0.5\height}{\includegraphics[scale=0.450]{img/eqTPcommcond_2.pdf}}}_{=0} \
 		- \raisebox{-0.5\height}{\includegraphics[scale=0.450]{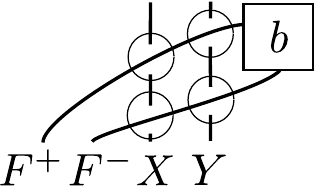}} \nonumber\\
 		&= r_\bF(X \ot Y) \ .
	\end{align}
\end{proof}

\begin{rem} \label{rem:comments-on-comm-cond}
\begin{enumerate}[(i)]
\item 
If $\CatD$ is itself braided, there is a generic class of solutions to the requirement that  $\Fc$ and $\Fa$ must be mutually transparent, namely, $\Fc = \iota^+(\tFc)$ and $\Fa = \iota^-(\tFa)$ for some $\tFc$, $\tFa$ in $\CatD$, where $\iota^\pm$ are as in \eqref{eq:Functors=i+,i-}. More generally, in the setting discussed in Section \ref{ssec:pert-def-and-cat} and Table \ref{tab:F-notation},
 if $F^\pm_{\Cat} \in \Cat$ 
 and $\aind^\pm$ are the functors as introduced after 
 Proposition \ref{prop:alphaind}, 
 then $F^+ = \aind^+(F^+_{\Cat})$ and $F^- = \aind^-(F^-_{\Cat})$ are mutually transparent.
\item 
 Let $g_1: \Fc \to \Fc$, $g_2: \Fa \to \Fa$ be morphisms in $\CatD$ and 
 assume that $(X,m) \in \CatD_{\Fc \oplus \Fa}$ satisfies the commutation condition  with 
 $\bF: \Fc \ot \Fa \to \one$. Then $(X,m \circ (g_1 \oplus g_2))$ 
 satisfies the commutation condition  with $\bF \circ (g_1 \ot g_2)$. 
 In particular, if $(X,m)$ satisfies the commutation condition  with $\bF$, then so does $(X,-m)$. 
 Suppose that $\CatD$ is $\C$-linear, as is the case in the examples we study in 
 Sections \ref{sec:uncompboson} and \ref{sec:compboson} below. 
 In this case, for each $(X,m)$ which satisfies the commutation condition  with $\bF$, the condition also holds for the family $\big(X,m \circ \{(w\cdot\id) \oplus (w^{-1}\cdot\id)\}\big)$, where $w \in \C^\times$. 
\end{enumerate}
\end{rem}

We are now going to relate the commutation condition to 
Yetter-Drinfeld modules over tensor algebras.
As Yetter-Drinfeld modules are defined in braided categories, we restrict 
ourselves to the case that $\CatD$ is braided and comment briefly on the general case in 
Remark \ref{rem:D-not-braided} below.

Specifically, for $\tilde F^\pm \in \CatD$ the commutation condition in $\CatD_F$ with $F = \iota^+(\tilde F^+) \oplus \iota^-(\tilde F^-)$ will turn out to be equivalent to the Yetter-Drinfeld condition for $\talg(\tilde F^\pm)$.
The Hopf algebras $\talg(\tilde F^\pm)$ are understood to both satisfy the bialgebra axiom (see Figure \ref{fig:Alg}\,b in Appendix \ref{appx:alg-etc-in-braided-cats}) with respect to the braiding in $\CatD$. In order to avoid lots of tildes, up to Remark \ref{rem:D-not-braided} we will depart from our previous notation and write $F^\pm$ instead of $\tilde F^\pm$. That is, $F^\pm$ are objects in $\CatD$ instead of $\DZ(\CatD)$.

\medskip

\begin{figure} 
 \centering
 \begin{align*}
 \begin{array}{cccc}
 	\text{(a)} 
 	& \raisebox{-.5\height}{$ \rho^{*n} := \raisebox{-0.5\height}{\includegraphics[scale=0.50]{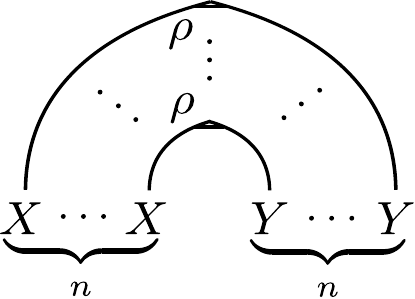}} $} \hspace{5em}
 	\text{(b)}
 	& \raisebox{-.5\height}{\raisebox{-0.5\height}{\includegraphics[scale=0.50]{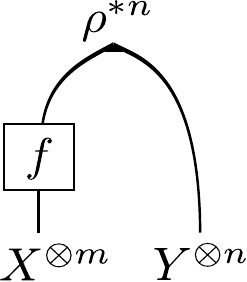}} 
 		= \raisebox{-0.5\height}{\includegraphics[scale=0.50]{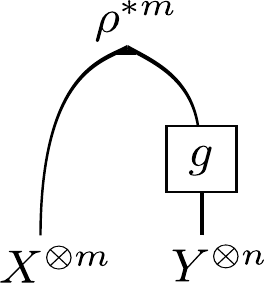}}}
 \end{array} 
 \end{align*}
 \caption{(a) Definition of $\rho^{*n}$ for a morphism $\rho : X \ot Y \to \one$. (b) Identity fulfilled by $\rho$-dual morphisms.} \label{fig:Duality}
\end{figure}

Given objects $X,Y$ in a monoidal category and a morphism $\rho : X \ot Y \to \one$, we define $\rho^{*n}$ as in Figure \ref{fig:Duality}\,(a) (cf.\ \cite[Sect.\,2.4]{Be95}). It is understood that $\rho^{*0} : \one \otimes \one \to \one$ is the unit isomorphism.
Two morphisms $f: X^{\ot m} \to X^{\ot n}$ and $g: Y^{\ot n} \to Y^{\ot m}$ are said to be $\rho$-{\em dual}, 
if the identity in Figure \ref{fig:Duality}\,(b) holds. 

\begin{df} (cf.\ \cite[Sect.\,2.4\,\&\,Def.\,3.1.2]{Be95})\footnote{ 
	We call a \emph{Hopf pairing} what is 
	called \emph{bialgebra pairing} in \cite{Be95}, since the natural compatibility 
	condition with the antipode is a consequence; 
	Yetter-Drinfeld modules -- originally called ``crossed bimodules'' in 
	\cite{Yet90} -- usually carry a module- and a comodule-structure. 
	We take the fact that the category of ``generalised Yetter-Drinfeld modules'', 
	carrying two module-structures, is denoted by $\mathcal{DY}$ in \cite{Be95} 
	as a justification to call its objects Yetter-Drinfeld modules.}
 Let $H_1,H_2$ be Hopf algebras in a braided category. 
 \begin{enumerate}[(i)]
  \item A \emph{Hopf pairing} for $H_1,H_2$ is a morphism $\rho: H_1 \ot H_2 \to \one$ 
 	such that the (co)algebra structures on $H_1$ and $H_2$ are $\rho$-dual in the sense that
 	\begin{gather} \label{eq:Hopfpairing_a}
  	\raisebox{-0.5\height}{\includegraphics[scale=0.50]{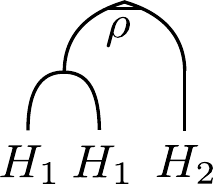}}
  	\ = \raisebox{-0.42\height}{\includegraphics[scale=0.50]{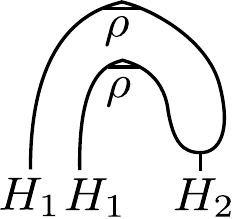}} 
	\quad ,  	\qquad 
  	\raisebox{-0.5\height}{\includegraphics[scale=0.50]{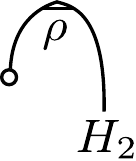}}
  	= \raisebox{-0.57\height}{\includegraphics[scale=0.50]{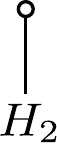}}
  	\end{gather}
  	and
 	\begin{gather} \label{eq:Hopfpairing_b}
  	\raisebox{-0.5\height}{\includegraphics[scale=0.50]{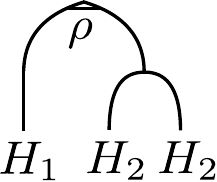}}
  	= \ \raisebox{-0.42\height}{\includegraphics[scale=0.50]{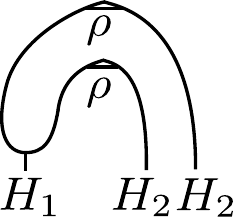}} 
	\quad ,  	\qquad 
  	\raisebox{-0.5\height}{\includegraphics[scale=0.50]{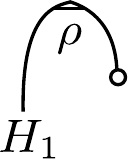}}
  	\ = \raisebox{-0.57\height}{\includegraphics[scale=0.50]{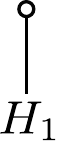}}
  	\end{gather}
 	hold. 
  \item Let $\rho: H_1 \ot H_2 \to \one$ be a Hopf pairing. 
 	A \emph{Yetter-Drinfeld module} $(X,m_1,m_2)$ for $H_1,H_2,\rho$ is an 
 	$H_1$-left module $(X,m_1)$ and an $H_2$-right module $(X,m_2)$ satisfying the 
 	Yetter-Drinfeld condition: 
 	\begin{gather} \label{eq:YD-cond}
  		\raisebox{-0.5\height}{\includegraphics[scale=0.50]{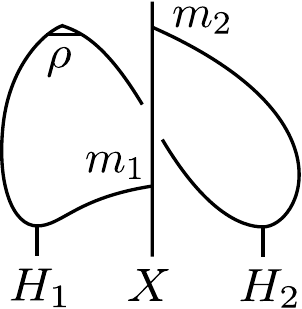}} 
  		\quad = \quad \raisebox{-0.5\height}{\includegraphics[scale=0.50]{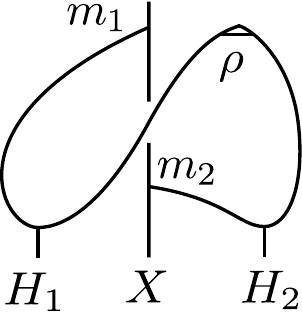}} \quad .
 	\end{gather}
 	Denote the category of Yetter-Drinfeld modules for $H_1,H_2,\rho$ by 
 	$H_1\YD{\rho} H_2$. 
 \end{enumerate}
\end{df}

\begin{ex} \label{ex:YD-modules}
	Suppose that $H_1$ and $H_2$ have invertible antipodes and let $X$ be an object. 
	Then 
	$(X \ot H_1,m_1,m_2)$ is a Yetter-Drinfeld module for 
	\begin{align} \label{eq:YD-mod-struc-H1--a}
	m_1 = \ \raisebox{-0.5\height}{\includegraphics[scale=0.50]{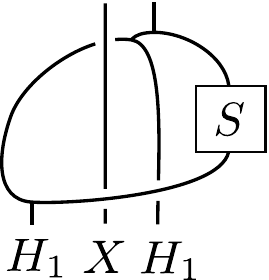}} \quad , 
	\qquad 
	m_2 = \ \raisebox{-0.5\height}{\includegraphics[scale=0.50]{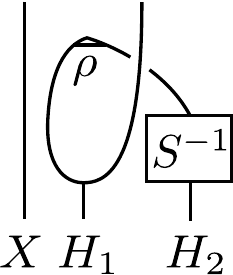}} \ . 
	\end{align} 
	Similarly, $(H_1 \ot X,\tilde m_1,\tilde m_2)$ is a Yetter-Drinfeld module with
	\begin{align} \label{eq:YD-mod-struc-H1--b}
	\tilde m_1 = \ \raisebox{-0.5\height}{\includegraphics[scale=0.50]{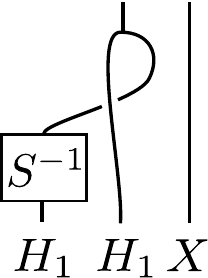}} \quad , 
	\qquad 
	\tilde m_2 = \ \raisebox{-0.5\height}{\includegraphics[scale=0.50]{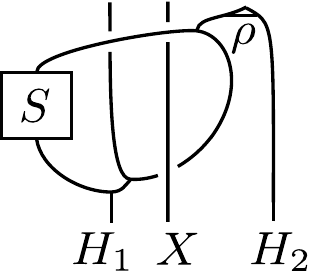}} \ . 
	\end{align} 
	The proof that $(X \ot H_1,m_1,m_2)$ satisfies the Yetter-Drinfeld 
	condition is given in \cite[Prop.\,B.4 \& Fig.\,2]{Ma95b} (C.4 in arXiv version)
	in the case $X = \one$. 
	The proof for $X \neq \one$ and the proof that $H_1 \ot X$ is a Yetter-Drinfeld module are analogous.
	A related kind of modules is described 
	in \cite[Sect.\,3.3]{SemTip11}. 
\end{ex}

\begin{lem}[\cite{Schau96}] \label{lem:extHopf}
	Let $\CatD$ be braided and suppose that Condition \ref{cond:councop} 
	is fulfilled. 
	Let $\Fc,\Fa \in \CatD$
	and let $\bF: \Fc \ot \Fa \to \one$ 
	be a morphism. 
	Then there is a unique 
	Hopf pairing $\rho(b): \talg(\Fc) \ot \talg(\Fa) \to \one$ satisfying
	$\rho(b) \circ (\iota_{\Fc} \ot \iota_{\Fa}) = \bF$.
\end{lem}

The Hopf pairing $\rho(b)$ is a generalisation of 
Lusztig's pairing, see \cite[Prop.\,1.2.3]{Lus94}.
\begin{proof}
{\em Existence:} 
	Note that the tensor algebra $\talg(V)$ of an object $V \in \CatD$ is ($\N$-)graded as a bialgebra. Moreover, it is the 
	free graded bialgebra over $V$ in the following sense: given a graded bialgebra 
	$A \cong \bigoplus_{n \in \N} A_n$ in $\CatD$ such that $A_0 = \one$ and a morphism 
	$f|_V: V \to A_1$, there exists a unique map of graded bialgebras 
	$f: \talg(V) \to A$ with $f \circ \iota_1 = f|_V$ 
	(cf.\ the comment above Def.\,2.5 in \cite{Schau96}). 
	Here, $\iota_n, \pi_n$ are the canonical 
	embeddings and projections for $\bigoplus_{n \in \N} V^{\ot n}$. 
	
	The \emph{tensor coalgebra} 
	$C(V)$ of $V$ is equal to 
	$\talg(V) = \bigoplus_{n \in \N} V^{\ot n}$ as an object in $\CatD$, 
	but its coproduct is the deconcatenation coproduct $\Delta: C(V) \to C(V) \ot C(V)$, 
	$\Delta = \sum_{n \in \N} \sum_{k=0}^n (\iota_k \ot \iota_{n-k}) \circ \pi_n$, 
	and its counit is $\varepsilon = \pi_0$. 
	The tensor coalgebra has a unique algebra structure with unit $\iota_0: \one \to C(V)$ 
	turning it into a graded bialgebra. This is obtained from the universal property of the tensor coalgebra, see \cite[Cor.\,2.4]{Schau96}. 
	Hence, there is a unique map of bialgebras $\talg(V) \to C(V)$ sending 
	$V$ to $V$. It is given by 
	$\Sym = \sum_{n\in\N} \iota_n \circ \mathbf{S}_n \circ \pi_n$ \cite[Thm.\,2.9]{Schau96}, 
	where the $\mathbf{S}_n$'s are the braid group symmetrisers as in \cite[Def.\,2.6]{Schau96}. 
	
	Define 
	$\tilde\rho(b): \left( \bigoplus_{n \in \N} (\Fc)^{\ot n} \right) 
		\ot \left( \bigoplus_{n \in \N} (\Fa)^{\ot n} \right) \to \one$ 
	by $\tilde\rho(b) := \sum_{n \in \N} b^{*n} \circ (\pi_n \ot \pi_n)$. 
	Observe that $\tilde\rho(b)$ satisfies \eqref{eq:Hopfpairing_a} for 
	$H_1 = \talg(\Fc)$ and $H_2 = C(\Fa)$, 
	i.e.\ $\tilde\rho(b) \circ (\mu_{\talg(\Fc)}\ot \id) = 
		\tilde\rho(b) \circ (\id \ot \tilde\rho(b) \ot \id) 
		\circ (\id \ot \id \ot \Delta_{C(\Fa)})$ 
	and the compatibility condition with the counit also holds. 
	Similarly, $\tilde\rho(b)$ satisfies \eqref{eq:Hopfpairing_b} for 
	$H_1 = C(\Fc)$ and $H_2 = \talg(\Fa)$. 
	Define $\rho(b)$ as 
	\begin{align}
		\rho(b) := \sum_{n \in \N} b^{*n} \circ (\mathbf{S}_n \ot \id)\circ(\pi_n \ot \pi_n) \ . 
	\end{align}
	Let $S_n = \text{Aut}(\{1,\ldots,n\})$ be the $n$-th symmetric group, $B_n$ the $n$-th Artin 
	braid group, let $L: S_n \to B_n$ be the Matsumoto section 
	(see \cite{Schau96}, above Def.\,2.6) and 
	$\sigma_n$ be the canonical map $B_n \to \text{Aut}(V^{\ot n})$ 
	induced by the braiding. One can convince oneself that 
	$b^{*n} \circ (\sigma_n(L(f)) \ot \id) 
		= b^{*n} \circ (\id \ot \sigma_n(L(r \circ f \circ r)))$ 
	with $r \in S_n$ being the ``reflection map'' sending $k$ to $n+1-k$. 
	As $\mathbf{S}_n = \sum_{f \in S_n} \sigma_n(L(f))$, this implies 
	$\rho(b) = \sum_{n \in \N} b^{*n} \circ (\id \ot \mathbf{S}_n)\circ(\pi_n \ot \pi_n)$. 
	It follows that $\rho(b)$ is a Hopf pairing $\talg(\Fc) \ot \talg(\Fa) \to \one$, 
	and by its definition it satisfies $\rho(b) \circ (\iota_{\Fc} \ot \iota_{\Fa}) = \bF$. 

\medskip\noindent
{\em Uniqueness:} Let $\rho': \talg(\Fc) \ot \talg(\Fa) \to \one$ be 
	a Hopf pairing. 

	We first show that $\rho'$ preserves the tensor degree, i.e. 
	$\rho' \circ (\iota_m \ot \iota_n) = 0$ unless $m = n$. 
	By property \eqref{eq:Hopfpairing_a} and by definition of 
	the counit on $\talg(F^-)$, we have 
	$\rho'\circ (\eta \ot \iota_n) = \varepsilon \circ \iota_n = 0$ for 
	all $n > 0$. 
	Suppose now 	that $(\rho' \circ (\iota_{m'} \ot \iota_n) \neq 0) \Rightarrow (m' = n)$ 
	holds for all $0 \leq m' < m$ (we know this for $m = 1$, already). 
	Since the coproduct on $\talg(F^+)$ preserves the tensor 
	degree, one can write 
	$\Delta\circ\iota_m = \sum_{k=0}^m (\iota_k \ot \iota_{m-k})\circ\Delta_{k,m-k}$ 
	for some $\Delta_{k,m-k}:\, F^{\ot m} \to F^{\ot k} \ot F^{\ot (m-k)}$. 
	By property \eqref{eq:Hopfpairing_b}, we have 
	$\rho'\circ(\iota_{m} \ot \iota_n) 
	= \rho'\circ(\iota_m \ot (\mu_{n-1,1}\circ(\iota_{n-1} \ot \iota_1)))	
	= \rho'\circ(\id_{\talg(F^+)} \ot \rho' \ot \id_{F^-})
		\circ ((\Delta \circ \iota_{m}) \ot (\iota_{n-1} \ot \iota_1)) 
		= \sum_{k = 0}^{m} \rho'\circ(\id_{(F^+)^{\ot k}} \ot \rho' \ot \id_{F^-})
		\circ (\Delta_{k,m-k} \ot (\iota_{n-1} \ot \iota_1))$. 
	Since $k < m$ or $m-k < m$ holds always, only a summand 
	with $k=1$ and $m-1 = n-1$ can contribute, but this means $m = n$. 
	Induction over $m$ shows that $\rho'$ preserves the degree.
	
	By iterating, say, property \eqref{eq:Hopfpairing_a} of the Hopf pairing one finds 
	$\rho' \circ (\iota_n \ot \iota_n) = b^{*n} \circ (\iota_1 \ot \cdots \ot \iota_1 \ot \Delta_{1,\dots,1})$.
	Here $\Delta_{1,\dots,1}$ is the $n$-fold iteration of the coproduct, giving a map $\talg(\Fa) \to \talg(\Fa)^{\ot n}$, applied to $(\Fa)^{\ot n}$ and then composed with the projection to $\Fa$ in each component.
	By \cite[Cor.\,2.8]{Schau96}, $\Delta_{1,\dots,1} = \mathbf{S}_n$, completing the argument.
\end{proof}

\begin{prop} \label{thm:comp_eqiv_YD}
 Let $\CatD$ be braided and suppose that Condition \ref{cond:councop} 
 is fulfilled. Let $F \in \DZ(\CatD)$ be of the form $F = \iota^+(\Fc) \oplus \iota^-(\Fa)$ 
 for $\Fc,\Fa \in \CatD$. Given a morphism $\bF: \Fc \ot \Fa \to \one$ 
 in $\CatD$, let $\rho: \talg(\Fc) \ot \talg(\Fa) \to \one$ be its
 extension 
 to a Hopf pairing via Lemma \ref{lem:extHopf}. 
 Let $X \in \CatD$ and $m_\pm$ be left-/right-actions of $\talg(F^\pm)$ 
 on $X$. 
The following are equivalent:
 \begin{enumerate}[(a)]
 \item $(X,m_+,m_-)$ is a Yetter-Drinfeld module for $\talg(\Fc), \talg(\Fa), \rho$ 
 \item The object $(X,m)$ of $\CatD_F$ with $m = m_+ \circ (\iota_\Fc \ot \id_X)  
	\,\oplus \,m_- \circ (\id_X \ot \iota_\Fa) \circ c_{X,\Fa}^{-1}$ satisfies the commutation condition  with $\bF$. 
 \end{enumerate}
\end{prop}
\begin{proof}
(a)$\,\Rightarrow\,$(b): 
	If $(X,m_+,m_-)$ is a Yetter-Drinfeld module, then by 
	\begin{align}
		\raisebox{-0.5\height}{\includegraphics[scale=0.50]{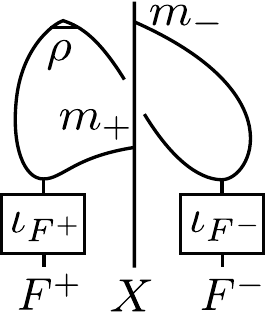}} 
		\quad &= \quad \raisebox{-0.5\height}{\includegraphics[scale=0.50]{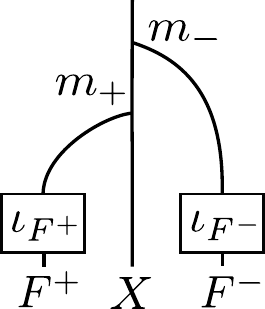}} 
		\quad + \quad \raisebox{-0.5\height}{\includegraphics[scale=0.50]{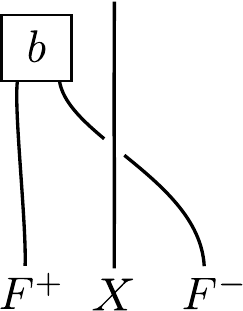}} , \nonumber\\ \nonumber\\
		\raisebox{-0.5\height}{\includegraphics[scale=0.50]{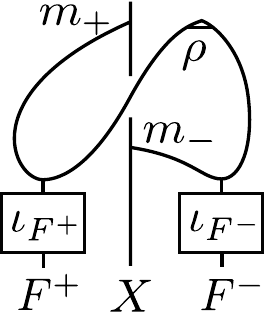}} 
		\quad &= \quad \raisebox{-0.5\height}{\includegraphics[scale=0.50]{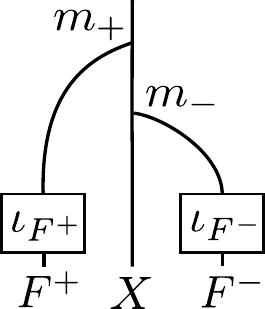}} 
		\quad + \raisebox{-0.5\height}{\includegraphics[scale=0.50]{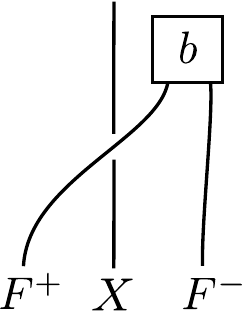}} \ ,
	\end{align}
	the commutation condition  for  $(X,m)$ with $\bF$ is just the Yetter-Drinfeld condition restricted to the subobjects 
	$\Fc$ and $\Fa$ of $\talg(\Fc),\talg(\Fa)$. 

\medskip\noindent
(b)$\,\Rightarrow\,$(a): 
Let $\iota_A: A \to \talg(\Fc)$, $\iota_{B}: B \to \talg(\Fc)$ and 
	$\iota_C: C \to \talg(\Fa)$ be morphisms in $\CatD$. 
	Suppose that the Yetter-Drinfeld condition holds on 
	$\iota_A \ot \id_X \ot \iota_C$ and $\iota_{B} \ot \id_X \ot \iota_C$. Then it also holds on 
	$(\mu \circ (\iota_{A} \ot \iota_{B})) \ot \id_X \ot \iota_C$: 
	\begin{align}
		&\raisebox{-0.5\height}{\includegraphics[scale=0.50]{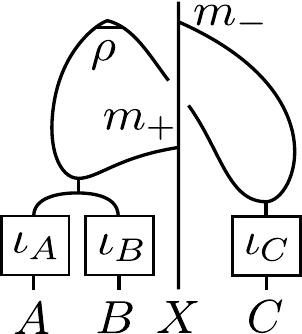}} 
		\quad = \quad \raisebox{-0.5\height}{\includegraphics[scale=0.50]{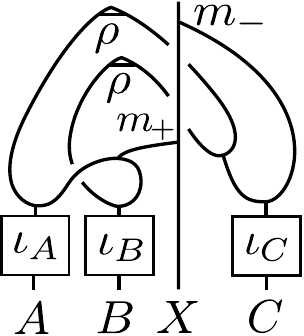}} 
		\quad = \quad \raisebox{-0.5\height}{\includegraphics[scale=0.50]{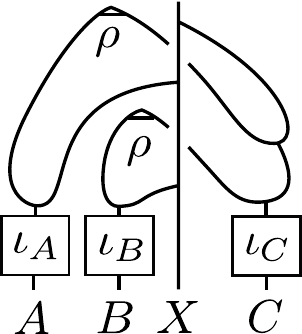}} \nonumber\\ \nonumber\\
		&= \ \raisebox{-0.5\height}{\includegraphics[scale=0.50]{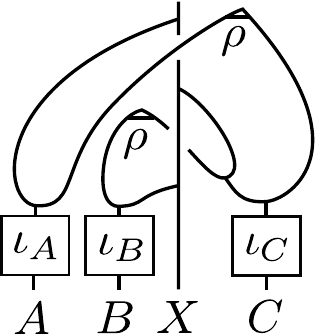}} 
		\quad = \quad \raisebox{-0.5\height}{\includegraphics[scale=0.50]{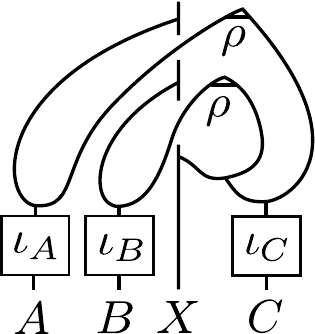}} 
		\quad = \ \raisebox{-0.5\height}{\includegraphics[scale=0.50]{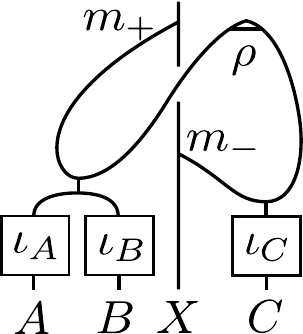}} \ .
	\end{align}
For fixed $\iota_C$, and by taking $A = F^+$ and $B = (F^+)^{\otimes n}$, an inductive argument now shows that the Yetter-Drinfeld condition holds on all of $\talg(F^+)$. An analogous argument for $\talg(\Fa)$ then shows that the Yetter-Drinfeld condition follows from the commutation condition  for $(X,m)$ with $\bF$.
\end{proof}

The Yetter-Drinfeld condition is compatible with 
the standard braided tensor product for left and right modules, 
see \cite[Sect.\,3.3]{Be95}. 
That is to say, if $(X,m_X^1,m_X^2)$ and $(Y,m_Y^1,m_Y^2)$ 
are Yetter-Drinfeld modules for Hopf algebras 
$H_1,H_2$ and a Hopf pairing $\rho$, then so is 
$(X \ot Y, m_{X \ot Y}^1, m_{X \ot Y}^2)$, 
where 
\begin{align}
m_{X \ot Y}^1 := \ \raisebox{-0.5\height}{\includegraphics[scale=0.50]{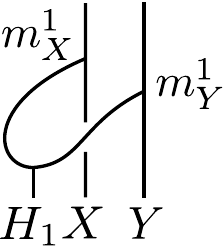}} \ , \qquad 
m_{X \ot Y}^2 := \ \raisebox{-0.5\height}{\includegraphics[scale=0.50]{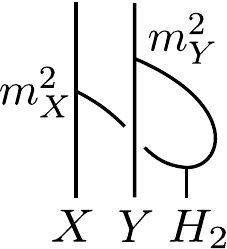}} \ .
\end{align}
This makes the category of Yetter-Drinfeld modules 
monoidal \cite[Sect.\,3.3]{Be95}. 
Refining the above theorem one obtains:

\begin{thm}\label{thm:YD-com-rel-mon}
	With the conditions and notation as in Proposition \ref{thm:comp_eqiv_YD}, 
	the category $\talg(\Fc)\YD{\rho}\talg(\Fa)$ 
	and the full subcategory of $\CatD_F$ consisting of objects satisfying the commutation condition  with $\bF$ 
	are isomorphic as monoidal categories.
\end{thm}
\begin{proof}
	Let $\CatD'_F$ be the full subcategory of $\CatD_F$ 
	consisting of objects satisfying the commutation condition  with $\bF$. 
	Two inverse functors $\talg(\Fc)\YD{\rho}\talg(\Fa) \to\CatD'_F$ 
	and $\CatD'_F \to \talg(\Fc)\YD{\rho}\talg(\Fa)$ are described on 
	objects in Proposition \ref{thm:comp_eqiv_YD} above and on morphisms 
	they are defined to be the identity. 
	Monoidality
	may be shown in a similar way as in 
	Proposition \ref{prop:tensormod}. 
\end{proof}

The statement of Theorem \ref{thm:YD-com-rel-mon} is the main observation made in the present paper: In Section \ref{ssec:PertDefNonlocalConsCharges} we linked the existence of non-local conserved charges in a perturbed conformal field theory to the commutation condition in Definition \ref{df:commcond}. The above result in turn relates the commutation condition to a standard mathematical object, namely Yetter-Drinfeld modules, and shows that their tensor products agree with the composition of the corresponding non-local charges.

In the following we will briefly make a connection to the study of Nichols algebras and comment on the situation where $\CatD$ is not braided.

\medskip

An object $X$ in a monoidal category is said to be (left-)dualisable
if there are an object $X^*$ and morphisms 
$X \ot X^* \to \one$ (\emph{evaluation})
and 
$\one \to X^* \ot X$ (\emph{coevaluation}), such that
\begin{align}
 	\raisebox{-0.5\height}{\includegraphics[scale=0.50]{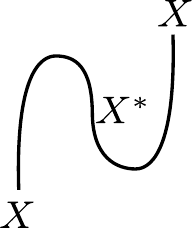}} 
 		= \raisebox{-0.6\height}{\includegraphics[scale=0.50]{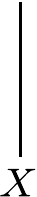}} \qquad 
 		\text{and} \qquad 
 	\raisebox{-0.5\height}{\includegraphics[scale=0.50]{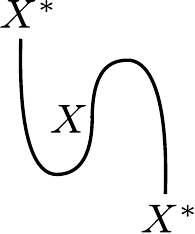}} 
 		= \raisebox{-0.6\height}{\includegraphics[scale=0.50]{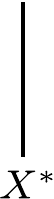}} 
	 \ .
\end{align}

\begin{rem} \label{rem:YD-braided}
	If $H_2$ is dual to $H_1$, i.e.\ if there is a coevaluation for $\rho$, 
	then $H_1\YD{\rho}H_2$ allows for a braiding. This is because the coevaluation can be used to turn the right-action of $H_2$ into 
	a right-coaction of $H_1$. Hence, $H_1\YD{\rho}H_2$ and, with the notation 
	as in \cite{Be95}, ${}_{H_1}\mathcal{DY}(\CatD)^{H_1^{op}}$ are monoidally isomorphic.\footnote{
Some care has to be taken because $H_1^{op}$ is a bialgebra only in $\CatD^\text{rev}$; we refer to \cite{Be95} for details.}
	The latter category is shown to be braided in \cite[Sect.\,3.4]{Be95}. 
\end{rem}
	
Tensor algebras typically have no duals because they are infinite direct sums. Nonetheless, it would be desirable to have a braiding on at least part of $\talg(\Fc)\YD{\rho}\talg(\Fa)$ since this would 
imply that the corresponding non-local conserved charges mutually commute.
In special cases this can be achieved as follows.

Consider Yetter-Drinfeld modules with the additional property that the radicals of $\rho(b)$ in $\talg(\Fc)$ and $\talg(\Fa)$ act trivially. In other words, consider Yetter-Drinfeld modules of the quotients of $\talg(\Fc)$ and $\talg(\Fa)$ by the radicals of $\rho(b)$.
If $\bF: \Fc \ot \Fa \to \one$ is a duality (as we will assume from here on), these 
	quotients are called the \emph{Nichols} or \emph{bitensor algebras} $\B(F^\pm)$ of $F^\pm$, 
	see e.g.\ \cite[Def.\,2.5 \& Prop.\,2.9]{Schau96} or \cite[Prop.\,2.10]{AndSchn02}. 
	The (graded) Hopf pairing $\rho(b)$ of $\talg(\Fc)$ and $\talg(\Fa)$ 
	descends to a nondegenerate graded Hopf pairing 
	$\tilde\rho = \sum_{n \in \N} \tilde\rho_n$ 
	of $\B(\Fc)$ and $\B(\Fa)$, where 
	each $\tilde\rho_n: \B(\Fc)_n \ot \B(\Fa)_n \to \one$ is a 
	duality (provided that the $F^\pm$ themselves sit in some full abelian monoidal subcategory of $\CatD$ which has duals).

For certain choices of $F^\pm$ and $b$ it may happen that the quotients  $\B(F^\pm)$ are ``small enough'' for $\tilde\rho$ to be a duality. By Remark \ref{rem:YD-braided}, in this case the category $\B(\Fc)\YD{\tilde\rho}\B(\Fa)$ is braided and hence provides a full braided subcategory of $\talg(\Fc)\YD{\rho}\talg(\Fa)$ (see e.g.\ \cite[Thm.\,4.3]{AndSchn02} and references therein for examples where this happens). 
Nichols algebras with duals have recently been studied in a different context, namely in relation to screening operators in conformal field theory \cite{SemTip11, Sem11a, Sem11}.

A possibility to obtain a braided subcategory of $\talg(\Fc)\YD{\rho}\talg(\Fa)$ without assuming that $\B(F^\pm)$ have duals is to consider Yetter-Drinfeld modules  for $\B(\Fc)$ and 
	$\B(\Fa)$ of the form $\B(\Fc) \ot X$, where 
	the action is given in Example \ref{ex:YD-modules}. To see that the subcategory generated by
	such modules is braided, note that
	$\B(\Fa)$ acts ``locally nilpotently'' on modules $\B(\Fc) \ot X$. 
	So the infinite sum occuring when writing down the formal 
	coevaluation for $\tilde\rho$ actually becomes finite 
	(cf.\ \cite[Prop. 6.6.3]{EFK98} for the case of quantum groups). 

\begin{rem}\label{rem:D-not-braided}For the discussion of Yetter-Drinfeld modules we assumed for simplicity that $\CatD$ is braided. Let us sketch how the results can be formulated in the more general setting of Section \ref{sec:tensor_algebras}.
Let thus $\CatD$ be a monoidal category satisfying Condition \ref{cond:councop},
	let $F \in \DZ(\CatD)$ be of the form $F = \Fc \oplus \Fa$ 
	for some mutually transparent $\Fc,\Fa \in \DZ(\CatD)$, and 
	let $\bF: \Fc \ot \Fa \to \one$ 
	be a morphism in $\CatD$. 
Let us write $\overline{\Fa}$ instead of $\Fa$ when considering $\Fa$ as an object in $\DZ(\CatD)^\mathrm{rev}$. The corresponding Hopf algebra is $\talg(\overline{\Fa}) \in \DZ(\CatD)^\mathrm{rev}$. Note that the coproduct on $U\talg(F^-) = U\talg(\overline{F^-}) \in \CatD$ changes when the tensor algebra is taken in $\DZ(\CatD)^\mathrm{rev}$ instead of $\DZ(\CatD)$.
	Similar to Lemma \ref{lem:extHopf} one verifies that 
	$b$ extends to a morphism 
	$\rho(b): \talg(\Fc) \ot \talg(\overline{\Fa}) \to \one$ in $\CatD$ 
	satisfying the Hopf pairing conditions \eqref{eq:Hopfpairing_a} 
	and \eqref{eq:Hopfpairing_b}. 
	Mutual transparency of $\Fc$ and $\Fa$ is needed for 
	showing that the $n$'th symmetriser $\mathbf{S}_n$ for $\talg(F^+)$ 
	is $b^{*n}$-dual to the one for $\talg(\overline{F^-})$. 
	Let $X \in \CatD$ and let $m_+: \talg(\Fc) \ot X \to X$, 
	$m_-: X \ot \talg(\Fa) \to X$ be actions in $\CatD$. 
	Then the following are equivalent:
	\begin{enumerate}[(a)]
	\item The morphisms $m_\pm$ and $\rho(b)$ satisfy:
		\begin{align}\label{eq:YD-cond-for-nonbraided-cat}
		&\raisebox{-0.5\height}{\includegraphics[scale=0.50]{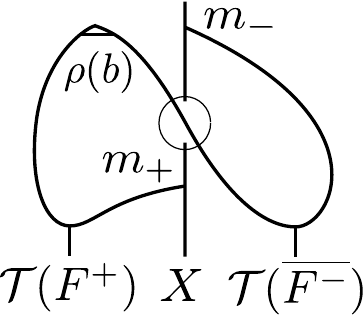}} 
		\quad = \quad 
		\raisebox{-0.5\height}{\includegraphics[scale=0.50]{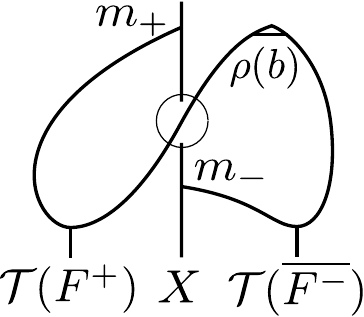}} \ ,
		\end{align}
	\item The object $(X,m)$ with $m = m_+ \circ (\iota_\Fc \ot \id_X) 
		\, \oplus \, m_- \circ (\id_X \ot \iota_\Fa) \circ \varphi_{\Fa,X}$
		of $\CatD_F$ satisfies the commutation condition  with $\bF$.
	\end{enumerate}
	This can be proven in a similar manner as in Proposition \ref{thm:comp_eqiv_YD}, 
	but the assumption that $F^+$ and $F^-$ are mutually transparent is 
	crucial. 
	One can also check that \eqref{eq:YD-cond-for-nonbraided-cat} is 
	compatible with tensor products. 
	This allows to generalise Theorem \ref{thm:YD-com-rel-mon}. 
\end{rem}

\section{Uncompactified free boson \& sin(h)-Gordon theory}
\label{sec:uncompboson}

In this section we apply our general recipe to the theories
considered in \cite{BLZ96, BLZ97b}, which motivated our work.\footnote{
   In  \cite{BLZ96, BLZ97b} the free boson is considered with background charge. This amounts
   to modifying the stress tensor by a total derivative and changes the central charge away from 1.
   We implicitly take the free boson with $c=1$, but this does not affect the algebraic properties we discuss.}
We show how to recover a deformed enveloping algebra 
of the loop algebra $L\lsl_2$ as the quantum algebra 
underlying the integrable structure in those theories. 

In this section, all vector spaces, algebras, etc.\ are over $\C$. 

\subsection{Input data: the braided category $\Cat$ and the choice of bulk field $\bF$}
\label{ssec:unc_input-data}

We say that an algebra $A$ acts 
\emph{semisimply} on a vector space $V$ if $V$ decomposes into 
a (possibly infinite) direct sum of simple modules. 
We denote by $\Repss(A)$ the category of semisimple representations. If 
$B \subseteq A$ is a subalgebra, we write $\RepAss{B}(A)$ for the 
category of representations of $A$ which restrict to semisimple 
representations of $B$. 

We will make use of the topological algebras $\E(\h; X)$, which we now describe.
Let $\h$ and $X$ be either lists of ``generators'' or vector spaces, thought to be spanned by generators. $\E(\h; X)$ is a topological completion of the algebra generated by variables $\h$ and $X$, where the variables from $X$ are non-commuting while those from $\h$ commute among themselves but not with $X$. The details have been deferred to Appendix \ref{apx:algE}.
The completion $\E(\h,X)$ has the property that 
exponentials $e^h$ of generators $h \in \h$ are elements of $\E(\h;X)$. 

We will now restrict our attention to the case of a single generator $h$.
Define a coproduct $\Delta$, counit $\varepsilon$ and antipode $S$ on $\E(h;0)$ by 
\begin{align}
	\Delta(h) = h \ot \one + \one \ot h ~~, \qquad 
	\varepsilon(h) = 0~~, \qquad 
	S(h) = -h \ .
\end{align}
The coproduct $\Delta$ is a map into the completed tensor product 
$\E(h;0) \bar\ot \E(h;0) := \E(h_1,h_2;0)$, where $h_1,h_2$ are to be identified with $h \ot \one$ and 
$\one \ot h$, respectively.
This equips $\E(h;0)$ with the structure of a (topological) Hopf algebra, which we will abbreviate 
by $\mathcal A$. To make $\mathcal A$ quasitriangular, we fix a complex number $\kappa$ and define the
R-matrix 
\begin{align}
	\mathcal R_{\mathcal A} = e^{\kappa \cdot h \ot h} \ ,
	\label{eq:C-gr-vsp-Rmat}
\end{align}
which is to be seen as an element of 
$\mathcal A \bar\ot \mathcal A$.
We then let 
\begin{align}
	\Cat := \Repss(\mathcal A) \ .
\end{align}
The R-matrix $\mathcal R_{\mathcal A}$ induces a braiding $c$ 
on $\Cat$, given by 
$c_{V,V'} := \tau_{V,V'} \circ (\rho \ot \rho')(\mathcal R_{\mathcal A})$ 
for semisimple representations $V = (V,\rho)$, $V' = (V',\rho')$ of $\algA$ and with $\tau$ being the tensor flip in $\Vect$; since all $\mathcal A$-modules involved are semisimple, $h$ acts diagonalisably and so the action of $\mathcal R_{\mathcal A}$ is well-defined. 

Objects of $\Cat$ are just vector spaces together with a 
diagonalisable endomorphism (the action of $h$). Simple objects of 
$\Cat$ consist of a one-dimensional vector space on which $h$ acts by multiplication with
a complex number.  
We denote by $\C_\alpha$ the representation of $\mathcal A$ 
on $\C$ determined by $h.1 = \alpha \in \C$. By definition of the coproduct on $\mathcal A$,
the tensor product of these simple objects satisfies  
\begin{align}
	\C_\alpha \ot \C_\beta \cong \C_{\alpha + \beta} \ .
\end{align}
The braiding $c$ on two simple objects is computed from the R-matrix to be
\begin{align}
	c_{\C_\alpha,\C_\beta} : \C_\alpha \ot \C_\beta \to \C_\beta \ot \C_\alpha, \qquad 
		1 \ot 1 \mapsto e^{\kappa \alpha\beta} \cdot 1 \ot 1 \ . 
		\label{eq:braiding-on-simples}
\end{align}

This completes the description of the braided monoidal category $\Cat$. The following remark outlines the relation to the free boson. 

\begin{rem}[Relation to the free boson CFT] \label{rem:unc-boson--cat-VOA} 
	The relevance of $\Cat = \Repss(\mathcal A)$ to free boson conformal field 
	theories arises as follows. 
	The Heisenberg Lie algebra is the Lie algebra with generators 
	$a_n$, $n\in\Z$ and $K$ and brackets $[a_m,a_n] = m \, \delta_{m+n,0} \, K$ and $[K,a_n] = 0$
	(see e.g.\ \cite[Sect.\,2.5]{KacVOAs}).
	To a simple object $\C_\alpha$ we associate the highest weight 
	representation $H_\alpha$ of the Heisenberg Lie algebra with highest 
	weight $\alpha$ and with $K$ acting as the identity. 
	This correspondence of simple objects extends to an equivalence 
	of $\C$-linear categories between $\Cat$ and the category of 
	representations of the Heisenberg algebra which are bounded below, have a diagonalisable action of $a_0$, and where $K$ acts as the identity, see \cite[Sect.\,3.5]{KacVOAs} and \cite[Thm.\,2.4]{RuSympFerm} for details. 

	To the highest weight vector of $H_\alpha$ we associate 
	a chiral field $V_\alpha(z)$, namely the 
	normal ordered exponential (vertex operator) 
	\begin{align}
		V_\alpha(x) = \, : e^{i\alpha\varphi(x)} : \qquad , ~~  x \in \R \ .
	\end{align}
	Here, $\varphi(x)$ is the chiral bosonic field 
	$\varphi(x) = q + a_0 \cdot x + i\sum_{n\neq 0} \frac{a_n}n e^{-in x}$, 
	with $q$ being conjugate to $a_0$, i.e.\ $[q,a_0] = i$. 
	The tensor product $\C_\alpha \ot \C_\beta \cong \C_{\alpha + \beta}$ 
	in $\Cat$ reflects the ``fusion'' 
	\begin{align}
		V_\alpha(x) \, V_\beta(y) \, \sim \, (x-y)^{\alpha\beta} \, \big( V_{\alpha + \beta}(x) + \text{reg}(x-y) 	\big)	\qquad , ~~ x > y \ .
	\end{align}
	The braiding $c$ reflects the behaviour of vertex operators under analytic continuation along a specified path exchanging $x$ around $y$
	\begin{align}
		V_\alpha(x)\,V_\beta(y) \,=\, e^{i\pi \alpha\beta} \big[ V_\beta(y)\,V_\alpha(x) \big]_{\text{anal.\,cont.}} 
		\qquad , ~~ x > y\ .
	\end{align}
	This agrees with \eqref{eq:braiding-on-simples} if $\kappa$ is set to $\pi i$. 
	A detailed discussion of how to obtain tensor product, associator and braiding from free
	boson vertex operators can be found in Sections 3--5 and Proposition 6.1 of 
	\cite{RuSympFerm} (the conventions there result in the inverse braiding relative to here). 
\end{rem}

Next we choose the bulk field $b$ to be used to perturb the free boson CFT according to the prescription in Section \ref{ssec:PertDefNonlocalConsCharges}. For the uncompactified free boson, the category $\CatD$ of Sections \ref{sec:tensor_algebras}--\ref{sec:comm_cond_YD} is given by $\Cat$ itself. We thus need to fix objects $F^+, F^- \in \DZ(\Cat)$ and a morphism $b : F^+ \otimes F^- \to \one$ in $\Cat$ (the forgetful functor $\DZ(\Cat) \to \Cat$ is implicit).

Fix a constant $\omega \in \C^\times$
 and let $F_\omega = \simp{\omega} \oplus \simp{-\omega} \in \Cat$ and
$F^\pm = \iota^\pm(F_\omega) \in \DZ(\Cat)$, where 
$\iota^\pm$ are the functors from \eqref{eq:Functors=i+,i-}. We set
\begin{align}
F := F^+ \oplus F^-  \,\in\, \DZ(\Cat) \ . 
\label{eq:uncomp-free-boson-F}
\end{align}
Let $i_{\alpha,\beta}$ be the canonical identifications 
$\simp{\alpha} \ot \simp{\beta} \stackrel{\sim}{\to} \simp{\alpha+\beta}$ 
determined by $1 \ot 1 \mapsto 1$. 
For $\zeta,\xi \in \C$ define the morphism 
$\bF_{\zeta,\xi} : F^+ \ot F^- \to \one$ in $\Cat$ (i.e.\  $\bF_{\zeta,\xi} : F_\omega \ot F_\omega \to \one$) 
as
\begin{align}
 \bF_{\zeta,\xi} = \zeta \cdot i_{\omega,-\omega} 
 	+ \xi \cdot i_{-\omega,\omega}  \ ,
	\label{eq:uncom-b_ze-xi_def}
\end{align}
where we omitted the obvious projections $F_\omega \to \C_{\pm\omega}$. 
Note that any morphism $F_\omega \ot F_\omega \to \one$ is of the form $\bF_{\zeta,\xi}$.

\begin{rem}[Relation to the sin(h)-Gordon model]\label{eq:sin(h)-Gorden-uncomp}
	The Hamiltonian of the two-dimensional massless sine-Gordon model 
	on a cylinder is given by 
	$H_0 + \lambda\int_0^{2\pi} {:}\cos(\omega \varphi^\text{full}(x)){:} \,dx$, 
	where $H_0$ is the Hamiltonian of the (non-chiral) free massless boson $\varphi^\text{full}$ and 
	$\lambda, \omega$ are real constants.
	The Hamiltonian of the sinh-Gordon model is given by 
	the same formula but with purely imaginary $\omega$.
   (See e.g.\ \cite{NT09} for a recent treatment and further references.) 
	The perturbing bulk field is 
	${:}\cos(\omega \varphi^\text{full}(x)){:} 
		\,=\, \frac 12 \,{:}e^{i\omega\varphi^\text{full}(x)}{:} 
		+ \frac 12 \,{:}e^{-i\omega\varphi^\text{full}(x)}{:} 
		\,= \frac 12 V^\text{full}_\omega(x) + \frac 12 V^\text{full}_{-\omega}(x)$, 
	where $V^\text{full}_\omega$ is the ``full vertex operator''
	$V_\omega \ot V_{-\omega}$. 
	By construction, the full vertex operator $V^\text{full}_\omega$ corresponds to
	the morphism $i_{\omega,-\omega}$. Thus, choosing
	the bulk field $\bF = \bF_{\zeta,\xi}$ 
	with $\zeta = \xi$ will describe the sine-Gordon or 
	sinh-Gordon model, depending on whether $\omega \in \R$ or $\omega \in i\R$. 
\end{rem}

\subsection{Perturbed defects: modules of the tensor algebra $\talg(F)$ in $\Cat$}
\label{ssec:unc_bosonisation}

As argued in Section \ref{ssec:pert-def-and-cat}, 
defects perturbed by the chiral constituents $F^+$ and $F^-$ of 
the bulk field $\bF$ are determined by objects of the 
category $\Cat_F$ with $F$ as in \eqref{eq:uncomp-free-boson-F}, or, equivalently, by modules of the 
tensor algebra $\talg(F)$ in $\Cat$. 
We would now like to describe $\talg(F)\Mod_\Cat$ as 
a category of representations of an ordinary Hopf algebra. 
For Hopf algebras in the braided category $\Cat$ we can, by Tannaka-Krein reconstruction, describe the category 
of its modules in $\Cat$ 
as a category of representations of an ordinary Hopf algebra. 
This algebra can, for instance, be found using the ``bosonisation'' formula 
in \cite[Thm\,4.11]{Ma95}. 
Recall that $\talg(F)$ is a Hopf algebra in $\DZ(\Cat)$ but not in $\Cat$. However, one can still describe $\talg(F)\Mod_\Cat$ as 
a category of representations of an ordinary Hopf algebra along the same lines; this is what we will do next.

\medskip

Let $\algA_F$ be the (topological) Hopf algebra given as an associative algebra by 
$\E(h$; $f^+$, $f^-$,$\bar f^+$,$\bar f^-)$ modulo 
\begin{align} \label{eq:algaf1}
 [h,f^\pm] = \pm \omega f^\pm \ , \qquad [h,\bar f^\pm] = \pm \omega\bar f^\pm \ , 
\end{align}
equipped with the coproduct given by 
\begin{gather} \label{eq:def-coproduct-A_F}
 \Delta(h) = h \ot \one + \one \ot h \ , \\
 \Delta(f^\pm) = f^\pm \ot \one + e^{\pm \kappa\omega \cdot h} \ot f^\pm ~~, \qquad 
 \Delta(\bar f^\pm) = \bar f^\pm \ot \one + e^{\mp \kappa\omega \cdot h} \ot \bar f^\pm\ . \nonumber
\end{gather}
Note that $\algA_F$ contains the algebra $\algA = \E(h;0)$ introduced in the previous section as a sub-Hopf-algebra. Thus, restricting a representation $(V,\rho)$ of $\algA_F$ to 
$\algA$ gives a tensor functor 
from $\RepAss{\algA}(\algA_F)$ to $\Repss(\algA)$. 
The actions of $f^\pm$, $\bar f^\pm$ give $(V,\rho|_{\algA})$ 
the extra structure of a $\talg(F)$-module and we obtain: 

\begin{thm}[Category for perturbed defects] \label{thm:algA_F}
 Let $\algA_F$ be as defined above. 
 The following categories are monoidally isomorphic: 
 \begin{align} 
 \talg(F)\Mod_{\Cat} \, \cong \,\Cat_F \,\cong\, \RepAss{\algA}(\algA_F) \ . 
 \end{align}
\end{thm}
\begin{proof}
The first isomorphism was the content of Proposition \ref{prop:tensormod}. 
For the second we define a functor 
\begin{align}
 \fcAlgPtUnc: \RepAss{\algA}(\algA_F) \to \Cat_F
\end{align}
as follows: If $V = (V,\rho)$ 
is an object 
of $\RepAss{\algA}(\algA_F)$, then 
let $m: F \ot V \to V$ be the map 
\begin{align} \label{eq:def-of-m}
m(\phi \ot v) 
	:= \rho(\phi^+ \cdot f^+ + \phi^- \cdot f^- 
		+ \bar \phi^+ \cdot \bar f^+ + \bar \phi^- \cdot \bar f^-).v  \ ,
\end{align} 
where 
$\phi = \phi^+ \oplus \phi^- \oplus \bar \phi^+ \oplus \bar \phi^- 
	\in \iota^+(\C_\omega \oplus \C_{-\omega}) \oplus \iota^-(\C_\omega \oplus \C_{-\omega}) 
	= F$ 
and on the right hand side, $\phi^+, \phi^-, \bar \phi^+, \bar \phi^-$ 
are merely considered as complex numbers. 
The map $m$ is an intertwiner of $\algA$-modules: 
\begin{align} \label{eq:m-morph-of-h-mods}
&\rho(h).m(\phi \ot v) 
	= \rho(\phi^+ \cdot [h,f^+] + \phi^- \cdot [h,f^-] 
		+ \bar \phi^+ \cdot [h,\bar f^+] + \bar \phi^- \cdot [h,\bar f^-]).v \nonumber\\
	&\qquad\qquad + m(\phi \ot (\rho(h).v)) \nonumber\\
	&\qquad= \rho(\omega \phi^+ \cdot f^+ - \omega \phi^- \cdot f^- 
		+ \omega \bar \phi^+ \cdot \bar f^+ - \omega \bar \phi^- \cdot \bar f^-).v 
		+ m(\phi \ot (\rho(h).v)) \nonumber\\
	&\qquad= m(\chi(h).\phi) \ot v) + m(\phi \ot (\rho(h).v)) 
	= m((\chi \ot \rho)(\Delta(h))(\phi \ot v)) \ , 
\end{align}
where $\chi$ denotes the representing homomorphism $\chi: \algA \to \End(F)$ 
and $\phi \in F$ as above. 
Hence, $\fcAlgPtUnc(V) := (V,m)$ 
is an object of $\Cat_F$. 
On morphisms $\fcAlgPtUnc$ is the identity.

For the converse direction observe that a morphism 
    $m: F \ot V \to V$ 
in particular determines four 
linear maps on the vector space $V$, namely 
$v \mapsto m((1 \oplus 0 \oplus 0 \oplus 0) \ot v)$, etc. 
Together with the action of $h$ they determine an action 
of the free algebra with generators $h, f^\pm, \bar f^\pm$. 
By Proposition \ref{prop:apx_algE} and diagonalisability of the 
action of $h$, it lifts to a representation 
of the completion $\algA_F=\E(h; f^\pm, \bar f^\pm)$. 
Then \eqref{eq:m-morph-of-h-mods} shows 
that \eqref{eq:algaf1} holds. Hence, one obtains an object in 
$\RepAss{\algA}(\algA_F)$. 

This shows that $\fcAlgPtUnc$ is an equivalence.
To see monoidality, let 
$V = (V,\rho_V)$, $W = (W,\rho_W)$ be 
two objects in $\RepAss{\algA}(\algA_F)$. 
Write $\fcAlgPtUnc(V) = (V, m_V)$, $\fcAlgPtUnc(W) = (W, m_W)$ and 
$\fcAlgPtUnc(V \ot W) = (V \ot W, m_{V \ot W})$, and let
$\mathcal R_\algA^{\pm 1} = \sum_j (\mathcal R_\algA^{\pm 1})^{(1)}_j \ot (\mathcal R_\algA^{\pm 1})^{(2)}_j$ be a decomposition of 
$\mathcal R_\algA^{\pm 1} = e^{\pm\kappa h \ot h}$ 
into elementary tensors. 
The action of 
$\phi = \phi^+ \oplus \phi^- \oplus \bar \phi^+ \oplus \bar \phi^- \in F$ 
on $v \ot w \in \fcAlgPtUnc(V,\rho_V) \ot \fcAlgPtUnc(W,\rho_W)$ 
is given by 
\begin{align}
&\tpmorph(m_V,m_W)(\phi \ot v \ot w) 
 \nonumber\\
& \stackrel{\eqref{eq:Tpmorph}}{=} 
	m_V(\phi \ot v) \ot w 
		+ ((\id_V \ot m_W) \circ (\varphi_{F,V} \ot \id_W))(\phi \ot v \ot w) \nonumber\\
&\stackrel{\eqref{eq:def-of-m}}{=} \rho_V(\phi^+ f^+ + \phi^- f^- 
		+ \bar \phi^+ \bar f^+ + \bar \phi^- \bar f^-).v \ot w \nonumber\\
	&\qquad + \sum_j \rho_V\left(\left(\mathcal R_\algA\right)^{(2)}_j\right).v 
	\ot \rho_W\left( 
	\chi\left(\left(\mathcal R_\algA\right)^{(1)}_j\right).(\phi^+ f^+ + \phi^- f^-)\right).w \nonumber\\
	&\qquad + \sum_j \rho_V\left(\left(\mathcal R_\algA^{-1}\right)^{(1)}_j\right).v 
	\ot \rho_W\left( 
	\chi\left(\left(\mathcal R_\algA^{-1}\right)^{(2)}_j\right).(\bar \phi^+ \bar f^+ + \bar \phi^- \bar f^-)\right).w \nonumber\\
&= \rho_V(\phi^+ f^+ + \phi^- f^- 
		+ \bar \phi^+ \bar f^+ + \bar \phi^- \bar f^-).v \ot w \nonumber\\
	&\qquad + \rho_V\left(e^{\kappa\omega h}\right).v 
		\ot \rho_W\left( \phi^+ f^+ \right).w 
		+ \rho_V\left(e^{- \kappa\omega h}\right).v 
		\ot \rho_W\left( \phi^- f^- \right).w \nonumber\\
	&\qquad + \rho_V\left(e^{-\kappa\omega h}\right).v 
		\ot \rho_W\left( \bar \phi^+ \bar f^+ \right).w
		+ \rho_V\left(e^{\kappa\omega h}\right).v 
		\ot \rho_W\left( \bar \phi^- \bar f^- \right).w \nonumber\\
&= (\rho_V \ot \rho_W)(\Delta(\phi^+ f^+ + \phi^- f^- 
		+ \bar \phi^+ \bar f^+ + \bar \phi^- \bar f^-)).(v \ot w) \ ,
\end{align}
The last line is precisely $m_{V \ot W}$, completing the proof of monoidality. 
\end{proof}

\subsection{Conventions for $\tilde U_\hbar(\lsl_2)$ and $\tilde U_\hbar(L\lsl_2)$}
\label{ssec:unc_U_h}

We will need two versions of quantum groups, one which uses the exponential $e^{\hbar h}$ for a number $\hbar$ and the generator $h$, and one which uses generators $k^{\pm1}$ instead. The former will be used in the treatment of the uncompactified free boson and we recall its definition presently. The generators $k^{\pm1}$ appear in the application to the compactified free boson in Section \ref{sec:compboson}.

For $q \in \C^\times \setminus \{-1,1\}$ let $[n]_q := \frac{q^n - q^{-n}}{q - q^{-1}}$. 

\begin{df}[{cf.\ \cite[Def.-Prop.\,6.5.1]{CP94}}] \label{df:qgrps_uncomp}
	Let $\hbar \in \C \setminus i\pi\Z$ and set $q := e^\hbar$. 
\begin{enumerate}[(i)]
	\item $\tilde U_\hbar(\lsl_2)$ is 
	the topological Hopf algebra 
		which is the quotient of $\E(h;e^+,e^-)$ 
		modulo the closure of the 
		two-sided ideal generated by the relations 
		\begin{align} 
		 	[h,e^\pm] = \pm2e^\pm \ , \qquad 
			[e^+,e^-] = \frac{e^{\hbar\cdot h} - e^{-\hbar\cdot h}}{e^\hbar - e^{-\hbar}} \ ,
		\end{align}
		together with the coproduct and antipode given by 
		\begin{gather} 
		\Delta(h) = h \ot \one + \one \ot h \ , \\
		\Delta(e^+) = e^+ \ot e^{\hbar \cdot h} + \one \ot e^+ \ , \qquad 
		\Delta(e^-) = e^- \ot \one + e^{-\hbar \cdot h} \ot e^- \ , \nonumber\\ 
		S(h) = -h \ , \qquad S(e^+) = -e^+e^{-\hbar \cdot h} \ , \qquad S(e^-) = -e^{\hbar \cdot h}e^- \ . \nonumber
		\end{gather}
	\item $\tilde U_\hbar(L\lsl_2)$ is the topological Hopf algebra 
		which is the quotient of $\E(h_0,h_1; e_0^\pm, e_1^\pm)$ modulo the 
		closure of the two-sided ideal generated by the relations 
		($i,j=0,1$) 
		\begin{gather} \label{eq:def-relns-U_a}
			h_0 = -h_1, \qquad [h_i,e_i^\pm] = \pm 2e_i^\pm \ , \qquad 
			[e_i^+,e_j^-] = \delta_{i,j}\frac{e^{\hbar\cdot h_i} - e^{-\hbar\cdot h_i}}{e^\hbar - e^{-\hbar}} \ , \\
		(e_i^+)^3e_j^+ - [3]_{q} (e_i^+)^2e_j^+e_i^+ + [3]_{q} e_i^+e_j^+(e_i^+)^2 - e_j^+(e_i^+)^3 = 0 \ , \label{eq:def-U_a-Serre-relns}\nonumber\\
		(e_i^-)^3e_j^- - [3]_{q} (e_i^-)^2e_j^-e_i^- + [3]_{q} e_i^-e_j^-(e_i^-)^2 - e_j^-(e_i^-)^3 = 0 \ , \nonumber
		\end{gather}
		together with the coproduct and antipode given by 
		\begin{gather} \label{eq:def-coalgebra-structure-U_a}
		\Delta(h_0) = h_0 \ot \one + \one \ot h_0 \ , \\
		\Delta(e_i^+) = e_i^+ \ot e^{\hbar \cdot h_i} + \one \ot e_i^+ \ , \qquad 
		\Delta(e_i^-) = e_i^- \ot \one + e^{-\hbar \cdot h_i} \ot e_i^- \ , \nonumber\\ 
		S(h_0) = -h_0 \ , \qquad S(e_i^+) = -e_i^+e^{-\hbar \cdot h_i} \ , \qquad S(e_i^-) = -e^{\hbar \cdot h_i}e_i^- \ . \nonumber
		\end{gather}
	\item 
		Denote by $\tilde U_\hbar(L\lsl_2)^+, \tilde U_\hbar(L\lsl_2)^-$ 
		the sub(-Hopf-)algebras of $\tilde U_\hbar(L\lsl_2)$ 
		generated by $e_0^+, e_1^+, h_0$ and $e_0^-,e_1^-,h_0$, respectively.
\end{enumerate}
\end{df}

   Below we will relate representations of $\tilde U_\hbar(L\lsl_2)$ to the category $\Cat_F$ discussed in the previous section. As argued in Section \ref{sec:PertDef2D_FT}, the objects in $\Cat_F$ describe non-local conserved charges, and these charges commute if the corresponding elements in the Grothendieck ring $K_0(\Cat_F)$ commute. In this context, the following theorem will be useful.

\begin{thm} \label{thm:UqCommGrothUnc}
 Let $\hbar \in \C^\times$ and assume that $q := e^\hbar$ is not a root of unity. 
 The Grothendieck ring $\Gr\big(\Repfd(\tilde U_\hbar(L\lsl_2))\big)$ of the category 
 of finite dimensional $\tilde U_\hbar(L\lsl_2)$-modules is commutative.
\end{thm}

The proof of this theorem is given in Appendix \ref{apx:B}. 
Our proof relates the statement to the known analogous statement on the Grothendieck 
ring of $U_q(L\lsl_2)$ (see Definition \ref{df:qgrps_comp} below), 
hence the restriction to the case of $q$ being not a root of unity. 

\begin{rem} \label{rem:unc-Groth-ring}
Note that the Grothendieck ring of $\Rep(\tilde U_\hbar(L\lsl_2))$ 
itself is trivial, as $\Rep(\tilde U_\hbar(L\lsl_2))$ 
contains infinite direct sums. Namely, let 
$V$ be a representation of $\tilde U_\hbar(L\lsl_2)$. Then so is $W = \bigoplus_{k=1}^\infty V$ 
and we have an exact sequence $0 \to V \to W \to W \to 0$. Hence, the 
classes $[V],[W] \in \Gr(\Rep(\tilde U_\hbar(L\lsl_2)))$ 
satisfy $[W] = [V] + [W]$, 
which implies $[V] = 0$. 
The easiest way to avoid this is to restrict to finite-dimensional representations.
\end{rem}

\subsection{The relation of $\tilde U_\hbar(L\lsl_2)$ to Yetter-Drinfeld modules for the bulk field $b$}
\label{ssec:unc_U_h-YD}

We come to the main result of this section. 
We show that the subcategory of $\Cat_F$ 
consisting of objects satisfying the commutation 
condition with $\bF_{\zeta,\xi}$ and which 
descend to modules of the Nichols algebras of 
$F^+$ and $F^-$ (see Section \ref{sec:comm_cond_YD}) 
is isomorphic to the category 
of representations of the quantum group 
$\tilde U_\hbar(L\lsl_2)$, at least for generic values of $\hbar$ (cf.\ Remark \ref{rem:Lusztig-f} below). 
The relation between the parameters is
\begin{align}
  \hbar = -\frac{\kappa\omega^2}2 \ ,
  \label{eq:uncomp-hbar-def}
\end{align}
where $\kappa$ controls the braiding on $\Cat$ via the R-matrix \eqref{eq:C-gr-vsp-Rmat} and $\omega$ defines the representation of the perturbing fields in $F$ as in \eqref{eq:uncomp-free-boson-F}; $\hbar$ is independent of $\zeta,\xi$ and the minus sign is a convention.
        
Having in mind our observation in Theorem \ref{thm:comp_eqiv_YD} 
relating the commutation condition to the 
Yetter-Drinfeld condition, 
this is not a surprising result mathematically: 
It was pointed out in \cite{Ma95b} 
that representation categories of $q$-deformed 
enveloping algebras can be obtained from 
categories of Yetter-Drinfeld modules of 
Nichols algebras in some braided categories. 
Moreover, \cite{Lus94} constructed the Borel-like 
subalgebras of $q$-deformed enveloping algebras 
by means of Nichols algebras in some braided 
categories, cf.\ \cite{Schau96}. 

\medskip

Fix $\zeta,\xi \in \C^\times$ for the remainder of this section. 
Let $V = (V,\rho)$ be an object of $\RepAss{\algA}(\algA_F)$ and write 
$(V,m) = \fcAlgPtUnc(V)$ for the corresponding object in $\Cat_F$, where  $\fcAlgPtUnc$ is 
the equivalence defined in Theorem \ref{thm:algA_F}.
The commutation condition for $(V,m)$ and $\bF_{\zeta,\xi}$ reads 
\begin{gather}
 \rho(f^+ \bar f^- - e^{-\kappa\omega^2} \bar f^- f^+) 
 	= \zeta \cdot \left( \id_V - \rho(e^{2\kappa\omega h}) \right) \ , \nonumber\\
 \rho(f^- \bar f^+ - e^{-\kappa\omega^2} \bar f^+ f^-) 
 	= \xi \cdot \left( \id_V - \rho(e^{-2\kappa\omega h}) \right) \ , \nonumber\\
 \rho(f^+ \bar f^+ - e^{\kappa\omega^2} \bar f^+ f^+) = 0 
 	= \rho(f^- \bar f^- - e^{\kappa\omega^2} \bar f^- f^-)  \ . 
 \label{eq:unc-commcond-1}
\end{gather}
For example, the first equation is obtained as follows. 
We need to satisfy \eqref{eq:compatible}, which is an equation for $\Cat$-homomorphisms 
$F^+ \ot F^- \ot V \to V$, that is, for grade preserving linear maps $F^+ \ot F^- \ot V \to V$. 
It must hold, in particular, on the vector 
$\phi^+ \ot \phi^- \ot v := (1 \oplus 0) \ot (0 \oplus 1) \ot v \in F^+ \ot F^- \ot V 
	= (\C_\omega \oplus \C_{-\omega}) \ot (\C_\omega \oplus \C_{-\omega}) \ot V$. 
On that vector we calculate for the left hand side of 
\eqref{eq:compatible}: 
\begin{align}
	&\big\{ m_\Fc \circ (\id_{\Fc} \ot m_\Fa) 
		- m_\Fa \circ (\id_{\Fa} \ot m_\Fc) \circ 
			(\varphi_{\Fc,\Fa} \ot \id_X)\big\}
		(\phi^+ \ot \phi^- \ot v) \nonumber\\
	&\quad = \rho(f^+)\rho(\bar f^-).v 
		- e^{-\kappa \omega^2}\big\{ m_\Fa \circ (\id_{\Fa} \ot m_\Fc)\big\}
			(\phi^- \ot \phi^+ \ot v) \nonumber\\
	&\quad = \rho(f^+)\rho(\bar f^-).v - e^{-\kappa \omega^2}\rho(\bar f^-)\rho(f^+).v \ , 
\end{align}
where we have used 
  $m_\Fc ((1 \oplus 0) \ot v) = \rho(f^+).v$ and 
  $m_\Fa ((0 \oplus 1) \ot v) = \rho(\bar f^-).v$, 
see the definition of $m$ in 
Theorem \ref{thm:algA_F}. 
The right-hand-side of \eqref{eq:compatible} -- evaluated on 
$\phi^+ \ot \phi^- \ot v$ -- is 
\begin{align}
	&\big\{ \bF_{\zeta,\xi} \ot \id_V 
		- (\id_V \ot \bF_{\zeta,\xi}) \circ \varphi_{\Fc \ot \Fa,V} \big\} 
			(\phi^+ \ot \phi^- \ot v) \nonumber\\
	&\quad= \zeta \cdot v 
		-  (\id_V \ot \bF_{\zeta,\xi}) 
		\big(\rho(e^{2\kappa \omega h}).v \ot \phi^+ \ot \phi^-\big)
	= \zeta \cdot \big( v - \rho(e^{2\kappa\omega h}).v \big),
\end{align}
where we used 
$\varphi_{\Fc \ot \Fa,V} = (c_{\Fc,V} \ot \id_\Fa) \circ (\id_\Fc \ot c_{V,\Fa}^{-1})$ 
and $c_{\Fc,V}(\phi^+ \ot v) = \rho(e^{\kappa \omega h}).v \ot \phi^+$, 
as well as 
$c_{V,\Fa}^{-1}(\phi^- \ot v) = \rho(e^{\kappa \omega h}).v \ot \phi^-$. 

By acting with $e^{\pm \kappa\omega h}$ on the 
conditions in \eqref{eq:unc-commcond-1} one checks that these conditions can equivalently be obtained by applying $\rho$ to the following expressions in terms of commutators:
\begin{gather} \label{eq:algAF_comm_cond}
 [f^+ e^{-\kappa\omega h}, \bar f^-] 
 	= \zeta e^{\kappa\omega^2} \cdot 
 		\left( e^{-\kappa\omega h} - e^{\kappa\omega h} \right) \ , \\
 [f^- e^{\kappa\omega h}, \bar f^+] 
 	= \xi e^{\kappa\omega^2} \cdot 
 		\left( e^{\kappa\omega h} - e^{-\kappa\omega h} \right) \ , \nonumber\\
 [f^+ e^{-\kappa\omega h}, \bar f^+] = 0 
 	= [f^- e^{\kappa\omega h}, \bar f^-] \ . \nonumber
\end{gather}
Define the bialgebra $\widetilde{\algA_F}$ as the quotient of $\algA_F$ by the relations \eqref{eq:algAF_comm_cond}. By construction, a representation 
of $\algA_F$ satisfies the commutation condition with $\bF_{\zeta,\xi}$ if and 
only if it descends to a representation of  $\widetilde{\algA_F}$.
So, pulling back representations along the canonical 
projection $\algA_F \to \widetilde{\algA_F}$, we can 
identify $\RepAss{\algA}(\widetilde{\algA_F})$ 
with the full subcategory of 
$\Cat_F \cong \RepAss{\algA}(\algA_F)$ 
consisting of objects satisfying the commutation condition 
with $\bF_{\zeta,\xi}$. 
This implies, by Theorem \ref{thm:YD-com-rel-mon}, that we 
have an isomorphism 
\begin{align}
\RepAss{\algA}(\widetilde{\algA_F}) 
	\cong \talg(F^+)\YD{\rho}\talg(F^-)
\end{align}
	of monoidal categories. 
Here, $\rho$ is the Hopf pairing $\talg(F^+) \ot \talg(F^-) \to \one$ 
induced by $\bF_{\zeta,\xi}: F^+ \ot F^- \to \one$ 
according to Lemma~\ref{lem:extHopf}. 

\medskip

We are now going to relate 
$\RepAss{\algA}(\widetilde{\algA_F})$ 
to representations of the quantum group $\tilde U_\hbar(L\lsl_2)$. Define the bialgebra homomorphism $\psi : \algA_F \to  \tilde U_\hbar(L\lsl_2)$ as
\begin{gather} \label{eq:algAF-pert-bialghomom}
 h \mapsto \frac \omega 2 h_0 \ , \qquad f^+ \mapsto e_0^+e^{-\hbar \cdot h_0} \ , \qquad f^- \mapsto \xi \frac{e^\hbar - e^{-\hbar}}{e^{2\hbar}} e_1^+e^{-\hbar \cdot h_1} \ , \\
 \bar f^+ \mapsto e_1^- \ , \qquad \bar f^- \mapsto \zeta \frac{e^\hbar - e^{-\hbar}}{e^{2\hbar}} e_0^- \ . \nonumber
\end{gather}
It is easily checked that $\psi$ is compatible with the algebra and 
coalgebra structures of $\algA_F$. For example, 
\begin{align}
\psi([h,f^+]) &= \frac\omega 2 [h_0,e_0^+ e^{-\hbar \cdot h_0}] 
\stackrel{\eqref{eq:def-relns-U_a}}= \omega e_0^+ e^{-\hbar \cdot h_0} = \psi(\omega f^+) \ , \\
(\psi\ot\psi)\Delta(f^+) &\stackrel{\eqref{eq:def-coproduct-A_F}}= 
	(\psi\ot\psi)(f^+ \ot \one + e^{\kappa\omega h} \ot f^+) 
	= e_0^+e^{-\hbar \cdot h_0} \ot \one 
		+ e^{\frac{\kappa\omega^2}2 h_0} \ot e_0^+e^{-\hbar \cdot h_0} \nonumber\\
&= (e_0^+ \ot e^{\hbar \cdot h_0} + \one \ot e_0^+)
		(e^{-\hbar \cdot h_0} \ot e^{-\hbar \cdot h_0}) 
	\stackrel{\eqref{eq:def-coalgebra-structure-U_a}}= 
		\Delta(e_0^+)\Delta(e^{-\hbar \cdot h_0})
	= \Delta\psi(f^+) \ . \nonumber
\end{align} 
Moreover, $\psi$ respects the relations \eqref{eq:algAF_comm_cond}, e.g. 
\begin{align}
	&\psi([f^+ e^{-\kappa\omega h}, \bar f^-]) \stackrel{\eqref{eq:algAF-pert-bialghomom}}{=} 
		\zeta \frac{e^\hbar - e^{-\hbar}}{e^{2\hbar}} 
		[e_0^+ e^{-\hbar h_0} e^{\hbar h_0}, e_0^-] 
	\stackrel{\eqref{eq:def-relns-U_a}}{=} 
		\zeta e^{-2\hbar} \left(e^{\hbar h_0} - e^{-\hbar h_0}\right) \nonumber \\
	&\quad= \zeta e^{\kappa\omega^2} 
		\psi\big(e^{-\kappa \omega h} - e^{\kappa \omega h}\big) \ .
\end{align}
This results in a surjective bialgebra homomorphism 
$\widetilde{\algA_F} \to \tilde U_\hbar(L\lsl_2)$. Consequently, the pullback functor 
\begin{align}
 \F_{\zeta,\xi} \,:\, \RepAss{\langle h_0 \rangle}(\tilde U_\hbar(L\lsl_2)) 
 	~\longrightarrow~ \RepAss{\algA}(\widetilde{\algA_F}) \cong \Cat_F \ ,
	\label{eq:F-ze-xi_def}
\end{align}
where $\langle h_0 \rangle$ denotes the (topological) subalgebra of 
$\tilde U_\hbar(L\lsl_2)$ generated by $h_0$,
is full, faithful, exact and monoidal. Altogether we proved:

\begin{thm}
\label{thm:exrepcatfuncmon_uncpt1}
The functor $\F_{\zeta,\xi}$ identifies the monoidal category
	$\RepAss{\algA}(\tilde U_\hbar(L\lsl_2))$ 
	with a monoidal full subcategory of 
	$\talg(F^+)\YD{\rho}\talg(F^-)$. 
\end{thm}

In the case that $e^{\hbar} = e^{-\kappa\omega^2/2}$ is transcendental, we can describe the image of $\F_{\zeta,\xi}$ more explicitly in terms of Nichols algebras as introduced in the end of Section \ref{sec:comm_cond_YD}. This is explained in the next remark.

\begin{rem} \label{rem:Lusztig-f}
Following \cite{Lus94}, 
consider the field of fractions $\Q(v)$ for a formal variable $v$ and
let $\mathbf{'f}$ be the free associative algebra over $\Q(v)$ 
with 
generators $\theta_0,\theta_1$. By assigning degree $(1,0)$ to 
$\theta_0$ and degree $(0,1)$ to $\theta_1$, the algebra $\mathbf{'f}$ 
may be considered as an algebra object in the braided category $\Cat'$ of 
$\Z \times \Z$-graded vector spaces over $\Q(v)$. The braiding on two one-dimensional vector spaces of degrees $(m_1,m_2)$ and $(n_1,n_2)$ is multiplication by the element 
$v^{2(m_1-m_2)(n_1-n_2)}$ of $\Q(v)$. The Nichols algebra $\mathbf{f}$ of $\mathbf{'f}$ 
is given by the quotient of $\mathbf{'f}$ by the radical 
of the Hopf pairing described in \cite[Prop.\,1.2.3]{Lus94} 
and Lemma \ref{lem:extHopf}, see \cite[Ex.\,3.1]{Schau96}. 
It was shown in \cite[Thm.\,33.1.3(a)]{Lus94} that 
the Serre relations (the counterparts for $\mathbf{'f}$ of 
\eqref{eq:def-relns-U_a}), 
generate the radical of this pairing. 
If $e^{\kappa \omega^2/2}$ is transcendental there is a well-defined embedding $\Q(v) \to \C$
sending $v$ to $e^{\kappa \omega^2/2}$. Comparing the braiding on $\Cat'$ to \eqref{eq:braiding-on-simples} shows that we get
an exact, faithful (but non-full) braided monoidal functor from $\Cat'$ into our category 
$\Cat = \Repss(\algA)$ taking the simple object $\Q(v)_{(n_1,n_2)} \in \Cat'$ 
to the simple object $\C_{(n_1-n_2)\omega} \in \Cat$ and taking a morphism 
acting by multiplication with $f(v) \in \Q(v)$ on 
a simple object of $\Cat'$ to a morphism acting by 
multiplication with $f(e^{\kappa\omega^2/2}) \in \C$ on the corresponding simple 
object of $\Cat$. 
It takes $\mathbf{'f}$ to $\talg(F^+)$ and the ideal generated 
by the Serre relations in $\mathbf{'f}$ to the one generated 
by the Serre relations in \eqref{eq:def-relns-U_a} 
for 
$q = e^{\kappa \omega^2/2}$ (which are the same as those for $q^{-1}$). 
By exactness, the functor takes $\mathbf{f}$ to the Nichols algebra
$\B(F^+)$. Analogously, $\B(F^-)$ is the quotient 
of $\talg(F^-)$ by the ideal generated by the Serre relations. 
Altogether, the functor 
$\F_{\zeta,\xi}$ identifies $\RepAss{\algA}(\tilde U_\hbar(L\lsl_2))$ 
with the full subcategory $\B(F^+)\YD{\rho}\B(F^-)$ of 
$\talg(F^+)\YD{\rho}\talg(F^-)$.
\end{rem}

\subsection{Example: Properties of T and Q operators}
\label{ssec:uncomp_TQ-op}

Set $q = e^\hbar$ and denote by $V_n$, $n \in \N$, the representation of 
$\tilde U_\hbar(\lsl_2)$ where $h,e^+,e^-$ act on the vector space $\C^{n+1}$ 
by the matrices (see e.g.\ \cite[Ch. XVII.4]{Kas95}) 
\begin{gather} 
\rho(h){=}\left(\begin{smallmatrix} n & & & & \\ & n{-}2 & & & \\ & & \ddots & & \\ & & & {-}n{+}2 & \\ & & & & -n \end{smallmatrix}\right), \, 
\rho(e^+){=}\left(\begin{smallmatrix} 0 & [n]_q & & & \\ & 0 & [n{-}1]_q & & \\ & & \ddots & \ddots & \\ & & & 0 & 1 \\ & & & & 0 \end{smallmatrix}\right), \, 
\rho(e^-){=}\left(\begin{smallmatrix} 0 & & & & \\ 1 & 0 & & & \\ & \ddots & \ddots & & \\ & & [n{-}1]_q & 0 & \\ & & & [n]_q & 0 \end{smallmatrix}\right). 
\end{gather}
For $z \in \C^\times$, the evaluation homomorphism 
$\ev_z: \tilde U_\hbar(L\lsl_2) \to \tilde U_\hbar(\lsl_2)$ is given by (see, e.g.\ \cite{CP91}) 
\begin{align}
	\ev_z(e_0^\pm) = e^{\mp \hbar} z^{\pm 1} e^\mp, \quad
	\ev_z(e_1^\pm) = e^\pm, \quad
	\ev_z(h_0) = -h, \quad
	\ev_z(h_1) = h.
\end{align}
For an $\tilde U_\hbar(\lsl_2)$-module $V$, denote by $V(z)$ its 
pullback along the evaluation homomorphism $\ev_z$.

\begin{rem}
In Remark \ref{rem:comments-on-comm-cond}(ii) we saw that objects satisfying the commutation condition come in $\C^\times$-families. In the present example, this parameter is related to evaluation homomorphisms as follows. Recall the functor $\F_{\zeta,\xi}: \RepAss{\algA}(\tilde U_\hbar(L\lsl_2)) 
 	\to \RepAss{\algA}(\algA_F)$ from \eqref{eq:F-ze-xi_def}, where the constants $\zeta,\xi \in \C^\times$ give the bulk field morphism $b$ via \eqref{eq:uncom-b_ze-xi_def}.
Let $(X,m)$ be $\F_{\zeta,\xi}(V(1))$, considered as an object 
in $\Cat_F$ via Theorem \ref{thm:algA_F}, for some evaluation representation $V(1)$ of 
$\tilde U_\hbar(L\lsl_2)$. Then $(X,m \circ ((w\cdot\id) \oplus (w^{-1}\cdot\id)))$, 
$w \in\C^\times$, 
is isomorphic to $\F_{\zeta,\xi}(V(w^2))$. 
\end{rem}

For $q$ not a root of unity, $z \in \C^\times$ and a positive integer $n$, there is an exact sequence of representations \cite[Sect.\,4.9]{CP91} 
\begin{align} \label{eq:unc-T-exseq-1}
	0 \,\longrightarrow\, V_{n-1}(q^{n+2} z) \,\longrightarrow\, V_1(z) \ot V_n(q^{n+1}z) 
	\,\longrightarrow\, V_{n+1}(q^{n}z) \,\longrightarrow\, 0 \ .  
\end{align}
For $n \in \N$, $z \in \C^\times$, let 
$T_n(z)$ denote the class of $\F_{\zeta,\xi}(V_n(z))$ in 
the Grothendieck ring 
$\Gr\left( \Repfd(\algA_F) \right)$. 
By Theorems \ref{thm:UqCommGrothUnc} and \ref{thm:exrepcatfuncmon_uncpt1}, 
the elements $T_m(z)$ and $T_n(w)$ commute for each $m,n \in\N$ and $z,w \in \C^\times$. 
From 
\eqref{eq:unc-T-exseq-1} 
one obtains the relations 
\begin{align} \label{eq:T-Rel-1}
	T_1(z) T_n(q^{n+1}z) &= T_{n-1}(q^{n+2}z) 
		+ T_{n+1}(q^{n}z)~~, \quad z \in \C^\times \ . 
\end{align}
This reproduces \cite[Eqn.\,(4.13)]{BLZ99}. 
If the operators ${\cal T}_n(z) := \Op(\F_{\zeta,\xi}(V_n(z)))$ described in 
Section \ref{ssec:pert-def-and-cat} exist, then by the arguments in 
Section \ref{ssec:op+K0} they 
satisfy the same relations as the $T_n(z)$ do in the Grothendieck ring; 
in particular they mutually commute. These are the so-called T-operators.

\medskip

For $z \in \C^\times$, $m\in\Z$, and $q$ not a 
root of unity, denote by $Q_{+,m}(z)$ 
the representation of $\tilde U_\hbar(L\lsl_2)^+$ 
with basis $v_j$, $j \in -\N$, and action 
\begin{align}
	e_1^+.v_j = v_{j-1} \ , \qquad 
	e_0^+.v_j = z\frac{1-e^{2j \hbar}}{(e^\hbar - e^{-\hbar})^2} v_{j+1} \ , \qquad
	h_0.v_j = (2j+m) \cdot v_j \ .
\end{align}
These representations are pullbacks 
of $q$-oscillator representations, see Eqns.\ (3.8), (3.9), (D.6) of \cite{BLZ99}, \cite[Sect.\,2.3]{RW02}, 
and \cite[Sect.\,3.3.5]{Boos12}, and they were used in \cite{BLZ97a, BLZ99} to construct Q-operators. 
We have exact sequences 
\begin{align} 
	&0 \longrightarrow Q_{+,m-1}(zq^2) \longrightarrow Q_{+,m}(z) \ot V_1(z) 
	\longrightarrow Q_{+,m+1}(zq^{-2}) \longrightarrow 0 \ ,
	\nonumber \\[.5em]
&0 \longrightarrow Q_{+,m+1}(zq^{-2}) \longrightarrow V_1(z) \ot Q_{+,m}(z) 
	\longrightarrow Q_{+,m-1}(zq^{2}) \longrightarrow 0 \ . \label{eq:unc-Q-ex-seq} 
\end{align}
These may be obtained from \cite[Prop.\,2.2]{RW02} 
noting that $Q_{+,m}(z)$ is a subrepresentation of the ``$U_\hbar$ analogue''
of $M(z,q^{-m},\frac{-z}{(e^\hbar - e^{-\hbar})^2},0)$ 
as in \cite[Def.\,2.1]{RW02}. 
Furthermore, $Q_{+,m}(z) \ot V_1(w)$ and 
$V_1(w) \ot Q_{+,m}(z)$ are isomorphic for 
$z \neq w$ (this can be seen by specialising the universal R-matrix 
as in \cite[Sect.\,3.3]{RW02} or \cite[Sect.\,4.1]{Boos12}, or by direct calculation). 

Since $Q_{+,m}(z)$ is an infinite-dimensional representation, it does not give an element in
$\Gr\big( \Repfd(\algA_F) \big)$. There is no point in passing to the Grothendieck ring of all $\algA_F$-modules since that is trivial (Remark \ref{rem:unc-Groth-ring}).
But if the operator $\mathcal Q_{+,m}(z) := \Op(Q_{+,m}(z))$ 
can be defined, one expects from 
\eqref{eq:unc-Q-ex-seq}
and from the arguments in Section \ref{ssec:op+K0} that
\begin{align}
	\mathcal Q_{+,m}(z) \mathcal T_1(z) &= 
	\mathcal Q_{+,m-1}(zq^2) + \mathcal Q_{+,m+1}(zq^{-2}) 
	= \mathcal T_1(z) \mathcal Q_{+,m}(z) \ .
\end{align}
This reproduces the T-Q relation \cite[Eqn.\,(4.2)]{BLZ99}.\footnote{ The parameter $m$ does not appear in \cite[Eqn.\,(4.2)]{BLZ99} since there the Q-operators are defined in a way which is independent of the particular representation chosen.}

\begin{rem}\label{rem:T-vs-Q-uncomp}
We stress that the T-operators discussed above arise from representations of the {\em full} quantum group $\tilde U_\hbar(L\lsl_2)$ and therefore via Theorem \ref{thm:exrepcatfuncmon_uncpt1} provide objects in $\talg(F^+)\YD{\rho}\talg(F^-)$. In terms of perturbed defects as discussed in Section \ref{ssec:PertDefNonlocalConsCharges}, this means that -- assuming existence -- the operators ${\cal T}_n(z)$ will commute with the perturbed Hamiltonian. The Q-operators, on the other hand, are only representations of the Borel half $\tilde U_\hbar(L\lsl_2)^+$ and therefore do a priori {\em not} give non-local conserved charges of the perturbed theory. Indeed, it is not clear to us how to define Q-operators in the perturbed uncompactified free boson, i.e.\ in sin(h)-Gordon theory, in our formalism. (However, Q-operators are known in a lattice discretisation of sin(h)-Gordon theory, see \cite{NT09}.) 
\end{rem}

\section{Compactified free boson}
\label{sec:compboson}

The theory for a free boson with a circle as target space 
is more complicated than the 
uncompactified case since we need to introduce 
a nontrivial algebra $A$ and to describe the category $\CatD$ in terms of $A|A$-bimodules, see Section \ref{ssec:pert-def-and-cat}. 
It is possible to carry over results from the uncompactified to the compactified boson 
(see Proposition \ref{prop:fcCardyToAlg} below), 
but we would like to focus on a new feature: the appearance of the quantum group $U_q(L\lsl_2)$, rather than $\tilde U_\hbar(L\lsl_2)$ which we encountered in Section \ref{sec:uncompboson}.
The difference between $\tilde U_\hbar(L\lsl_2)$ and $U_q(L\lsl_2)$ is that the former is formulated via a generator $h$ where the latter involves $k^{\pm1}$. The two quantum groups are related by $q = e^{\hbar}$ and $k = e^{\hbar h}$. For $q$ a root of unity, $U_q(L\lsl_2)$ has so-called cyclic representations; these are absent for $\tilde U_\hbar(L\lsl_2)$, irrespective of the value of $\hbar$. The main results of this section are:
\begin{itemize}
\item Cyclic representations lead to non-local conserved charges which behave similarly to the Q-operators in the uncompactified boson. 
\item If $q$ is an even root of unity, T- and Q-operators do not always commute.
\end{itemize}

\subsection{Input data: algebra, bimodules and the bulk field $b$}\label{ssec:comp-input-data}

In Remark \ref{rem:top-def-in-RCFT} we stated that a rational CFT is described by a choice of a rational vertex operator algebra $\mathcal{V}$ and of a special symmetric Frobenius algebra $A$ in $\Cat = \mathrm{Rep}(\mathcal{V})$. The defect category in this case is $\CatD \cong A\text{-Mod$_\Cat$-}A$. The Heisenberg vertex operator algebra underlying the free boson is not rational, but we will assume that the categorical description of perturbed defects works analogously. 

\medskip

In Section \ref{ssec:unc_input-data} we introduced the topological Hopf algebra $\algA$, the simple $\algA$-modules $\simp{\alpha}$ and the canonical identifications $i_{\alpha,\beta} : \simp{\alpha} \ot \simp{\beta} \stackrel{\sim}{\to} \simp{\alpha+\beta}$. In view of Remark \ref{rem:unc-boson--cat-VOA} we set 
\begin{align}
  \Cat = \Repss(\algA) \ . 
\end{align}
We fix a nonzero complex number $r \in \C$ and define an algebra $A_r$ in $\Cat$ as follows. As an object in $\Cat$,
\begin{align} A_r := \bigoplus_{n \in \Z} \simp{nr} \ . \end{align}
The multiplication $\mu: A_r \ot A_r \to A_r$ and 
unit map $\eta: \one \to A_r$ are given by 
$\eta := \iota_0$ and $\mu := \sum_{m,n \in \Z} \iota_{m+n} \circ i_{mr,nr} \circ (\pi_m \ot \pi_n)$,
where $\pi_n: A_r \to \simp{nr}$ and $\iota_n: \simp{nr} \to A_r$ 
are the canonical projection and inclusion maps. As a difference to the rational CFT setting we point out that the algebra $A_r$ cannot be Frobenius since the underlying object is not self-dual. We define
\begin{align}
  \CatD = A_r\text{-Mod$_\Cat$-}A_r \ .
\end{align}
According to our above assumption on the similarity of the compactified free boson and rational CFT, the category $\CatD$ is interpreted as describing topological defects of the compactified free boson. (Of course, the algebraic considerations in this section are not affected by the validity of this interpretation.)
It is easy to describe $A_r|A_r$-bimodules explicitly. This is done in the following remark.

\begin{rem} \label{rem:Ar-bimods}
	Let $(W, \rho)$ be an object of $\Cat$ 
	with representation $\rho: \algA \to \End_\Vect(W)$ and let 
	$ B \in \End( W)$ be invertible and such that 
	$ B$ commutes with $\rho(\fxh)$. From this data we obtain an $A_r|A_r$-bimodule $X(W,B)$ as follows.
	The underlying object of $\Cat$ is
	\begin{align}
		X(W,B) = \textstyle \big(\bigoplus_{n \in \Z} W, \ 
			\bigoplus_{n \in \Z} (\gamma_{nr} \, \id_{W} + \rho)\big) 
	\end{align} 
	of $\Cat$, where $\gamma_{nr}$ is the algebra homomorphism $\algA \to \C$ sending $\fxh$ to $nr \in \C$. 
	Let $s: \bigoplus_{n \in \Z} W \to \bigoplus_{n \in \Z} W$ 
	be the shift $s\left( \oplus_{n \in \Z} w_n \right) := \oplus_{n \in \Z} w_{n-1}$. 
	The left and right action of $A_r$ on $X(W,B)$ are
	\begin{align}
m_l(1_{mr} \ot \id_{W_\algA}) := s^m \quad , \qquad
m_r(\id_{W_\algA} \ot 1_{mr}) := \big( \oplus_{n \in \Z} B^m \big) \circ s^m \ .
\label{eq:Ar-action-X(b,x)}
\end{align}
	It is not hard to see that any $A_r|A_r$-bimodule is isomorphic to one of these. 
	The simple bimodules are of the form $X(\C_\beta,\xi\cdot\id_{\C_\beta})$ 
	for  $\beta \in \C$ and $\xi \in \C^\times$ 
	and we write $X(\beta,\xi) := X(\C_\beta,\xi\cdot\id_{\C_\beta})$, for short. 
	We have
\begin{align} \label{eq:XX-isos}
X(\beta,\xi) \,\cong\, X(\beta',\xi')
\quad \Leftrightarrow \quad \beta - \beta' \in r\Z ~~\text{and}~~ \xi = \xi' \ .
\end{align}
	Furthermore
	\begin{align}\label{eq:X-bimod-tensor}
	X(\beta,\xi) \otA X(\beta',\xi') \cong X(\beta+\beta', \xi\xi') \ .
	\end{align}
    From this we can conclude that the bimodules $X(\beta,\xi)$ are group-like and their isomorphism classes are parametrised by the multiplicative group $\C^\times \times \C^\times$, in agreement with \cite[Sect.\,4.1]{Fuchs07} (there, an additional automorphism is allowed, resulting in an extra $\Z_2$).
Recall the functor $\alpha : \DZ(\Cat) \to \DZ(A\BMod[\Cat] A)$ from Proposition \ref{prop:alphaind} and the resulting two functors $\alpha^\pm : \Cat \to \DZ(A\BMod[\Cat] A)$ given in Remark \ref{rem:alpha-frs-relation}. We have
\begin{align}\label{eq:alpha+-for-Ar}
  U\alpha^\pm(\C_\beta) \cong X(\beta, e^{\pm \kappa r\beta}) \ ,
\end{align}
where $U$ is the forgetful functor $\DZ(\Cat)\to\Cat$.
\end{rem}

With the detailed understanding of simple $A_r|A_r$-bimodules we can now verify that the prescription \eqref{eq:Hom_VV-iso-Hom_AA}, originally formulated for rational CFTs, correctly recovers the state space of a compactified free boson:

\begin{rem} \label{eq:comp-bulkspace-check}
The space of bulk states $\mathcal{F}$ is a representation of two copies of the Heisenberg vertex operator algebra. To compute the multiplicity $M_{p,q}$ of the highest weight representation $H_p \otimes_{\C} H_{q}$ (cf.\ Remark \ref{rem:unc-boson--cat-VOA}) in $\mathcal{F}$, the prescription \eqref{eq:Hom_VV-iso-Hom_AA} instructs us to evaluate $M_{p,q} = \dim\Hom_{A_r|A_r}(\alpha^+(\C_p) \ot_A \alpha^-(\C_q), A_r)$. 
By \eqref{eq:X-bimod-tensor} and \eqref{eq:alpha+-for-Ar}, we have 
$\alpha^+(\C_p) \ot_A \alpha^-(\C_q) \cong X(p+q, e^{\kappa r(p-q)})$ 
and from this we read off 
\begin{align} \label{eq:Mpq-free-boson}
  M_{p,q} = \begin{cases} 1 & \text{if } p+q \in r\Z \text{ and } p-q \in 
  \frac{2\pi i}{\kappa r}\Z \\
		0 & \text{else}. \end{cases} 
 \end{align}
According to Remark \ref{rem:unc-boson--cat-VOA}, we should choose the constant $\kappa$ defining the braiding as $\kappa = i \pi$. We obtain the conditions $p+q \in r\Z$ and $p-q \in \tfrac{2}{r} \Z$, which -- for real $r$ and in appropriate conventions -- describe the standard charge lattice of a free boson compactified on a circle of radius $r$.
\end{rem}

Having described the space of bulk states $\mathcal{F}$, we now fix the representation of the perturbing field $F$ and the bulk field morphism $b$. We choose $F$ to be the same as in the uncompactified case in the sense that
\begin{align}
  F := F^+ \oplus F^- = \alpha(F_\mathrm{unc}) 
  \qquad , ~~ \text{where}~~ F^\pm = \alpha^\pm(F_\omega)  \ .
  \label{eq:comp-free-boson-F}
\end{align}
Here, $F_\omega = \C_\omega \oplus \C_{-\omega}$ as in Section \ref{ssec:unc_input-data} and $F_\mathrm{unc}$ is the choice for $F$ in the uncompactified boson, cf.\ \eqref{eq:uncomp-free-boson-F}. 
To specify the bulk field morphism $b : F^+ \otA F^- \to A_r$, first note that by \eqref{eq:X-bimod-tensor} and \eqref{eq:alpha+-for-Ar} we have $F^+ \otA F^- \cong 
  X(-2\omega,0) \oplus X(0,e^{-2\kappa r \omega}) \oplus X(0,e^{2\kappa r \omega}) \oplus X(2\omega,0)$. 
We would like to restrict our attention to $b$'s that are non-vanishing only on pairs of ``opposite charge'', that is, we want $b$ to factor through $X(0,e^{\pm 2\kappa r \omega})$. 
For there to be a nonzero 
bimodule map  $X(0,e^{\pm2\kappa r \omega}) \to A_r$ at all, by \eqref{eq:XX-isos} we must have
\begin{align}\label{eq:omega-via-t}
   \omega = \tfrac{\pi i}{\kappa r} \cdot t 
   \qquad , ~~ t \in \Z \ ,
\end{align}
and we will demand $t \neq 0$ to have a nontrivial perturbation. For $\omega$ of the form \eqref{eq:omega-via-t}, the objects $A_r$ and $F_\omega$ are {\em mutually transparent} in $\Cat$, i.e.\ they satisfy
\begin{align}
  c_{F_\omega,A_r} = c_{A_r,F_\omega}^{-1} \ .
\end{align}
Because of this, $U\alpha^+(F_\omega) = U\alpha^-(F_\omega)$ as $A_r|A_r$-bimodules in $\Cat$. After this preparation, we can finally specify our choice for the bulk field morphism $b$:
\begin{align}
  b_{\zeta,\xi} := 
  \Big[~ &
  U\alpha^+(F_\omega) \ot_{A_r} U\alpha^-(F_\omega)
  =
  U\alpha^+(F_\omega) \ot_{A_r} U\alpha^+(F_\omega)
\nonumber \\
  & \xrightarrow{~\sim~}
  U\alpha^+(F_\omega \ot F_\omega)
  \xrightarrow{U\alpha^+((b_\mathrm{unc})_{\zeta,\xi})}
  U\alpha^+(\one)
  =
  A_r ~\Big] \ ,
  \label{eq:com-b_ze-xi_def}
\end{align}
where $(b_\mathrm{unc})_{\zeta,\xi}$ is the choice \eqref{eq:uncom-b_ze-xi_def} of $b$ in the uncompactified case. Below, we will skip the initial isomorphism and just write $b_{\zeta,\xi} = U\alpha^+((b_\mathrm{unc})_{\zeta,\xi})$.

\begin{rem}
The parameter $t$ has a straightforward interpretation in the sin(h)-Gordon model discussed in Remark \ref{eq:sin(h)-Gorden-uncomp}. The field $\varphi^\text{full}(x)$ is now periodic with period $2 \pi r$ (cf.\ Remark \ref{eq:comp-bulkspace-check}). The perturbing term ${:}\cos(\omega \varphi^\text{full}(x)){:}$ must respect this period, and this requires $\omega \in \frac{1}{r} \Z$, in agreement with \eqref{eq:omega-via-t} for $\kappa=i\pi$. 
\end{rem}

Since we have chosen $F$ and $b_{\zeta,\xi}$ for the compactified boson in the same way as for the uncompactified boson, it is possible to use the functor $\alpha$ to transport results from $\Cat_{F_{\mathrm{unc}}}$ to $\CatD_F$. This is made precise by the following proposition whose proof can be found in Appendix \ref{apx:fcCardyToAlg}. We will not elaborate further on this and instead move on to prepare the ground for new features not present in the uncompactified boson. 

\begin{prop} \label{prop:fcCardyToAlg}
Let $\Cat$ be an abelian braided monoidal category with 
right-exact tensor product, let $A$ be an algebra in $\Cat$ and write $\CatD = A\BMod[\Cat] A$. 
Fix objects $F^+_\Cat, F^-_\Cat \in \Cat$ such that $F^-_\Cat$ and $A$ are mutually transparent.
\begin{enumerate}
\item
Set $F_{\DZ\Cat} := \iota^+(F^+_\Cat) \oplus \iota^-(F^-_\Cat)$ and 
 $F := \alpha^+(F^+_\Cat) \oplus \alpha^-(F^-_\Cat)$. 
The functor $\Cat_{F_{\DZ\Cat}} \rightarrow \CatD_F$ given by
 \begin{gather} \label{eq:C-to-AAbim}
	(X,m) \mapsto (U\aind^+(X), U\aind^+(m)) \quad , \qquad 
	f \mapsto U\aind^+(f) 
 \end{gather}
 is monoidal. 
\item
 If $b_\Cat: F^+_\Cat \ot F^-_\Cat \to \one$ is a morphism in $\Cat$, 
 then the functor \eqref{eq:C-to-AAbim} maps objects 
 satisfying the commutation condition 
 with $b_\Cat$ to objects 
 satisfying the commutation condition with 
 $U\alpha^+(b_\Cat): U\alpha^+(F^+_\Cat) \ot_A U\alpha^-(F^-_\Cat) \to A$. 
\item
 If the functor $A \ot (-): \Cat \to \Cat$ 
 is exact and $F^\pm_{\Cat} \ot (-): \Cat \to \Cat$ 
 are right-exact, then the functor \eqref{eq:C-to-AAbim} is exact. 
\end{enumerate}
\end{prop}

\subsection{A subcategory of the category for perturbed defects}
\label{ssec:comp-bosonisation}

In the following we will restrict our attention to a full subcategory of $\CatD = A_r\BMod[\Cat]A_r$ which is semisimple and consists of direct sums of simple objects isomorphic to $X(\beta,1)$ for some $\beta \in \C$. We will write
\begin{align}
  \CatD_\mathrm{ss}^1 := \big\langle X(\beta,1) \big\rangle_{\beta \in \C} ~\subset \CatD \ .
\end{align}
The reason for this is two-fold: firstly, this subcategory is sufficient to describe the $Q$-like operators arising from cyclic representations of $U_q(L\lsl_2)$, and secondly, the associativity isomorphisms can be trivialised on $\CatD_\mathrm{ss}^1$ but not on $\CatD$ (or even on the full semisimple subcategory generated by all $X(\beta,\xi)$). The last point is explained in detail in Appendix \ref{apx:fcAB}. In order for $\alpha^\pm(\C_\omega) \cong X(\omega, e^{\pm\kappa\omega r}) = X(\omega, e^{\pm \pi i t})$ 
to be contained in $\CatD_\mathrm{ss}^1$ we will require that
\begin{align}
  t \in 2\Z \ .
  \label{eq:com-fb-t-even}
\end{align}

Continuing along the same lines as in Section \ref{sec:uncompboson}, we will now express $\CatD_\mathrm{ss}^1$ as the (semisimple part of the) representation category of an appropriate bialgebra $\algB$ over $\C$; eventually this will play the role 
of the Cartan-like subalgebra of $U_q(L\lsl_2)$.
The bialgebra $\algB$ is generated as a unital algebra by elements $k,k^{-1}$ with relations
\begin{align}
 kk^{-1} = 1 = k^{-1}k \ .
\end{align}
It is equipped with the coproduct, counit and antipode determined by 
\begin{gather}
 \Delta(k) = k \ot k \ , \qquad \Delta(k^{-1}) = k^{-1} \ot k^{-1} \ , \nonumber\\
 \varepsilon(k) = 1 \ , \qquad S(k) = k^{-1} \ .  
\end{gather}
Denote by $\Repss(\algB)$ the category of semisimple $\algB$-modules, that is, modules on which $k$ acts diagonalisably. We will define a functor
\begin{align}
  \fcAB = \Big[ \, \Repss(\algB) \xrightarrow{~\fcAB^0~} \Cat \xrightarrow{~\iota^0~} 
  \DZ(\Cat) \xrightarrow{~\alpha~} \DZ(\CatD)  \xrightarrow{~U~} \CatD \,\Big] \ ,
  \label{eq:RepB-2-AmodA-def}
\end{align}
where the constituents $\fcAB^0$ and $\iota^0$ will be described momentarily, and we will show that $\fcAB$ is a monoidal equivalence onto the subcategory $\CatD_\mathrm{ss}^1$.

Let $\mathrm{mod}_r: \C \to \C$ (``remainder mod $r$'') be the map 
which sends numbers in the class $z + r\Z$ 
to the unique representative $z' \in z + r\Z$ 
for which the real part of $\frac {z'}r$ 
is contained in the interval $[0,1)$, i.e.\ $z'$ is 
in the strip $r\cdot ([0,1) \times i\R)$. 
For some choice of logarithm $\ell(z)$ of $z \in \C^\times$ 
define $\mathcal{L} : \C^\times \to \C$ by 
\begin{align} \label{eq:comp_Psi}
	\mathcal{L}(z) := \mathrm{mod}_r\left( r \cdot \frac{\ell(w)}{2\pi i} \right) \ . 
\end{align}
For another choice of logarithm of $z$, $\widetilde\ell(z) = \ell(z) + 2\pi in$ for 
some $n \in \Z$, we obtain 
\begin{align} 
	\tilde{\mathcal{L}}(z) &= \mathrm{mod}_r\left( r \cdot \frac{\widetilde\ell(z)}{2\pi i} \right) 
	= \mathrm{mod}_r\left( r\cdot\frac{\ell(z) + 2\pi i n}{2\pi i} \right) 
	= \mathcal{L}(z) \ ,
\end{align}
so that $\mathcal{L}$ is independent of that choice. 

The (non-monoidal) functor $\fcAB^0 : \Repss(\algB) \to \Cat$ is defined as follows.
A representation $(V,\rho_\algBim)$ in $\Repss(\algB)$ is a direct sum of  
eigenspaces $V_z = \ker(z\cdot\id - \rho_\algB(k))$ for eigenvalues $z \in \C^\times$ of $\rho_\algB(k)$. Define an $\algA$-action $\rho_\algA$ on $V$ by 
\begin{align} \label{eq:def-rho-A}
	\rho_\algA(\fxh).v = \mathcal{L}(z) \cdot v 
	\qquad \text{for} \qquad v \in V_z \ .
\end{align} 
Then $\fcAB^0$ sends $(V,\rho_\algB)$ to $(V,\rho_\algA)$ and acts
as the identity on hom-sets. 
Next, let $\iota^0 : \Cat \to \DZ(\Cat)$ be the functor 
\begin{align}
	V \mapsto (V,\tau_{V,-}) \ , \qquad f \mapsto f \ ,
\end{align}
where $\tau_{V,W}: V \ot W \to W \ot V$ is the tensor flip, 
$\tau_{V,W} (v \ot w) = w \ot v$. This completes the definition of the functor $\fcAB$ in \eqref{eq:RepB-2-AmodA-def}

\begin{rem} \label{rem:image-of-bimod-functor}
The functor $U\circ\alpha\circ\iota^0$ maps a simple object 
$\C_\beta \in \Cat$ to $X(\beta,1) \in A_r\BMod[\Cat]A_r$. 
Let $\C_z$ be the simple object in $\Repss(\algB)$ with $k$ acting 
on $\C$ by $k.1 = z$, and let $\beta$ be in the image of $\mathrm{mod}_r$ and satisfy $z = \exp(\frac{2\pi i}r \beta)$. Then $\fcAB(\C_z) = X(\beta,1)$.
\end{rem}

It turns out that the composition $\mathcal{I}$ in \eqref{eq:RepB-2-AmodA-def} is monoidal, even though $\fcAB^0$ 
is not: 

\begin{prop} \label{prop:fcAB}
	The functor $\fcAB : \Repss(\algB) \to \CatD_\mathrm{ss}^1$ is a monoidal equivalence.
\end{prop}

For the proof of this proposition we refer to Appendix \ref{apx:fcAB}. 
It is easy to see that $\fcAB$ is an equivalence of categories and that
$\fcAB(X \ot Y) \cong \fcAB(X) \otA \fcAB(Y)$; the slightly lengthy part of the proof is to show that the associator can be trivialised on $\CatD_\mathrm{ss}^1$.

\medskip

Our next aim is to find an object $\ptfE \in \DZ(\Repss(\algB))$ corresponding to $F \in \DZ(\CatD)$, together with a full monoidal embedding $\Repss(\algB)_\ptfE \to \CatD_F$. For the remainder of this section we fix 
\begin{align}\label{eq:def-of-q}
q := \exp\!\big(\tfrac{\pi i}r \omega\big) 
	\stackrel{\eqref{eq:omega-via-t}}= \exp\!\big(\tfrac{\kappa\omega^2}t\big) \ . 
\end{align}
The object $\ptfE := (\tilde \ptfE, \varphi_{\ptfE,-})$  
in $\DZ(\Repss(\algB))$ consists of the $\algB$-module 
\begin{align}
 \tilde \ptfE := \C_{q^{2}} \oplus \C_{q^{-2}} \oplus \C_{q^{2}} \oplus \C_{q^{-2}} \ , 
\end{align}
	where $\C_{q^{\pm 2}}$ is as in Remark \ref{rem:image-of-bimod-functor}, together with
the half-braiding 
\begin{align}
	\varphi_{\tilde \ptfE,(V,\rho)} := \tau_{\tilde E,V} \circ \big( & \phantom{\oplus}
		(\id \ot \rho(k^{t/2})) \circ (p_1 \ot \id_V) \oplus 
		(\id \ot \rho(k^{-t/2})) \circ (p_2 \ot \id_V) \nonumber\\
	&\oplus (\id \ot \rho(k^{-t/2})) \circ (p_3 \ot \id_V) \oplus 
		(\id \ot \rho(k^{t/2})) \circ (p_4 \ot \id_V) \big) \ ,
 \label{eq:halfbraiding_F0}
 \end{align}
	where $p_1,\ldots,p_4$ are the canonical projections 
	$\tilde E \to \C_{q^{\pm 2}}$. The powers of $\pm t/2$ make sense because in \eqref{eq:com-fb-t-even} we assumed $t$ to be even. Now note that since $\mathcal{L}(q^{2}) = \mathrm{mod}_r(\omega)$, by \eqref{eq:XX-isos} there is an isomorphism
\begin{align}
  g: \fcAB(\tilde E) \xrightarrow{~\sim~} X(\omega,1) \oplus X(-\omega,1) \oplus X(\omega,1) \oplus X(-\omega,1) = UF 
  \label{eq:iso-g-E-UF}
\end{align}
of $A_r|A_r$-bimodules.
Let us fix a choice for the isomorphism $g$ for the rest of this section.

\begin{lem} \label{lem:FFnot}
The diagram 
	\begin{align} \label{dg:lem-FFnot}\xymatrixcolsep{4pc} 
		\xymatrix{ UF \ot_{A_r} \fcAB(X) \ar[r]^-{g^{-1} \ot_{A_r} \id} \ar[d]^{\varphi_{F,\fcAB(X)}} 
			&\fcAB(U\ptfE) \ot_{A_r} \fcAB(X) = \fcAB(U\ptfE \ot X) \ar[d]^{\fcAB(\varphi_{\ptfE,X})} \\
		\fcAB(X) \ot_{A_r} UF \ar[r]^-{\id \ot_{A_r} g^{-1}} 
			&\fcAB(X) \ot_{A_r} \fcAB(U\ptfE) = \fcAB(X \ot U\ptfE) } 
	\end{align} 
	commutes for every  $X\in\Repss(\algB)$.
\end{lem}
\begin{proof}
 It suffices to show commutativity of \eqref{dg:lem-FFnot} in the 
 case of $X$ being a simple $\algB$-module. 
 So let $X = \C_z$ for $z = \exp(\frac{2\pi i}r \beta)$
 as in Remark \ref{rem:image-of-bimod-functor}. 
 Then $\varphi_{\ptfE,X}$ is -- up to a tensor flip -- 
 multiplication with the matrix 
 \begin{align}
	\left(\begin{smallmatrix} z^{\wzconst/2} & & & \\ & z^{-\wzconst/2} & & \\
			& & z^{-\wzconst/2} & \\ & & & z^{\wzconst/2} \end{smallmatrix}\right)
\label{eq:I-monoidal-aux}
 \end{align}
 on $U\ptfE \ot X \cong \C^4 \ot \C \cong \C^4$ (in $\Vect$). 
 On the other hand, to describe $\varphi_{F,\fcAB(X)}$, 
 let us restrict to the $U\aind^+(\C_\omega)$-summand of $UF$; the others 
 are treated similarly. Recall from Proposition \ref{prop:alphaind} that $\aind: \DZ(\Cat) \to \DZ(A_r\BMod[\Cat]A_r)$ is braided monoidal. Now $\aind^+(\C_\omega) = \aind(\C_{\omega},c)$ and $\fcAB(X) = U \aind(\C_{\beta},\tau)$, so that by definition of the braiding on $\DZ(\Cat)$, the following diagram commutes:
\begin{align} 
  \xymatrix{
 \aind^+(\C_\omega) \ot_{A_r} \fcAB(X) \ar[rr]^\sim \ar[d]^{\varphi_{\alpha^+(\C_\omega),\fcAB(X)}} &&	
   \aind\big((\C_{\omega},c) \ot (\C_{\beta},\tau)\big) \ar[d]^{c_{\C_{\omega},\C_{\beta}}}
\\
   \fcAB(X) \ot_{A_r} \aind^+(\C_\omega)  \ar[rr]^\sim  &&	
   \aind\big((\C_{\beta},\tau) \ot (\C_{\omega},c)\big) 
  }
\end{align} 
But $c_{\C_{\omega},\C_{\beta}} = e^{\kappa \omega \beta} \cdot \tau_{\C_{\omega},\C_{\beta}}$ and $e^{\kappa \omega \beta} \overset{\eqref{eq:omega-via-t}}=  \exp(\frac{\pi i t}r \beta)  = z^{\wzconst/2}$,  matching the first entry in \eqref{eq:I-monoidal-aux}.
\end{proof}

\begin{prop}\label{prop:functor-Rep(algB)_E-to-D_F}
The functor $\fcAB$ and the isomorphism $g$
 induce a full, faithful, exact monoidal functor 
 $\fcAB_g : \Repss(\algB)_{\ptfE} \to \CatD_F$. 
\end{prop}
\begin{proof}
The functor $\fcAB_g$ sends an object $(X,m)$ of $\Repss(\algB)_{\ptfE}$ to 
 \begin{align} 
 	(\fcAB(X),\, \fcAB(m) \circ (g^{-1} \ot \id)) 
 	\in (A_r\BMod[\Cat]A_r)_F \end{align} 
and a morphism $f$ to $\fcAB(f)$. This functor is monoidal by Lemma \ref{lem:FFnot}. It is exact 
 since a complex in $(A_r\BMod[\Cat]A_r)_F$ is exact if and only if the 
 underlying complex in $A_r\BMod[\Cat]A_r$ is exact (cf.\ Lemma \ref{lem:exact_iff}) 
 and since $\fcAB$ is an exact functor. 
\end{proof}

We can now ``bosonise'' $\Repss(\algB)_{\ptfE}$ in very much the 
same way as in Section \ref{ssec:unc_bosonisation}. 
Define $\algPtBim$ as the 
bialgebra over $\C$ which, as a unital algebra, is the extension of $\algB$ by 
$f^+, f^-, \bar f^+, \bar f^-$ with relations ($q$ as in \eqref{eq:def-of-q}) 
\begin{align} \label{eq:comm_rels_algPtBim}
 k f^\pm k^{-1} = q^{\pm 2} f^\pm \ , \qquad k \bar f^\pm k^{-1} = q^{\pm 2} \bar f^\pm \ ,
\end{align}
together with the coproduct on $k,k^{-1}$ inherited from $\algB$, and 
for the other generators: 
\begin{gather} \label{eq:comp_algPtBim-coprod}
 \Delta(f^\pm) = f^\pm \ot \one + k^{\pm \wzconst/2} \ot f^\pm \ , \qquad 
 \Delta(\bar f^\pm) = \bar f^\pm \ot \one + k^{\mp \wzconst/2} \ot \bar f^\pm \ .
\end{gather}
Then construct a functor 
$\RepAss{\algB}(\algPtBim) \to \Repss(\algB)_{\ptfE}$ 
 as follows: If $(V,\rho)$ is an object of 
$\RepAss{\algB}(\algPtBim)$, then 
let $m: \ptfE \ot V \to V$ be the map 
\begin{align} \label{eq:comp--def-of-m}
m(\phi \ot v) := \rho(\phi^+ \cdot f^+ + \phi^- \cdot f^- 
		+ \bar \phi^+ \cdot \bar f^+ + \bar \phi^- \cdot \bar f^-).v \ ,
\end{align} 
where 
$\phi = \phi^+ \oplus \phi^- \oplus \bar \phi^+ \oplus \bar \phi^- 
	\in \alpha^+(\C_\omega \oplus \C_{-\omega}) \oplus \alpha^-(\C_\omega \oplus \C_{-\omega}) 
	= F$.
The functor $\RepAss{\algB}(\algPtBim) \to \Repss(\algB)_{\ptfE}$ 
maps $(V,\rho)$ to $((V,\rho|_\algBim),m)$ and acts 
as the identity on hom-sets. 
It is in fact an isomorphism, as is 
seen in very much the same way as in Theorem~\ref{thm:algA_F}. 
Altogether, we obtain:

\begin{prop} \label{prop:fcAlgPtBim}
	The categories $\RepAss{\algB}(\algPtBim)$ and 
	$\Repss(\algB)_{\ptfE}$ are strictly isomorphic 
	as monoidal categories. 
\end{prop}

\subsection{Conventions for $U_q(\lsl_2)$ and $U_q(L\lsl_2)$}
\label{ssec:comp-uq}

In Section \ref{sec:uncompboson} we considered quantum groups 
$\tilde U_\hbar(\lsl_2)$ and $\tilde U_\hbar(L\lsl_2)$ whose 
Cartan-like subalgebras were generated by a primitive element $h$. 
Now we use the standard versions 
of the deformed enveloping algebras of $\lsl_2$ and 
$L\lsl_2$ with a group-like generator $k$, which satisfies the 
relations of $e^{\hbar h}$ in $\tilde U_\hbar(\lsl_2)$ 
and $\tilde U_\hbar(L\lsl_2)$, respectively, if $q = e^\hbar$. 

\begin{df} \label{df:qgrps_comp} (cf.\ \cite{CP91})\footnote{ \cite{CP91} 
	define $U_q(\hat\lsl_2)$, while $U_q(L\lsl_2)$, as in our definition, is 
	the quotient of $U_q(\hat\lsl_2)$ by the Hopf ideal generated by 
	$k_0 - k_1^{-1}$}
	Let $q \in \C^\times \setminus \{-1,1\}$. 
\begin{enumerate}[(i)]
	\item Define $U_q(\lsl_2)$ as the Hopf algebra, generated as an associative unital algebra by elements 
		$e^\pm$, $k$, $k^{-1}$ subject to the relations 
		\begin{align} kk^{-1} = 1 = k^{-1}k \ , \qquad ke^\pm k^{-1} = q^{\pm2}e^\pm \ , 
			\qquad [e^+,e^-] = \frac{k - k^{-1}}{q - q^{-1}} \ , \end{align}
		together with the coproduct and antipode given by 
		\begin{gather} \Delta(e^+) = e^+ \ot k + \one \ot e^+ \ , \qquad 
		 \Delta(e^-) = e^- \ot \one + k^{-1} \ot e^- \ , \\
		 \Delta(k) = k \ot k \ , \qquad \Delta(k^{-1}) = k^{-1} \ot k^{-1} \ , \nonumber\\
		 S(k) = k^{-1} \ , \qquad S(k^{-1}) = k \ , \qquad S(e^+) = -e^+k^{-1} \ , \qquad S(e^-) = -ke^- \ . \nonumber
		\end{gather}
	\item Define $U_q(L\lsl_2)$ as the Hopf algebra, 
		generated as an associative unital algebra 
		by elements $e_0^\pm$, $e_1^\pm$ and $k = k_0 = k_1^{-1}$, 
		$k^{-1} = k_0^{-1} = k_1$ subject 
		to the relations ($i,j=0,1$) 
		\begin{gather} kk^{-1} = 1 = k^{-1}k \ , \qquad k_ie_i^\pm k_i^{-1} = q^{\pm2}e_i^\pm \ , \qquad 
			[e_i^+,e_j^-] = \delta_{i,j}\frac{k_i - k_i^{-1}}{q - q^{-1}} \ , \\
		 (e_i^+)^3e_j^+ - [3]_{q} (e_i^+)^2e_j^+e_i^+ + [3]_{q} e_i^+e_j^+(e_i^+)^2 - e_j^+(e_i^+)^3 = 0 \ , \nonumber\\
		 (e_i^-)^3e_j^- - [3]_{q} (e_i^-)^2e_j^-e_i^- + [3]_{q} e_i^-e_j^-(e_i^-)^2 - e_j^-(e_i^-)^3 = 0 \ , \nonumber
		\end{gather}
		together with the coproduct and antipode given by 
		\begin{gather} \Delta(e_i^+) = e_i^+ \ot k_i + \one \ot e_i^+ \ , \qquad 
		\Delta(e_i^-) = e_i^- \ot \one + k_i^{-1} \ot e_i^- \ , \\
		 \Delta(k) = k \ot k \ , \qquad \Delta(k^{-1}) = k^{-1} \ot k^{-1} \ , \nonumber\\
		 S(k) = k^{-1} \ , \qquad S(k^{-1}) = k \ , \qquad S(e_i^+) = -e_i^+k_i^{-1} \ , 
			\qquad S(e_i^-) = -k_ie_i^- \ . \nonumber \end{gather}
		Denote by $U_q(L\lsl_2)^+, U_q(L\lsl_2)^-$ the sub(-Hopf-)algebras of $U_q(L\lsl_2)$ 
		generated by $e_0^+, e_1^+, k, k^{-1}$ and $e_0^-,e_1^-,k,k^{-1}$, respectively.
\end{enumerate}
\end{df}

For $q$ not a root of unity the Grothendieck ring $\Gr(\Repfd(U_q(L\lsl_2)))$ of the 
category of finite dimensional representations of $U_q(L\lsl_2)$ is commutative \cite[Corollary 2]{FR98}. 
Otherwise, if $q$ is a root of unity, this may fail, 
see Section \ref{ssec:comp-TQ_cycl} below.

\subsection{Commutation condition in $\CatD_F$ and representations of $U_q(L\lsl_2)$}
\label{ssec:comp-functor-from-uq}

Recall the choice of the bulk field morphism $b_{\zeta,\xi}$ made in \eqref{eq:com-b_ze-xi_def}. We would like to formulate the commutation condition of $\CatD_F$ for $b_{\zeta,\xi}$ in $\Repss(\algB)_{\ptfE}$ as defined in Section~\ref{ssec:comp-bosonisation}.
To this end write $\tilde E = \tilde E^+ \oplus \tilde E^-$ with $\tilde E^+,\tilde E^- := \C_{q^{2}} \oplus \C_{q^{-2}}$ and denote by $j_{z,z'} : \C_z \ot \C_{z'} \stackrel\sim\to \C_{zz'}$ for the canonical identifications in $\Repss(\algB)$ determined 
by $1 \ot 1 \to 1$. Set 
\begin{align}
 d_{\zeta,\xi} : \tilde\ptfE^+ \ot \tilde\ptfE^- \to \one \ , 
 \qquad d_{\zeta,\xi} := \zeta \cdot j_{q^{2},q^{-2}} + \xi \cdot j_{q^{-2},q^{2}} \ .
\end{align}
Then $\fcAB(d_{\zeta,\xi}) = b_{\zeta,\xi} \circ (g|_{\tilde\ptfE^+} \ot g|_{\tilde\ptfE^-})$, where 
$\fcAB$ was defined in \eqref{eq:RepB-2-AmodA-def} and $g$ is the isomorphism from \eqref{eq:iso-g-E-UF}.

If $(V,\rho)$ is an object of $\RepAss{\algB}(\algPtBim)$, then the commutation 
condition (with respect to $d_{\zeta,\xi}$) for its image in $\Repss(\algB)_{\ptfE}$ under the isomorphism 
of Proposition \ref{prop:fcAlgPtBim} reads 
\begin{gather}
 \rho(f^+ \bar f^- - q^{-t} \bar f^- f^+) = \zeta \cdot \left( \id_V - \rho(k^{t}) \right) \ , \\
 \rho(f^- \bar f^+ - q^{-t} \bar f^+ f^-) = \xi \cdot \left( \id_V - \rho(k^{-t}) \right) \ , \nonumber\\
 \rho(f^+ \bar f^+ - q^{t} \bar f^+ f^+) = 0 = \rho(f^- \bar f^- - q^{t} \bar f^- f^-) \ . \nonumber
\end{gather}
This can be seen in a similar fashion as in Section \ref{ssec:unc_U_h-YD}.
Multiplying with $k^{\pm t/2}$, these relations may equivalently 
be written as 
\begin{gather} \label{eq:algBF_comm_cond}
 \rho([f^+ k^{-t/2}, \bar f^-]) = \zeta q^{t} \cdot \rho\left( k^{-t/2} - k^{t/2} \right) \ , \\
 \rho([f^- k^{t/2}, \bar f^+]) = \xi  q^{t} \cdot \rho\left( k^{t/2} - k^{-t/2} \right) \ , \nonumber\\
 \rho([f^+ k^{-t/2}, \bar f^+]) = 0 = \rho([f^- k^{t/2}, \bar f^-]) \ . \nonumber
\end{gather}
Comparing to Definition \ref{df:qgrps_comp}\,(ii) this shows that for $t=-2$ there is a surjective bialgebra homomorphism $\psi : \algPtBim \to U_q(L\lsl_2)$ 
given by 
\begin{gather}\label{eq:homom-to-U_q--compactified-case}
 k \mapsto k_0 \ , \qquad f^+ \mapsto e_0^+k_0^{-1} \ , 
 \qquad f^- \mapsto \xi \frac{q - q^{-1}}{q^2} e_1^+k_1^{-1} \ , \nonumber\\
	\bar f^+ \mapsto e_1^- \ , 
	\qquad \bar f^- \mapsto \zeta \frac{q - q^{-1}}{q^2} e_0^-  \ . 
\end{gather}
As in Section \ref{ssec:unc_U_h-YD}, one checks that $\psi$ 
respects the relations \eqref{eq:algBF_comm_cond}. 
Let 
$\F_{\zeta,\xi}$ be the functor from $\RepAss{\langle k \rangle}(U_q(L\lsl_2))$ to $\RepAss{\algB}(\algPtBim)$ pulling back representations of $U_q(L\lsl_2)$ along $\psi$.
Here, $\langle k \rangle$ stands for the sub(-Hopf-)algebra of 
$U_q(L\lsl_2)$ generated by $k$. 

\medskip

Denote by $\CatD_F'$ the full subcategory of $\CatD_F$ 
satisfying the commutation condition with $b_{\zeta,\xi}$.
Combining the functor $\F_{\zeta,\xi}$ with previous results, we obtain:

\begin{thm}
\label{thm:exrepcatfuncmon_cpt1}
Suppose the parameter $t$ in \eqref{eq:omega-via-t} is equal to $-2$. 
	The monoidal functor 
\begin{align}
\RepAss{\langle k \rangle}(U_q(L\lsl_2))
&\xrightarrow{\F_{\zeta,\xi}}
\RepAss{\algB}(\algPtBim)
\xrightarrow[\text{\rm Prop.\,\ref{prop:fcAlgPtBim}}]{\sim}
\Repss(\algB)_{\ptfE}
\nonumber\\
&
\xrightarrow[\text{\rm Prop.\,\ref{prop:functor-Rep(algB)_E-to-D_F}}]{\fcAB_g}
\CatD_F'
\xrightarrow[\text{\rm Thm.\,\ref{thm:YD-com-rel-mon}}]{\sim}
\talg(F^+)\YD{\rho}\talg(F^-)
\end{align}
identifies 
	$\RepAss{\langle k \rangle}(U_q(L\lsl_2))$ 
	with a full subcategory of 
	$\talg(F^+)\YD{\rho}\talg(F^-)$. 
\end{thm}

\subsection{Example: Cyclic representations}
\label{ssec:comp-TQ_cycl}

By \eqref{eq:def-of-q} 
and \eqref{eq:omega-via-t}, $q$ is a root of unity iff $\frac{\pi i}{\kappa r^2} \in \Q$. In this case,
$U_q(L\lsl_2)$ allows for so-called cyclic representations. 

\begin{rem}
According to Remark \ref{rem:unc-boson--cat-VOA}, when relating the results in Sections \ref{sec:uncompboson} and \ref{sec:compboson} to the free boson, i.e.\ when choosing the braided category $\Cat$ to be representations of the Heisenberg VOA, we should take $\kappa = i \pi$. Then the above condition on the existence of cyclic representations requires $r^2 \in \Q$. The conformal weights of the fields in $A_r$ are of the form $\tfrac12 m^2 r^2 \, (\mathrm{mod}\,\Z)$ for $m \in \Z$. This contains fields of integer weight if and only if $r^2 \in \Q$. In these cases the compactified boson has an extended chiral symmetry and becomes a rational CFT. Thus the cyclic representations occur precisely for the rational free boson.
\end{rem}

Let us now describe the cyclic representations, following \cite{Kor03}.\footnote{ 
To write down results from \cite{Kor03} in our conventions, 
we fix the bialgebra isomorphism 
from our $U_q(\lsl_2)$ to Korff's $U_q(\lsl_2)$ sending 
$e^+$ to $f$, $e^-$ to $e$ and $k$ to $k^{-1}$; 
and we fix the bialgebra isomorphism from our $U_q(L\lsl_2)$ to Korff's $U_q(\tilde\lsl_2)$ 
determined by $e_j^+ \mapsto f_j$, $e_j^- \mapsto e_j$ and $k_j \mapsto k_j^{-1}$ 
($j=0,1$).} 
If $q$ is not a root of unity, the centre of 
$U_q(\lsl_2)$ is generated by the quadratic 
Casimir $\mb c := qk^{-1} + q^{-1}k + (q - q^{-1})^2 e^+e^-$.
For $q$ a root of unity, as we assume here, the centre is enlarged by the additional generators 
\begin{align}
	\mb x := \left( (q - q^{-1})e^- \right)^{N'} \ , \qquad 
	\mb y := \left( (q - q^{-1})e^+ \right)^{N'} \ , \qquad 
	\mb z^{\pm 1} := k^{\mp N'} \ , 
\end{align}
where $N'$ is the order of $q^2$ as a root of unity. 
An irreducible representation is called \emph{semi-cyclic} if $\mb x$ or $\mb y$ act by nonzero endomorphisms and it is called \emph{cyclic} if both do. 
A cyclic or semi-cyclic irreducible representation $V^p$ of $U_q(\lsl_2)$ is 
-- up to isomorphism -- determined by the eigenvalues
$p = (x,y,z,c)$ of $\mathbf x, \mathbf y, \mathbf z, \mathbf c$. 
For $w \in \C^\times$, 
denote by $V^p(w)$ the evaluation representation of $V^p$, 
i.e.\ the pullback of $V^p$ along the evaluation homomorphism 
$\ev_w: U_q(L\lsl_2) \to U_q(\lsl_2)$, defined in a 
similar way as in Section \ref{ssec:uncomp_TQ-op}.\footnote{\label{fn:Korff-ev-param} Note 
that the evaluation homomorphism $\ev_w$ in Section \ref{ssec:uncomp_TQ-op} differs 
from the one in \cite[Sect.\,2.3]{Kor03}. If we denote the latter by $\ev^{\text{Kor}}_w$ then 
$\ev_w = \ev^{\text{Kor}}_{q/w}$. 
This implies in particular that Korff's $\mu$, the quotient 
of the evaluation parameters for a tensor product $V^p(w) \ot V_1(w')$, equals 
$w/{w'}$ in our conventions.} 
Let $V_1$ denote the fundamental representation of 
$U_q(\lsl_2)$ analogously to Section \ref{ssec:uncomp_TQ-op}. 
If $w'/w + w/w' \neq c$ and 
$q$ is an odd root of unity, then we have 
$V^p(w) \ot V_1(w') \cong V_1(w') \ot V^p(w)$.
If $w'/w + w/w' \neq c$ and 
$q$ is an even root of unity, then there is no 
intertwiner $V^p(w) \ot V_1(w') \to V_1(w') \ot V^p(w)$ 
at all. 

\begin{rem}
Denote by $\Repfd(-)$ the finite dimensional representations of a given algebra.
One can check that $V^p(w) \ot V_1(w')$ and 
$V_1(w') \ot V^p(w)$ are irreducible in case of 
$w'/w + w/w' \neq c$. 
Since the Grothendieck ring $\Gr(\Repfd(U_q(L\lsl_2)))$ 
is freely generated, as an abelian group, 
by the classes of simple objects (Jordan-H\"older theorem), 
the absence of an intertwiner implies that $\Gr(\Repfd(U_q(L\lsl_2)))$ 
is not commutative for $q$ an even root of unity.
By Theorem~\ref{thm:exrepcatfuncmon_cpt1}, this also implies that 
the subcategory of dualisable Yetter-Drinfeld modules in $\talg(F^+)\YD{\rho}\talg(F^-)$ 
has a non-commutative 
Grothendieck ring. 
Note, however, that (semi-)cyclic representations $V^p(w)$ of 
$U_q(L\lsl_2)$ are not representations of the Nichols algebras 
of $F^\pm$: By definition or by \cite[Thm.\,2.9]{Schau96}, the Nichols algebra 
of $F^+$ is the quotient of $\talg(F^+)$ by the kernel of 
the braided symmetriser $\Sym$, as in Lemma \ref{lem:extHopf}. We have 
$\Sym|_{\talg(\C_{q^2})} = \sum_{n \in \N} q^{-\frac{n(n-1)}2}[n]_q!\cdot\iota_n\circ\pi_n$, 
where $[n]_q! := [1]_q[2]_q \cdots [n]_q$. 
In the root of unity case, $[n]_q = 0$ for $n \geq N'$. 
Hence, representations of $U_q(\lsl_2)$ where $(e^-)^{N'}$ acts 
nonzero cannot correspond to representations of the Nichols algebra of 
$F^+$. 
\end{rem}

If $w'/w + w/w' = c$, then the tensor products become 
decomposable and there are exact sequences 
\begin{align} \label{eq:UqSeq}
	\begin{CD} 0 @>>> V^{p_1}(q w) @>>> 
	V^p(w) \ot V_1(w') 
	@>>> V^{p_2}(q^{-1} w) @>>> 0 \ , \end{CD} \nonumber \\[.5em]
	\begin{CD} 0 @>>> V^{p_3}(q^{-1} w) @>>> 
	V_1(w') \ot V^p(w) 
	@>>> V^{p_4}(q w) @>>> 0 \ , \end{CD} \end{align}
where 
\begin{align}
	p_1 &= \big(x \,,\, q^{N'}y\,,\, q^{N'}z\,,\, q \tfrac{w}{w'} + q^{-1}\tfrac{w'}{w} \big) &,~~~~ 
	p_2 &= \big(x\,,\, q^{N'}y\,,\, q^{N'}z\,,\, q^{-1} \tfrac{w}{w'} + q \tfrac{w'}{w} \big) \ , \nonumber\\
	p_3 &= \big(q^{N'}x \,,\, y \,,\, q^{N'}z \,,\, q^{-1} \tfrac{w}{w'} + q \tfrac{w'}{w} \big) &,~~~~
	p_4 &= \big(q^{N'}x \,,\, y \,,\, q^{N'}z \,,\, q \tfrac{w}{w'} + q^{-1} \tfrac{w'}{w} \big) \ .
\end{align} 
The first sequence is a rewriting of the sequence given in \cite[Sect.\,4]{Kor03}. 
The second sequence may be obtained by dualising the first one.\footnote{ 
To do so, use \cite[Prop.\,4.5]{CP91} and 
$S(\mathbf x) = -\mathbf{xz}^{-1}$, $S(\mathbf y) = -\mathbf{yz}$, 
$S(\mathbf z) = \mathbf{z}^{-1}$ and $S(\mathbf c) = \mathbf{c}$, 
where $S$ is the antipode in $U_q(\lsl_2)$. 
Hence, $V^p(w)^* \cong V^{p'}(q^2w)$, if 
$p = (x,y,z,c)$ and $p' = (-x/z,-yz,z^{-1},c)$. 
Also, $V^p(w) \cong V^{p'}(q^{-2}w)^*$.}

Denote by $T_1(w')$ the class of $V_1(w')$ in the 
Grothendieck ring of $\Repfd(U_q(L\lsl_2))$, and 
by $Q_p(w)$ the class of $V^p(w)$. We distinguish two cases:
\begin{itemize}
\item 
$w'/w + w/w' \neq c$ :
If $q$ is an odd root of unity, $T_1(w')$ and $Q_p(w)$ commute, and if $q$ is an even root of unity they do not commute. 
\item $w'/w + w/w' = c$ : One obtains the $T$-$Q$-like relation 
\begin{align}
	Q_p(w)T_1(w') = Q_{p_1}(q w) + Q_{p_2}(q^{-1} w) 
		= T_1(w') Q_{p_5}(w) \ , 
\end{align} 
where $p_5 = (q^{N'}x\,,\,q^{N'}y\,,\,q^{N'}z\,,\,c)$. 
If $q$ is an odd root of unity then 
$p_5 = p$ because $q^{N'} = 1$, and $T_1(w')$ and $Q_p(w)$ commute.
If $q$ is an even root of unity, then 
$T_1(w')$ and $Q_p(w)$ do not commute, but 
``pulling $T_1$ past $Q_p$'' flips the signs of the 
parameters $x,y$ and $z$.
\end{itemize}
This concludes our discussion of cyclic representations. We have seen two new features which arise for the free boson compactified on a circle of rational radius squared. Firstly, non-local conserved charges constructed from quantum group representations no longer necessarily commute, and secondly, there is a {\em finite-dimensional} representation of the {\em full} quantum group (rather than an infinite-dimensional representation of a Borel half, cf.\ Remark \ref{rem:T-vs-Q-uncomp}) which behaves similar to a Q-operator. Being representations of the full quantum group, via Theorem \ref{thm:exrepcatfuncmon_cpt1} and the construction in Section \ref{ssec:PertDefNonlocalConsCharges}, such cyclic representations provide conserved charges for the perturbed compactified free boson.

\appendix

\section{Appendix}

\subsection{Algebraic structures in monoidal categories: standard definitions}
\label{appx:alg-etc-in-braided-cats}

In order to remind the reader 
and to fix the notation used in the paper 
we collect some standard definitions such as the 
notion of monoidal centre, algebra, Hopf algebra, module, bimodule, ... 
in a (braided) monoidal category. 

\medskip

A monoidal category is given by the data $(\Cat,\ot,\alpha,\one,\lambda,\rho)$, 
where $\Cat$ is some category, $\ot: \Cat \times \Cat \to \Cat$, 
$\one$ is the tensor unit of $\Cat$ and $\alpha$, $\lambda$ and $\rho$ are 
natural isomorphisms, $\alpha_{X,Y,Z}: (X \ot Y) \ot Z \to X \ot (Y \ot Z)$ 
(called associators), 
$\lambda_X: \one \ot X \to X$, $\rho_X: X \ot \one \to X$ (called 
unit isomorphisms), for 
$X,Y,Z \in \Cat$, subject to the usual axioms (see, e.g.\ \cite[Ch\,1.1]{BaKi00}).
If $\Cat$ is an abelian monoidal category with exact tensor product, then 
$\Gr(\Cat)$ denotes its Grothendieck ring (see, e.g.\ \cite[Ch\,2.1]{BaKi00}). 

The monoidal centre of a monoidal category is defined as follows.

\begin{figure} 
 \centering
	\begin{align*}
	 &\text{(a)} \quad
	 \raisebox{-0.5\height}{\includegraphics[height=0.18\textwidth]{img/figDZ_a.pdf}} 
	 \ := \ \varphi_{V,X}
	 &\text{(b)} \quad
	 \raisebox{-0.5\height}{\includegraphics[height=0.18\textwidth]{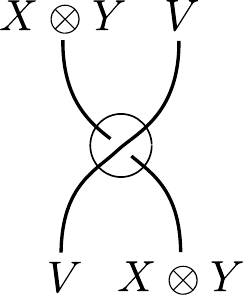}} 
	 = \raisebox{-0.5\height}{\includegraphics[height=0.18\textwidth]{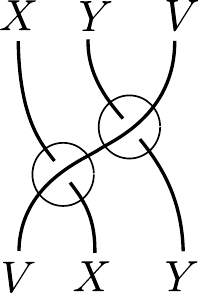}} \\&& \\
	 &\text{(c)} \quad 
	 \raisebox{-0.5\height}{\includegraphics[height=0.18\textwidth]{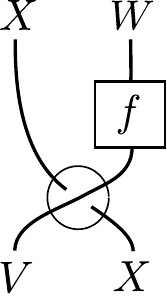}} 
	 \ = \ \raisebox{-0.5\height}{\includegraphics[height=0.18\textwidth]{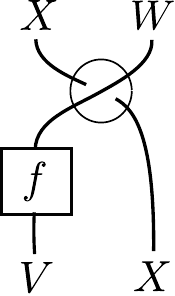}} 
	 &\text{(d)} \quad 
	 \raisebox{-0.5\height}{\includegraphics[height=0.18\textwidth]{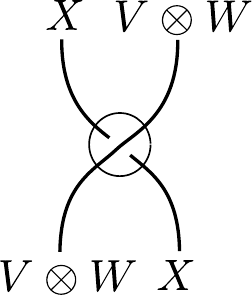}} 
	 = \raisebox{-0.5\height}{\includegraphics[height=0.18\textwidth]{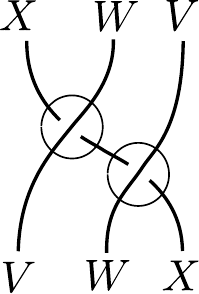}} 
	\end{align*} 
 \caption{The Drinfeld centre or monoidal centre: 
 	a) Graphical notation for the half braiding (recall footnote \ref{fn:graphics}); 
 	b)--d) axioms for the half braiding, see Definition \ref{df:DZ}.}
 \label{fig:DZ}
\end{figure}

\begin{df} \label{df:DZ} (cf.\ \cite{JS91, Ma95})
 Let $(\CatD,\ot,\one,\alpha,\lambda,\rho)$ be a monoidal category. 
 The \emph{(Drinfeld or monoidal) centre} $\DZ(\CatD)$ of $\CatD$ is the category given by 
\begin{itemize}
 \item objects: pairs $(V,\varphi_{V,-})$ where $V$ is an object of $\CatD$ and 
	$\varphi_{V,-}$ (the \emph{half-braiding}) is a natural isomorphism $V \ot (-) \to (-) \ot V$ such that 
	for all $X,Y \in \CatD$ we have (see Figure \ref{fig:DZ} (b))
	\begin{align} 
	 	\alpha_{X,Y,V} \circ \varphi_{V,X \ot Y} \circ \alpha_{V,X,Y} 
		= (\id_X \ot \varphi_{V,Y}) \circ \alpha_{X,V,Y} \circ (\varphi_{V,X} \ot \id_Y).
	\end{align}
 \item morphisms: a morphism from $(V,\varphi_{V,-})$ to $(W,\varphi_{W,-})$ is a morphism 
	$f: V \to W$ in $\CatD$ such that for each object $X \in \CatD$ we have 
	(see Figure \ref{fig:DZ} (c))
	\begin{align} \label{eqn:DZmorph}
	 	(\id_X \ot f) \circ \varphi_{V,X} = \varphi_{W,X} \circ (f \ot \id_X).
	\end{align}
\end{itemize}
\end{df}

 Define a tensor product $\ot$ on $\DZ(\CatD)$ by $(V,\varphi_{V,-}) \ot (W,\varphi_{W,-}) := (V \ot W, \varphi_{V \ot W,-})$, 
 where $\varphi_{V \ot W,-}$ is given for objects $X \in \CatD$ by (see Figure \ref{fig:DZ} (d)) 
 \begin{align} \varphi_{V \ot W,X} := \alpha_{X,V,W} \circ (\varphi_{V,X} \ot \id_W) 
	\circ \alpha_{V,X,W}^{-1} \circ (\id_V \ot \varphi_{W,X}) \circ \alpha_{V,W,X}.
 \end{align}
 Define a braiding $c$ on $\DZ(\CatD)$ by $c_{(V,\varphi_{V,-}),(W,\varphi_{W,-})} := \varphi_{V,W}$. 
This turns $\DZ(\CatD)$ into a braided monoidal category \cite{Ma95}.
Note that the forgetful functor 
\begin{align} \label{eq:forgetful}
	U : \DZ(\CatD) \to \CatD, \quad (V,\varphi_{V,-}) \to V, \quad f \mapsto f
\end{align}
is strictly monoidal.

\medskip

Next we define (co-/Bi-/Hopf-)algebras, modules and bimodules within 
a monoidal category. 
We refer to e.g.\ \cite{Ma95,Be95} for more details; our notation follows \cite{FRS02}.
The tensor algebras $\talg(F)$ used in Sections \ref{sec:tensor_algebras} 
and \ref{sec:comm_cond_YD} are Hopf algebras, and the Frobenius algebras $A$ in 
rational CFT (see e.g.\ Remark \ref{rem:top-def-in-RCFT}) are algebras 
and coalgebras. Bimodules are needed 
in rational CFT for the description of topological defects 
(see Remark \ref{rem:top-def-in-RCFT}). 
They show up, in particular, in our treatment of the 
compactified free boson in Section \ref{sec:compboson}. 

\begin{figure} 
 \centering
 \begin{align*}
 \begin{array}{cccc}
 	\text{(a)} 
	&\raisebox{-0.5\height}{\includegraphics[scale=.5]{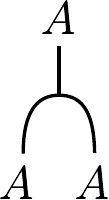}} 
 		:= \mu, \qquad 
	\raisebox{-0.5\height}{\includegraphics[scale=.5]{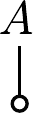}} 
 		\ := \eta, 
 	&\text{(b)} 
	&\raisebox{-0.5\height}{\includegraphics[scale=.5]{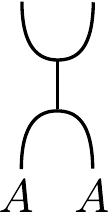}} 
	 	\ = \ \raisebox{-0.5\height}{\includegraphics[scale=.5]{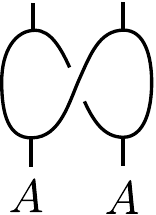}} \\ \\
	&\raisebox{-0.5\height}{\includegraphics[scale=.5]{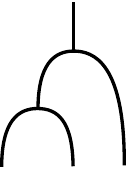}} 
	 	\ = \ \raisebox{-0.5\height}{\includegraphics[scale=.5]{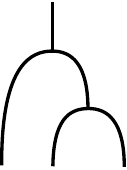}} \ ,
	\qquad\raisebox{-0.5\height}{\includegraphics[scale=.5]{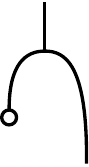}} 
	 	\ = \ \raisebox{-0.5\height}{\includegraphics[scale=.5]{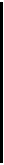}} 
	 	\ = \ \raisebox{-0.5\height}{\includegraphics[scale=.5]{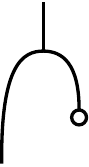}} 
	&&\raisebox{-0.5\height}{\includegraphics[scale=.5]{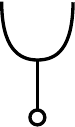}} 
	 	\ = \ \raisebox{-0.5\height}{\includegraphics[scale=.5]{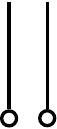}} \ , 
	 \quad \raisebox{-0.5\height}{\includegraphics[scale=.5]{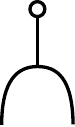}}
	 	\ = \ \raisebox{-0.5\height}{\includegraphics[scale=.5]{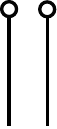}} \\ \\
	&&&\raisebox{-0.5\height}{\includegraphics[scale=.5]{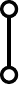}} \ = \ \id_{\one} \\
	\text{(c)} 
	&\raisebox{-0.5\height}{\includegraphics[scale=.5]{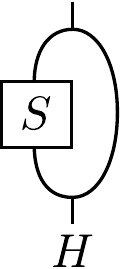}} 
	 	\ = \ \raisebox{-0.5\height}{\includegraphics[scale=.5]{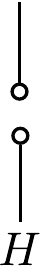}}
	 	\ = \ \raisebox{-0.5\height}{\includegraphics[scale=.5]{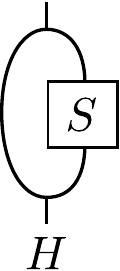}}
 \end{array} 
 \end{align*} 
 \caption{(a) Graphical notation and axioms for an algebra.
 	(b) Compatibility conditions between algebra and coalgebra structures in a bialgebra. (c) Condition on the antipode in a Hopf algebra.} 
	\label{fig:Alg}
\end{figure}

Let $(\Cat,\ot,\one,\alpha,\lambda,\rho)$ be a monoidal category. 
An \emph{algebra in $\Cat$} is a triple $(A,\mu,\eta)$ consisting of 
		an object $A\in\Cat$ and morphisms $\mu: A \ot A \to A$ (multiplication), 
		$\eta: \one \to A$ (unit) which are associative and unital (see Figure \ref{fig:Alg}\,(a)).
The {\em opposite algebra} $A^{op}=(A,\mu^{op},\eta)$ has multiplication $\mu^{op} := \mu \circ c_{A,A}$. 
A \emph{coalgebra in $\Cat$} is a triple $(A,\Delta,\varepsilon)$ consisting of 
		an object $A\in\Cat$ and morphisms $\Delta: A \to A \ot A$ (comultiplication), 
		$\varepsilon: A \to \one$ (counit) which are coassociative and counital.
For an algebra $(A,\mu,\eta)$ in $\Cat$, an \emph{$A$-(left) module} is 
		a pair $(M,m)$ where $M$ is an object in $\Cat$ and $m: A \ot M \to M$ is a morphism compatible with the unit and multiplication of $A$.
		Similarly, one defines an $A$-right module and $A|B$-bimodules (for the latter, left and right action have to commute).
We denote the category of 
	$A$-modules by $A\Mod_\Cat$ and that of $A|A$-bimodules  by 
   $A\BMod[\Cat] B$.

If $\Cat$ is in addition braided, one defines a \emph{bialgebra in $\Cat$} as a tuple $(A,\mu,\eta,\Delta,\varepsilon)$ such 
	that $(A,\mu,\eta)$ is an algebra and $(A,\Delta,\varepsilon)$ a coalgebra 
	in $\Cat$ and such that $\Delta,\varepsilon$ are morphisms of algebras 
	in the sense that the conditions in Figure \ref{fig:Alg}\,(b) hold.

\begin{df}
	A \emph{Hopf algebra} in a braided monoidal category $\Cat$ is a tuple $(H,\mu,\eta,\Delta,\varepsilon,S)$, where 
	$(H,\mu,\eta,\Delta,\varepsilon)$ is a bialgebra in $\Cat$ and $S$ is a morphism 
	$H \to H$ satisfying the conditions in Figure \ref{fig:Alg}\,(c).
\end{df}

Given an algebra $A$ in $\DZ(\Cat)$, by a slight abuse of notation we will denote the category of $U(A)$-modules in $\Cat$ by $A\Mod_\Cat$, instead of $U(A)\Mod_\Cat$. 
 If $(A,\mu,\eta,\Delta,\varepsilon)$ is a bialgebra in $\DZ(\CatD)$, 
 then we equip $A\Mod_\CatD$ with 
 a tensor product: For $A$-modules $(M_1,m_1)$, $(M_2,m_2)$, define${}^{\ref{fn:noassoc}}$ 
\begin{align}
 &(M_1,m_1) \ot (M_2,m_2) \\
 	&\qquad := (M_1 \ot M_2, (m_1 \ot m_2) \circ (\id_A \ot \varphi_{A,M_1} \ot \id_{M_2}) \circ (\Delta \ot \id_{M_1 \ot M_2})).\nonumber
\end{align}
 This turns $A\Mod_\CatD$ into a monoidal category.

\medskip

Let $\Cat$ be an abelian monoidal category with right-exact 
tensor product. Then the category $A\BMod[\Cat] A$ is abelian
(this can for example be seen using Proposition \ref{prop:D_F-abelian-etc},  cf.\ the proof of 
Proposition \ref{prop:fcCardyToAlg}  in Appendix \ref{apx:fcCardyToAlg})
and (non-strict) monoidal category with right-exact
tensor product $\ot_A$.

\subsection{Commutation condition \eqref{eq:comm-rel-defect-ops} implies 
conservation condition \eqref{eq:commutator-Hpert-O}} \label{app:comm-der}

In this appendix we use the notation set up in Section \ref{sec:PertDef2D_FT}, 
and in particular $\psi^\pm, \Phi$ as introduced in 
Section \ref{ssec:PertDefNonlocalConsCharges}. 

\medskip

Since $H_0$ is the generator of translations along the $y$-axis, commuting $H_0$ past a defect circle gives a derivative with respect to the position of the defect circle,\footnote{\label{fn:apx-ccond-comm} 
The commutators involving $H_0$ here and below might be understood as 
$\langle H_0v', \Op(\ldots) v \rangle - \langle \Op(\ldots)^* v', H_0 v \rangle = ...$ 
for all $v,v' \in \F$ ($H_0$ is unbounded, in general; recall that we supposed 
that $\Op(\ldots)$ is bounded).}
\begin{align} \label{eq:comm-rel-aux1}
  & \big[\,H_0 \, , \, \mathcal{O}\big(C_{\eps-\nu,\nu}^{X,\dots,X}(x_1,\dots,x_n)\big)(\psi^{\rho_1},\dots,\psi^{\rho_n})\,\big]
  \nonumber\\
  &\qquad = \frac{\partial}{\partial \nu} \mathcal{O}\big(C_{\eps-\nu,\nu}^{X,\dots,X}(x_1,\dots,x_n)\big)(\psi^{\rho_1},\dots,\psi^{\rho_n}) \ ,
\end{align}
where $\rho_i \in \{ \pm 1\}$.
Since $\overline L_{-1}\psi^+=0$, if $\psi^+$ is inserted at a point $x + iy$, then the operator $\mathcal{O}(\cdots)$ for that cylinder will be annihilated by\footnote{
  The partial derivatives with respect to $y$ are orthogonal to the defect line and require a local deformation of the defect line. This is allowed as the defect line is topological.}
$\bar\partial = \frac 12\left(\partial_x + i \partial_y\right)$. If $\psi^-$ is inserted at $x + iy$, the operator $\mathcal{O}(\cdots)$ will be annihilated by $\partial = \frac 12\left(\partial_x - i \partial_y\right)$. We can therefore trade the derivative with respect to $\nu$ in \eqref{eq:comm-rel-aux1} for a derivative $i \frac{\partial}{\partial x}$ for each $\psi^+$, and for a derivative $-i \frac{\partial}{\partial x}$ for each $\psi^-$. Altogether,
\begin{align} \label{eq:comm-rel-aux2}
&  \big[\,H_0 \, , \, \mathcal{O}\big(C_{\eps_1,\eps_2}^{X,\dots,X}(x_1,\dots,x_n)\big)(\psi^{\rho_1},\dots,\psi^{\rho_n}) \, \big]
\nonumber\\
&
\qquad = \Big(i^{\rho_1}\,\tfrac{\partial}{\partial x_1} + \cdots + i^{\rho_n}\, \tfrac{\partial}{\partial x_n}\Big)
   \mathcal{O}(C_{\eps_1,\eps_2}^{X,\dots,X}(x_1,\dots,x_n))(\psi^{\rho_1},\dots,\psi^{\rho_n}) \ .
\end{align}

Next we look at a regulated version of the multiple integrals. Define, for some $\delta>0$,
\begin{equation} 
	U(n,\delta) := \{ (x_1,\ldots,x_n) \in (0,L-\delta)^n 
		\,|\, x_j + \delta < x_{j+1} \text{ for } j=1,2,\ldots, n-1 \} \ .
\end{equation}
Now we define the regulated multiple integrals in analogy with \eqref{eq:e-def-integrals},
\begin{equation} \label{eq:e-def-integrals-aux1}
e_{n,\delta}^{\rho_1,\dots,\rho_n}
= \int_{U(n,\delta)} 
    f_{n}^{\rho_1,\dots,\rho_n}(x_1,\dots,x_n) \, d^n\mb x \ ,
\end{equation}
where $f_{n}^{\rho_1,\dots,\rho_n}(x_1,\dots,x_n)$ is as in \eqref{eq:f-def-integrands}.
Of course, by the assumption that the function $(x_1,\dots,x_n) \mapsto \langle v', f_{n}^{\rho_1,\dots,\rho_n}(x_1,\dots,x_n) v \rangle$ be integrable for all $v,v' \in \F$, these integrals are well-defined even if we set $\delta$ to zero. However, the right-hand-side of \eqref{eq:comm-rel-aux2} is not integrable, in general. 
After this preparation, we can conclude from \eqref{eq:comm-rel-aux2} that
\begin{equation} \label{eq:H0-comm-aux1}
[H_0,e_{n,\delta}^{\rho_1,\dots,\rho_n}]
= \sum_{k=1}^n i^{\rho_k} \int_{U(n,\delta)} \frac{\partial
f_{n}^{\rho_1,\dots,\rho_n}(x_1,\dots,x_n)}{\partial x_k} \, d^n\mb x \ .
\end{equation}
We note for later use that
\begin{equation} \label{eq:H0-comm-aux1b}
[H_0,e_{0,\delta}] = 0
\qquad \text{and} \qquad
\lim_{\delta\to0} [H_0,e_{1,\delta}^{\rho_1}] = 0 \ ,
\end{equation}
where the equality for $n=0$ follows 
as there are no insertions on the topological defect $X$, and the equality for $n=1$ is immediate from \eqref{eq:H0-comm-aux1} and periodicity of $f_1^{\rho_1}$. 
In the following we will assume that
\begin{equation}
  n \ge 2 \ .
\end{equation}
For $1 \le k \le n$ let 
\begin{align} 
	U(n,k,\delta) &:= \{ (x_1,\ldots,\widehat{x_k},\ldots,x_n) \in (0,L-\delta)^{n-1} 
		\,|\, x_j + \delta < x_{j+1} \\
	&\qquad \text{ for } j=1,\ldots,\widehat{k-1,k},\ldots, n-1; \, 
		x_{k-1} + 2\delta < x_{k+1} \} \ , \nonumber
\end{align}
where we set $x_0 := -\delta$, $x_{n+1} := L$ 
and where $\widehat{\cdots}$ means omission of the respective element(s). 
Together with \eqref{eq:H0-comm-aux1} we obtain 
\begin{align} \label{eq:H0-comm-aux2}
[H_0,e_{n,\delta}^{\rho_1,\dots,\rho_n}] &= \sum_{k=1}^{n} i^{\rho_k} \int_{U(n,k,\delta)} 
f_{n}^{\rho_1,\dots,\rho_n}(\dots,x_{k-1},x_{k+1}-\delta,x_{k+1},\dots) \, d^{n-1}\mb x
\nonumber\\
&-\sum_{k=1}^{n} i^{\rho_k} \int_{U(n,k,\delta)} 
f_{n}^{\rho_1,\dots,\rho_n}(\dots,x_{k-1},x_{k-1}+\delta,x_{k+1},\dots) \, d^{n-1}\mb x \ .
\end{align}
The omissions ``$\cdots$'' stand for the arguments of $f$ which remained unmodified with respect to \eqref{eq:H0-comm-aux1}. The arguments $x_{k-1}$ and $x_{k+1}$ are absent for $k=1$ and $k=n$, respectively. Now, using ($k=1,\ldots,n-1$) 
\begin{align}
	&\int_{U(n,k+1,\delta)} f_{n}^{\rho_1,\dots,\rho_n}
		(\dots,x_{k-1},x_{k},x_{k}+\delta,x_{k+2},\dots) \, d^{n-1}\mb x \\
	&\quad= \int_{U(n,k,\delta)} f_{n}^{\rho_1,\dots,\rho_n}
		(\dots,x_{k-1},x_{k+1}-\delta,x_{k+1},x_{k+2},\dots) \, d^{n-1}\mb x \ , \nonumber
\end{align}
shift the summation variable $k \leadsto k{+}1$ in the second sum of \eqref{eq:H0-comm-aux2} and combine the last term in the first sum and the first term in the second sum. This yields
\begin{align} \label{eq:H0-comm-aux3}
&[H_0,e_{n,\delta}^{\rho_1,\dots,\rho_n}]
 \nonumber\\
 &= i^{\rho_n} \int_{U(n,n,\delta)} 
	f_{n}^{\rho_1,\dots,\rho_n}(\dots,x_{n-1},L-\delta) \, d^{n-1}\mb x
~-~ i^{\rho_1} \int_{U(n,1,\delta)} 
	f_{n}^{\rho_1,\dots,\rho_n}(0,x_2,\dots) \, d^{n-1}\mb x \nonumber\\
&\qquad+ \sum_{k=1}^{n-1} 
\int_{U(n,k,\delta)} \left( i^{\rho_k} - i^{\rho_{k+1}} \right) 
f_{n}^{\rho_1,\dots,\rho_n}(\dots,x_{k-1},x_{k+1}-\delta,x_{k+1},\dots) \, d^{n-1}\mb x. 
\end{align}

Observe that 
\begin{align}
\int_{U(n,n,\delta)} f_{n}^{\rho_2,\dots,\rho_n,\rho_1}(\dots,x_{n-1},L-\delta) \, d^{n-1}\mb x 
= \int_{U(n,1,\delta)} f_{n}^{\rho_1,\dots,\rho_n}(0,x_2,\dots) \, d^{n-1}\mb x
\end{align}
by invariance of the theory under translations by $\delta$ in $x$-direction. Hence, the sum over $\rho_1,\dots,\rho_n \in \{\pm1\}$ of the first two terms in \eqref{eq:H0-comm-aux3} vanishes. 
Let $\psi \equiv \psi^+ + \psi^-$, and abbreviate 
$f_{n}^{\rho_1,\dots,\rho_n}(x_1,\dots,x_{n})$ by $\mathcal{O}'(\psi^{\rho_1},\dots,\psi^{\rho_n})$.
Consider the sum over $\rho_1,\dots,\rho_n \in \{\pm1\}$ of the integrand in the sum over $k$ in \eqref{eq:H0-comm-aux3}. If $\rho_k = \rho_{k+1}$ the prefactor is zero, so that we are left with
\begin{align} \label{eq:H0-comm-aux4}
&\sum_{\rho_1,\dots,\rho_n \in \{\pm1\}}
\left( i^{\rho_k}  - i^{\rho_{k+1}} \right) f_{n}^{\rho_1,\dots,\rho_n}(\dots,x_{k-1},x_{k+1}-\delta,x_{k+1},\dots)
\nonumber\\
&
= 2i \Big(
\mathcal{O}'(\psi,\dots,\psi,\psi^+,\psi^-,\psi,\dots,\psi)
-
\mathcal{O}'(\psi,\dots,\psi,\psi^-,\psi^+,\psi,\dots,\psi) \Big) \ ,
\end{align}
where $\psi^\pm$ are inserted in the $k$'th and $(k{+}1)$'st argument. The corresponding insertion points are $x_{k+1}-\delta$ and $x_{k+1}$. These two points approach each other for $\delta\to0$. Note that the OPE of $\psi^+$ with $\psi^-$ is regular. 
For this reason we will assume
that \eqref{eq:H0-comm-aux4} is dominated, uniformly in $\delta$, by an integrable function. 
Hence, we may interchange integration and the $\delta\to0$ limit, so that we can apply \eqref{eq:comm-rel-defect-ops}. We can now write:
\begin{align} \label{eq:H0-comm-aux5}
&
[H_0,e_{\eps_1,\eps_2;n}^X(\psi)] = \lim_{\delta \to 0}
\sum_{\rho_1,\dots,\rho_n \in \{\pm1\}} [H_0,e_{n,\delta}^{\rho_1,\dots,\rho_n}]
\nonumber\\
& \quad
= 2i \lim_{\delta \to 0}
\sum_{k=1}^{n-1} 
\int_{U(n,k,\delta)}\Big(  \mathcal{O}'(\psi,\dots,\psi,\psi^+,\psi^-,\psi,\dots,\psi) \nonumber\\
&\quad \hspace{10em} 
- \mathcal{O}'(\psi,\dots,\psi,\psi^-,\psi^+,\psi,\dots,\psi) 
\Big) \, d^{n-1}\mb x
\nonumber\\
& \quad
\overset{\text{\eqref{eq:comm-rel-defect-ops}}}= 2i \lim_{\eps \to 0}
\sum_{k=1}^{n-1} \int_{U_k}
\Big( 
\mathcal{O}(C^{\one,X}_{\eps_1-\eps,\eps,\eps_2}(x_k;x_1,\dots,\widehat{x_k,x_{k+1}},\dots,x_n))(\Phi;\psi,\dots,\psi)
\nonumber\\
& \hspace{7em} 
\qquad -  \mathcal{O}(C^{X,\one}_{\eps_1,\eps,\eps_2-\eps}(x_1,\dots,\widehat{x_k,x_{k+1}},\dots,x_n;x_k))(\psi,\dots,\psi;\Phi) 
\Big) \, d^{n-1}\mb x
\nonumber\\
& \quad
= 2i \lim_{\eps \to 0}
\int_0^L \hspace{-.5em} dy 
\int_{U(n-2,0)}
\Big( \mathcal{O}(C^{\one,X}_{\eps_1-\eps,\eps,\eps_2}(y;x_1,\dots,x_{n-2}))(\Phi;\psi,\dots,\psi)
\nonumber\\
&\hspace{8em} 
     \qquad -  \mathcal{O}(C^{X,\one}_{\eps_1,\eps,\eps_2-\eps}(x_1,\dots,x_{n-2};y))(\psi,\dots,\psi;\Phi) 
\Big) \, d^{n-2}\mb x \ ,
\end{align}
where $\widehat{x_k,x_{k+1}}$ stands for omitting the $k$'th and $(k{+}1)$'st argument and 
\begin{align}
U_k := \{(x_1,\ldots,\widehat{x_{k+1}},\ldots,x_n)\in\R^n \,|\, 0<x_1<\cdots<\widehat{x_{k+1}}<\cdots<x_n<L\}.
\end{align} 
Under the assumption that the $\eps_1,\eps_2\to 0$ limit exists, this shows that 
\begin{equation} \label{eq:H0-comm-aux6}
[H_0, \lim_{\eps_1,\eps_2 \to 0} e_{\eps_1,\eps_2;n}^X(\psi)] 
= 2i [H_\mathrm{int}, \lim_{\eps_1,\eps_2 \to 0} e_{\eps_1,\eps_2;n-2}^X(\psi)] \ ,
\end{equation}
where $H_\mathrm{int}$ was defined in \eqref{eq:HintDef}. 
To complete the argument that condition \eqref{eq:comm-rel-defect-ops} implies the commutator \eqref{eq:commutator-Hpert-O}, we compute
\begin{align} \label{eq:H0-comm-aux7}
&
[H_\mathrm{pert}(\Psi; -2i\lambda^2), \mathcal{O}(X,\psi;\lambda)] 
\nonumber\\
&
= 
\sum_{n=0}^\infty \lambda^n [H_0, \lim_{\eps_1,\eps_2 \to 0} e_{\eps_1,\eps_2;n}^X(\psi)] 
- 2 i 
\sum_{n=0}^\infty \lambda^{n+2} [H_\mathrm{int}, \lim_{\eps_1,\eps_2 \to 0} e_{\eps_1,\eps_2;n}^X(\psi)] 
\\
&
\overset{(*)}= 2 i 
\sum_{n=2}^\infty \lambda^n [H_\mathrm{int}, \lim_{\eps_1,\eps_2 \to 0} e_{\eps_1,\eps_2;n-2}^X(\psi)] 
- 2 i 
\sum_{n=0}^\infty \lambda^{n+2} [H_\mathrm{int}, \lim_{\eps_1,\eps_2 \to 0} e_{\eps_1,\eps_2;n}^X(\psi)] 
= 0 \ , \nonumber
\end{align}
where (*) follows from \eqref{eq:H0-comm-aux1b} and \eqref{eq:H0-comm-aux6}.

\subsection{The algebras $\E(\h; X)$}
\label{apx:algE}

Let $\h$, $X$ be finite dimensional vector spaces. 
Define $\E_0 = \E_0(\h; X)$ 
as the algebra freely generated by elements 
of $\h$ as commuting variables and elements of $X$ as non-commuting variables 
(modulo linear relations), i.e., 
\begin{align}
	\E_0(\h; X) := \talg(\h \oplus X)/I,
\end{align}
where $\talg(\h \oplus X)$ denotes the tensor algebra of $\h \oplus X$ 
and $I$ is the ideal generated by $\{ h \ot h' - h' \ot h \mid\; h,h' \in \h \}$. 
Fix some norm $\|.\|$ on $\h \oplus X$. 
For $r>0$ define $\|.\|_r$ as the norm on $\E_0$ given by
\begin{align} \label{eq:df_norm_r}
	\|y\|_r := \inf\Big\{ \sum_{n \in \N} 
		r^n \sum_{i=1}^{m_n} \prod_{j=1}^n \| y_{n,i,j} \| ~\Big|~ 
			y = \sum_{n \in \N} \sum_{i=1}^{m_n} \bigotimes_{j=1}^n y_{n,i,j} ~ + I \Big\}
\end{align}
for $y \in \E_0$, and where $y_{n,i,j} \in \h \oplus X$ denotes the $j$'th tensor factor in a decomposition of the homogeneous component of $y$ in $(\h \oplus X)^{\ot n}$ into a sum of $m_n$ pure tensors.
The norm $\|y\|_r$ can be understood as a weighted sum of projective cross norms. 

The algebra $\E_0$ is graded by the number of factors which are elements of $X$, 
i.e., 
\begin{gather} 
	\E_0 \cong \bigoplus_{n \in \N} \E_0^{(n)}, \qquad \text{where} \qquad 
	\E_0^{(0)} := S(\h) = \talg(\h)/I, \nonumber\\
	\E_0^{(n+1)} := S(\h) \ot X \ot \E_0^{(n)}.
\end{gather}
Let $\E^{(n)}$ be the completion of $\E_0^{(n)}$ with respect to the family 
of norms $\|.\|_r$, $r>0$.
Note, for example, that $\E^{(1)}$ is bigger than $\E^{(0)} \ot X \ot \E^{(0)}$ since it includes convergent sums of the form $\sum_n e_n \ot x \ot f_n$ for fixed $x \in X$.

Clearly, the completion does not depend on the 
initial choice of norm $\|.\|$ on $\h \oplus X$: 
if $c,C>0$ are such that $c\|x\|' \leq \|x\| \leq C\|x\|'$ for all 
$x \in \h \oplus X$ (all norms on the finite dimensional space $\h \oplus X$ are equivalent), then $\|a\|'_{cr} \leq \|a\|_r \leq \|a\|'_{Cr}$ 
holds for all $a \in \E_0$. 
Define 
\begin{align} \label{eq:def-of-E}
	\E(\h; X) := \bigoplus_{n \in \N} \E^{(n)}.
\end{align}
Denote the algebra 
$\E(\C h_1 \oplus \cdots \oplus \C h_n; \C x_1 \oplus \cdots \oplus \C x_m)$
by $\E(h_1,\ldots,h_n; x_1,\ldots,x_m)$.

Note that the multiplication on $\E_0$ is continuous, since 
$\|a \ot b\|_r \leq \|a\|_r\|b\|_r$ for all $a,b \in \E_0$ 
(and since it respects the grading $\E_0 \cong \bigoplus_{n \in \N} \E_0^{(n)}$). Hence, 
it extends to a continuous multiplication on the grade-wise completion and turns 
$\E(\h; X)$ into a topological algebra. 

\begin{rem}
The above completion may seem artificial, but it is in fact quite natural: 
the algebra $\E(h_1,\ldots,h_m)$ is isomorphic to the algebra of entire complex analytic functions in $m$ 
variables with the topology of uniform convergence on compact sets. 
To demonstrate this, we first note  two elementary facts.
\begin{itemize}
\item An entire complex analytic function 
is the same as a series 
\begin{align}
f(z_1,\ldots,z_m) = \sum_{\alpha \in \N^m} f_\alpha \cdot z^\alpha \ ,
\end{align}
where $f_\alpha \in \C$ and $z^\alpha := z_1^{\alpha_1}\cdots z_m^{\alpha_m}$, 
such that for all $r > 0$ the series 
$\|f\|'_r := \sum_{\alpha \in \N^m} |f_\alpha| \cdot r^{|\alpha|}$ 
converges (here, $|\alpha| := \alpha_1 + \cdots + \alpha_m$). 
This is because 
if $\|f\|'_r = \infty$ for some $r>0$, then for each $N \in \N$ 
there is $n>N$ and $\alpha \in \N^n$ such that $|f_\alpha| r^n \geq 2^{-n}$, 
i.e., $|f_\alpha \cdot (2r)^n| \geq 1$. 
Hence, the series $f(2r,\ldots,2r)$ does not converge in $\C$, a contradiction. 
\item 
The topology of uniform 
convergence on compact sets is induced by the family of 
norms $\|.\|'_r$, $r>0$: From $\lim_n \|f_n-f\|'_r = 0$ 
it follows that $(f_n)$ converges uniformly to $f$ on the compact set 
$\{z \in \C^m \,|\, |z_i| \leq r\}$. On the other hand, 
if a sequence $(f_n)$ converges uniformly to $f$ on 
the set $\{z \in \C^m \,|\, |z_i| \leq 2r\}$, 
then for every $\varepsilon > 0$ there is $N \in \N$ such that 
$\|f_n-f\|_\infty < \varepsilon$ for all $n > N$. 
With Cauchy's integral formula for the Taylor series 
coefficients $(f_n-f)_\alpha$ of $f_n - f$, i.e., 
$(f_n-f)_\alpha =\frac 1{(2\pi i)^m} 
	\int_{|z_i|=2r} \frac{f_n(z)-f(z)}{z^{\alpha + (1,\ldots,1)}}\,d\mb z$, 
one sees that $|(f_n-f)_\alpha| \leq \frac\varepsilon{(2r)^{|\alpha|}}$ 
for all $n>N$. 
This implies $\lim_n \|f_n-f\|'_r = 0$. 
\end{itemize}
Then denote by $\Phi$ the algebra map from $\E_0(h_1,\ldots,h_m)$ 
to the algebra of entire complex analytic functions in $m$ variables 
determined by 
\begin{align}
	\Phi(h_{i_1} \ot \cdots \ot h_{i_k} + I) = z_{i_1}\cdots z_{i_k} \ .
\end{align}
If we choose the norm $\|.\|$ on $\h$ such that $\|h\| := \|\Phi(h)\|_1'$ 
for all $h \in \h$, then by \eqref{eq:df_norm_r}
we have $\|y\|_r \leq \|\Phi(y)\|'_r$ 
for all $y \in \E_0$. 
On the other hand, $\|\Phi(y)\|'_r \leq \|y\|_r$, since 
if $y = \sum_{n \in \N} \sum_{i=1}^{m_n} \bigotimes_{j=1}^n y_{n,i,j}$, then 
\begin{align}
	\|\Phi(y)\|'_r = \Big\| \sum_{n \in \N} \sum_{i=1}^{m_n} \prod_{j=1}^n \Phi(y_{n,i,j}) \Big\|'_r 
		\leq \sum_{n \in \N} \sum_{i=1}^{m_n} \prod_{j=1}^n \|\Phi(y_{n,i,j})\|'_r 
		= \sum_{n \in \N} \sum_{i=1}^{m_n} \prod_{j=1}^n \|y_{n,i,j}\|_r \ .
\end{align}
In the second step, submultiplicativity of $\|.\|'_r$ was used 
($\|fg\|'_r \leq \|f\|'_r\|g\|'_r$). 
Together with the fact that the algebra of entire 
complex analytic functions is complete with respect to the topology 
of uniform convergence on compact sets (by Morera's 
theorem, uniform limits of holomorphic functions are holomorphic), 
one finds that $\Phi$ extends to an isomorphism of topological algebras. 
\end{rem}

The above remark implies in particular that $e^{\hbar h}$ is an element of $\E(h; 0)$; this will be important in the definition of the quantum groups $\tilde U_\hbar(\lsl_2)$ and $\tilde U_\hbar(L\lsl_2)$.

\medskip

We need the following result only for $\h$ acting diagonalisably 
on a vector space $V$. Since it does not make much more effort, 
we prove it for $\h$ acting \emph{locally finitely}. 
By this we mean that $\dim \, \mathrm{span}\{h.v \, | \, h \in \h\} < \infty$ 
for each $v \in V$. 

\begin{prop} \label{prop:apx_algE}
Let $V$ be a representation of $\E_0(\h; X)$ in which $\h$ acts locally finitely. There exists a unique extension to a representation of $\E(\h; X)$ such that for all $\varphi \in V^*$ and $v \in V$, the function $\E(\h; X) \to \C$, $x \mapsto \varphi(x.v)$, is continuous.
\end{prop}

\begin{proof}
Denote the action of $\E_0$ on $V$ by $\rho$.
We need to define an action $\tilde\rho$ of $\E(\h;X)$ on $V$. 
As $\E(\h;X)$ is the direct sum of $\E^{(n)}$'s, 
it suffices to define $\tilde\rho$ for 
elements of $\E^{(n)}$. 
 Fix $v \in V$. As $X$ is finite dimensional and $\h$ acts locally finitely, 
 the subspace $\rho(\E_0^{(n)}).v \subseteq V$ is finite dimensional. 
 Let $\|.\|_v$ be any norm on $\rho(\E_0^{(n)}).v$. 
 There is $r>0$ such that $\|\rho(h).w\|_v < r\|h\|\|w\|_v = \|h\|_r\|w\|_v$ 
 holds for all $h \in \h$ and $w \in \rho(\E_0^{(n)}).v$ and 
 $\|\rho(x).w\|_v < \|x\|_r\|w\|_v$ holds for all $x \in X$ 
 and $w \in \rho(\E_0^{(n-1)}).v$. 
 This is true because $\h \oplus X$ and $\rho(\E_0^{(n)}).v$ are finite dimensional. 
 With \eqref{eq:df_norm_r} it follows that $\|\rho(a).w\|_v < \|a\|_r\|w\|_v$ 
 for all $a \in \E_0^{(n)}$ and $w \in \rho(\E_0^{(n)}).v$. 
 Let $(y_i)_{i \in I}$ be a Cauchy net in $\E_0^{(n)}$ approximating $y$. 
 Then $\|\rho(y_i - y_j).v\|_v < \|y_i - y_j\|_r\|v\|_v$ for all $i,j \in I$ 
 and this means that $(\rho(y_i).v)_{i \in I}$ is a Cauchy net in $V$. 
 Define $\tilde\rho(y).v := \lim_{i \in I} \rho(y_i).v$.
 
 Since $\rho(\E_0^{(n)}).v$ is finite-dimensional it is closed, and hence also $\tilde\rho(y).v \in  \rho(\E_0^{(n)}).v$. Continuity of $y \mapsto \varphi(\tilde\rho(y).v)$ follows from the continuity of $y \mapsto \tilde\rho(y).v$ with respect to the $\|.\|_v$-norm and since $\varphi(-)$ is linear and hence continuous on the finite-dimensional space $\rho(\E_0^{(n)}).v$. Uniqueness is clear from $\tilde\rho(y).v = \lim_i \rho(y_i).v$ together with the fact that on $\rho(\E_0^{(n)}).v$, the forms $\varphi(-)$ are continuous and separate points.
\end{proof}

Define the tensor product $\bar\ot$ by 
\begin{align} 
	\E(\h; X) \bar\ot \E(\h'; X') := \E(\h \oplus \h'; X \oplus X') / J, 
\end{align}
where $J$ is the closure of the ideal generated by all elements 
$a a' - a' a \in \E(\h \oplus \h'; X \oplus X')$ for $a \in \h \oplus X$ 
and $a' \in \h' \oplus X'$. 

\begin{rem}
 Any linear homogeneous (with respect to the grading \eqref{eq:def-of-E}) 
 map $f: \h \oplus X \to \E(\h'; X')$ extends uniquely 
 to a continuous algebra homomorphism 
 $\tilde f: \E(\h; X) \to \E(\h'; X')$ (which is 
 necessarily grade-preserving). 
 To see this, let $\tilde f|_{\E_0}$ be the unique extension of $f$ 
 to $\E_0$, 
 choose some norm $\|.\|$ on $\h \oplus X$, let 
 $r > 0$ be a positive real number and let 
 $C > 0$ be such that $\|f(x)\|_r \leq C \cdot \|x\|$ holds for all 
 $x \in \h \oplus X$. Then, by definition of $\|.\|_r$ and $\|.\|_C$, 
 $\|\tilde f|_{\E_0}(x)\|_r \leq \|x\|_C$ holds true 
 for any $x \in \E_0(\h; X)$. Hence, $\tilde f|_{\E_0}$ is continuous. 
 	The claim follows as $\tilde f|_{\E_0}$ is homogeneous and 
 	the homogeneous components of $\E(\h'; X')$ are complete. 
 It follows that the coproducts in Definition \ref{df:qgrps_uncomp} 
 are well-defined continuous homomorphisms $\Delta: \E \to \E \bar\ot \E$. 
\end{rem}

In order to avoid the effort of introducing 
$\E(\h,X)$ for the definition of 
$\tilde U_\hbar$ (see Definition \ref{df:qgrps_uncomp}), it would be tempting to work with just $U_q$ (see Definition \ref{df:qgrps_comp}), but to be able to 
 consider representations of $U_q$ as 
 representations of $\h$ (e.g. Theorem \ref{thm:exrepcatfuncmon_uncpt1}), 
 we would need to fix ``logarithms'' for 
 the action of the generator ``$k$'' related to the braiding. As long as one is only 
 concerned with finite dimensional representations, the only eigenvalues of 
 $k$ are of the form $q^n$ for some $n \in \Z$. So if $q$ is not a root of unity, 
 there is a canonical ``logarithm'' for $k$. But the representations underlying 
 the ``$Q$ operators'' are infinite dimensional. Also, we do not a priori want to exclude 
 the root of unity case. 
 Hence, we prefer choosing a ``logarithm'' for $k$ by introducing 
 a generator $h$ and letting $k = e^{\hbar \cdot h}$ (as done in \cite{BLZ99}). But the exponential 
 series $e^{\hbar \cdot h}$ is not an element of the \emph{algebra} generated by $h$. 
 Hence we need to complete with respect to a suitable topology. 

Comparing to other approaches, for example 
in \cite[Ex.\,6.2]{Drin86}, Drinfeld defines $U_\hbar$ with primitive Cartan-like generators $h$, 
 but with $\hbar$ being a formal parameter instead of a complex number;
in \cite[Ch.\,6.2]{EFK98} one finds a version of $U_q$ with $q$ being a number and with generators ``$q^h$'', $h \in \h$, but without any specified topology (otherwise $h$ should be recovered as a limit). We are not aware of a version of Drinfeld's quantum group with numerical parameter in the literature. 

\subsection{Proof of Theorem \ref{thm:UqCommGrothUnc}}
\label{apx:B}

Let $U_q(L\lsl_2)$ be as in Definition \ref{df:qgrps_comp}. We will reduce the statement to commutativity of $\Repfd(U_q(L\lsl_2))$, which is known from \cite[Corollary 2]{FR98}. We will abbreviate $U_q := U_q(L\lsl_2)$ and $\tilde U_\hbar := \tilde U_\hbar(L\lsl_2)$.
 
 \medskip
 
 The Hopf algebra homomorphism $U_q \to \tilde U_\hbar$ 
 determined by 
\begin{gather}
 k_j \mapsto e^{\hbar \cdot h_j}, \qquad e_j^\pm \mapsto e_j^\pm, \qquad (j=0,1)
\end{gather}
 gives rise to an exact monoidal functor 
 $\fcG: \, \Repfd(\tilde U_\hbar) \to \Repfd(U_q)$, 
 which descends to a ring homomorphism 
 $g: \Gr(\Repfd(\tilde U_\hbar)) \to \Gr(\Repfd(U_q))$. 

Note that the only eigenvalues of $k_0$ in a finite dimensional module are of the form $q^n$, $n \in \Z$.
The functor $\fcG$ has a monoidal ``section'' $\mathcal S$ defined on the tensor subcategory of $U_q$-modules 
 where $k_0$ acts diagonalisably: it sends $X \in \Repfd(U_q)$ with 
 diagonalisable $k_0$-action to 
 $\mathcal S(X) \in \Repfd(\tilde U_\hbar)$, which is the same vector space 
 as $X$ and where the generators $e_j^\pm$, $j=0,1$, of $\tilde U_\hbar$ act as the corresponding 
 generators of $U_q$ on $X$, but $h_0$ acts on 
 $k_0$-eigenvectors $v \in X$ for eigenvalue 
 $q^n$ by $h_0.v = n\hbar v$. 
 Since finite dimensional representations have finite length, 
 $\Gr(\Repfd(U_q))$ as well as $\Gr(\Repfd(\tilde U_\hbar))$ 
 are freely generated as abelian groups by the 
 classes $[X]$ of simple objects $X$.
 Hence, the functor $\mathcal S$ descends to a ring homomorphism 
 $s: \Gr(\Repfd(U_q)) \to \Gr(\Repfd(\tilde U_\hbar))$ 
 satisfying $g \circ s = \id$. 
For the same reason, it suffices to show that the classes $[X],[Y]$ of any two simple objects 
 $X,Y \in \Repfd(\tilde U_\hbar)$ commute in $\Gr(\Repfd(\tilde U_\hbar))$.
 For $m \in \Z$ let $X_m$ be the simple $\tilde U_\hbar$-module which is $\C$ 
 as a vector space, where $e_j^\pm$ ($j=0,1$) act by $0$ and $h_0$ acts by multiplication 
 with $2\pi im/\hbar$. 
 Then $g([X_m]) = 1$ for all $m \in \Z$, and we have 
 $X_m \ot X \cong X \ot X_m$, and hence $[X_m][X] = [X][X_m]$, for all $X \in \Repfd(\tilde U_\hbar)$. 
 Note that if $X \in \Repfd(\tilde U_\hbar)$ is simple, then there 
 is a unique $m \in \Z$ such that $[X] = [X_m] \cdot s(g([X]))$. 
 So let $X,Y \in \Repfd(\tilde U_\hbar)$ be simple and let $m,n \in \Z$ be 
 such that $[X] = [X_m] \cdot s(g([X]))$ and $[Y] = [X_n] \cdot s(g([Y]))$. 
 Then 
 \begin{align}
 	[X][Y] &= [X_m][X_n]  s(g([X]))s(g([Y])) = [X_m][X_n]  s(g([X]) g([Y])) \\
 		&\overset{(*)}= [X_m][X_n]  s(g([Y])g([X])) = [Y][X], \nonumber
 \end{align}
 where in (*) the commutativity of $\Repfd(U_q)$ was used.

\subsection{Proof of Proposition \ref{prop:fcCardyToAlg}}
\label{apx:fcCardyToAlg}

{\em Part 1.}
	Denote the functor \eqref{eq:C-to-AAbim} by $\fcG$. 
	Since $F^-_\Cat$ and $A$ are mutually transparent, we 
	have $U\alpha^+(F^-_\Cat) = U\alpha^-(F^-_\Cat)$ in 
	$\CatD = A\BMod[\Cat]A$. Hence, $U\alpha^+(m)$ is indeed a morphism 
	$UF \ot_A U\alpha^+(X) \to U\alpha^+(X)$ in $\CatD$ if 
	$m \in \Hom_\Cat((F^+_\Cat \oplus F^-_\Cat) \ot X, X)$. 
	This shows that $\fcG$ is a functor 
	$\Cat_{F_{\DZ\Cat}} \to \CatD_F$. 
	Let $(X,m)$ and $(X',m')$ be 
	objects of $\Cat_{F_{\DZ\Cat}}$ and write $m'_\pm := m' \circ (\iota_{F^\pm} \ot \id_{X'})$. Then
	\begin{align}
	&U\alpha^+(T(m,m')) \stackrel{\eqref{eq:Tpmorph}}= U\alpha^+\left(m \ot \id_{X'} 
	+ (\id_X \ot (m'_+ \oplus m'_-)) \circ 
		((c_{F^+_\Cat,X} \oplus c_{X,F^-_\Cat}^{-1}) \ot \id_{X'})\right) \nonumber\\
	&\quad= 
		U\alpha^+(m \ot \id_{X'}) + U\alpha^+(\id_X \ot (m'_+ \oplus m'_-)) \circ 
		U\alpha^+((c_{F^+_\Cat,X} \oplus c_{X,F^-_\Cat}^{-1}) \ot \id_{X'}) \nonumber\\
	&\quad\overset{(*)}=
		U\alpha^+(m \ot \id_{X'}) + U\alpha^+(\id_X \ot (m'_+ \oplus m'_-)) \circ 
		((\varphi_{F^+,\fcG(X)} \oplus \varphi_{F^-,\fcG(X)}) \ot \id_{\fcG(X')}) \nonumber\\
	&\quad= T(U\alpha^+(m),U\alpha^+(m')) \ .
	\end{align}
For step (*) first note that $\alpha^+(c_{F^+_\Cat,X}) = \alpha \circ \iota^+(c_{F^+_\Cat,X}) = \alpha(c_{F^+_\Cat,X})$ as morphisms $A \ot F^+_\Cat \ot X \to A \ot X \ot F^+_\Cat$ in $\Cat$. By Proposition \ref{prop:alphaind}, $\alpha$ is braided, so that this agrees with the half-braiding $\varphi_{F^+,X}$ of $F^+$. The argument for $\varphi_{F^-,X}$ is analogous, and uses once more that $U\alpha^+(F^-_\Cat) = U\alpha^-(F^-_\Cat)$.

The above equality shows that $\fcG$ is monoidal. 

\medskip\noindent
{\em Part 2.}
That $\fcG$ preserves the commutation condition can be verified along the same lines as monoidality, we omit the details.
	
	\medskip
	
	In the proof of Part 3, we will several times make use of the following 
	\begin{obs}\label{obs:exact-in-BMod-iff-in-Cat}
	A sequence $B \to C \to D$ in $A\BMod[\Cat]A$ is exact if and only if 
	$B \to C \to D$ is exact in $\Cat$, where the functor forgetting 
	about the actions of $A$ is implicit.
	\end{obs}
	This is a consequence of Lemma \ref{lem:exact_iff}, 
	regarding $A\BMod[\Cat] A$ as a full abelian (non-monoidal) 
	subcategory 
	of $\Cat_{\iota^+(A \oplus A^{op})}$ 
	(use the braiding on $\Cat$ to turn the right action of $A$ into a left action of $A^{op}$ 
	and observe that kernels and cokernels of a bimodule morphism are bimodules). 

\medskip\noindent
{\em Part 3.}
	Let $0 \to X \to Y \to Z \to 0$ be an exact sequence in $\Cat$. 
	Since by assumption $A \ot (-): \Cat \to \Cat$ is exact,  also 
	$0 \to A \ot X \to A \ot Y \to A \ot Z \to 0$ is an exact sequence in $\Cat$. 
	Observation \ref{obs:exact-in-BMod-iff-in-Cat} now implies that  
	$0 \to U\aind^+(X) \to U\aind^+(Y) \to U\aind^+(Z) \to 0$ 
	is exact in $A\BMod[\Cat] A$. 
	It remains to show that right-exactness of 
	$F^\pm_\Cat \ot (-) : \Cat \to \Cat$ 
	implies right-exactness of 
	$U F \ot_A (-) : A\BMod[\Cat] A \to A\BMod[\Cat] A$. 
	With another application of Lemma \ref{lem:exact_iff},
	this will imply that $0 \to \mathcal G(X) \to \mathcal G(Y) \to \mathcal G(Z) \to 0$ is 
	exact in $(A\BMod[\Cat]A)_F$. 
	
	So let $B \to C \to D \to 0$ be an exact sequence in 
	$A\BMod[\Cat]A$ and note that, by Observation \ref{obs:exact-in-BMod-iff-in-Cat}, 
	$B \to C \to D \to 0$ is exact in $\Cat$. 
	We have to show that 
	$UF \ot_A B \to UF \ot_A C \to UF \ot_A D \to 0$ is exact in $A\BMod[\Cat]A$. 
	As there are isomorphisms 
	$U\alpha^\pm(F^\pm_\Cat) \ot_A B \cong F^\pm_\Cat \ot^\pm B$, 
	$\ldots$ in $A\BMod[\Cat]A$ (cf.\ Remark \ref{rem:alpha-frs-relation}) 
	we find that it suffices to show exactness of the sequence 
	$(F^+_\Cat \ot^+ B) \oplus (F^-_\Cat \ot^- B) \to \cdots \to 0$ 
	in $A\BMod[\Cat]A$. 
	By Observation \ref{obs:exact-in-BMod-iff-in-Cat}, this is 
	equivalent to the exactness of 
	$(F^+_\Cat \oplus F^-_\Cat) \ot B \to \cdots \to 0$ 
	in $\Cat$, which follows from the exactness of 
	$B \to C \to D \to 0$ in $\Cat$ together with 
	right-exactness of $F^\pm_\Cat \ot (-) : \Cat \to \Cat$.

\subsection{Proof of Proposition \ref{prop:fcAB}}
\label{apx:fcAB}

Before giving the proof of Proposition \ref{prop:fcAB}, we will compute the 3-cocycle defining the tensor product on the semisimple subcategory of $A_r\BMod[\Cat]A_r$ spanned by the simple objects $X(\beta,\xi)$. For the proof we only need that this cocycle is trivial on the span of $X(\beta,1)$, but the general result is of interest on its own.

\medskip

A \emph{pointed category} is a $\C$-linear abelian 
semisimple monoidal category with exact tensor product such 
that the set of isomorphism classes of simple objects forms a group $G$. This means that for $g,g' \in G$ and $x \in g$, $y \in g'$, we have $x \ot y \in gg'$. 
Let $\Cat_G$ be the $\C$-linear abelian semisimple 
(with arbitrarily large direct sums) category with 
simple objects $x_g$, $g \in G$, and tensor product $x_g \ot x_{g'} = x_{gg'}$. 
This category is a strict version of the category of $G$-graded 
vector spaces. 
Given a 3-cocycle $\psi \in Z^3(G,\C^\times)$, one obtains a monoidal structure 
on $\Cat_G$ by defining associativity constraints 
$\alpha_{g,g',g''} := \psi(g,g',g'') \cdot \id_{gg'g''}$ 
(the pentagon axiom is equivalent to the 3-cocycle condition, see \cite[App.\,E]{MooreSeib89}
or \cite[Sect.\,3]{JS93}). 
Denote this category by $\Cat_G(\psi)$. Cohomologous 
3-cocycles yield isomorphic monoidal structures, so that 
$\Cat_G(\psi)$ depends only on the cohomology class of $\psi$ 
in $H^3(G,\C^\times)$. 
Any pointed category containing arbitrarily large direct sums is 
monoidally equivalent to some $\Cat_G(\psi)$. The class of the 
3-cocycle is uniquely determined by the requirement that the equivalence 
takes $x_g$ to an object in the isomorphism class $g$. 

By Remark \ref{rem:Ar-bimods}, 
the full subcategory $(A_r\BMod[\Cat]A_r)_\mathrm{ss}$ of 
$A_r\BMod[\Cat]A_r$ formed by arbitrarily large 
direct sums of simple objects is pointed, with group of 
simple objects 
$G := \C^\times \times \C^\times \ (\cong (\C / r\Z) \times \C^\times)$. 
	For an element $g = (\eta,\xi) \in G$ we pick 
	$X_g := X(\mathcal{L}(\eta),\xi)$ as representative of the 
	corresponding isomorphism class. 
Recall from Remark \ref{rem:Ar-bimods} that 
\begin{align}
	X_g = X(\mathcal{L}(\eta),\xi) = \textstyle \big(\bigoplus_{n \in \Z} \C, \ 
		\bigoplus_{n \in \Z} \gamma_{\mathcal{L}(\eta) + nr} \, \id_{\C} \big)  
\end{align} 
as an object in $\Cat = \Repss(\algA)$, 
where $\gamma_{\mathcal{L}(\eta) + nr}$ is the algebra homomorphism 
$\algA \to \C$ sending $\fxh$ to $\mathcal{L}(\eta) + nr \in \C$. 
We denote by $\iota_{\mathcal{L}(\eta), n}$ the canonical embedding 
of $\C$ into the $n$'th summand of the underlying vector space 
$\bigoplus_{n \in \Z} \C$ of $X_g$. 
The tensor product on $(A_r\BMod[\Cat]A_r)_\mathrm{ss}$ 
is defined through choices of cokernels 
$(\cok_{X,Y}, X \otA Y)_{X,Y \in (A_r\BMod[\Cat]A_r)_\mathrm{ss}}$, 
see Appendix \ref{appx:alg-etc-in-braided-cats}. 
Different choices lead to canonically equivalent 
monoidal structures. Therefore we have the freedom 
to pick a convenient one. 
For simple objects $X_g$ and 
$X_{g'}$ with $g = (\eta,\xi)$ and $g' = (\eta',\xi')$ 
we let $\lambda_{g,g'}$ be the map 
\begin{align} 
\label{eq:3cocycle-aux1}
\lambda_{g,g'} : &\ X_g \ot X_{g'} \to X_{gg'} \ , \\ 
	&\iota_{\mathcal{L}(\eta), n}(a) \ot \iota_{\mathcal{L}(\eta'), m}(b) 
	\ \mapsto\ \iota_{\mathcal{L}(\eta \eta'), n+m+\sigma(\eta,\eta')}
		\left(\xi^{m}\cdot ab\right) \ , \quad a,b \in \C \ , \ 
		n,m \in \Z \ , \nonumber
\end{align}
where $\sigma$ is the 2-cocycle 
\begin{align} 
\sigma: &\ \C^\times \times \C^\times \to \Z \ , \\ 
	&(\eta,\eta') \mapsto 
	\frac{\mathcal{L}(\eta) + \mathcal{L}(\eta') - \mathcal{L}(\eta \eta')}r = 
	\begin{cases} 0 & \text{if }\mathrm{Re}\frac{\mathcal{L}(\eta) + \mathcal{L}(\eta')}r < 1 \ ,\\
		1 & \text{else} \ . \end{cases} \nonumber
\end{align}
The addition of $\sigma(\eta,\eta')$ on the right hand side of \eqref{eq:3cocycle-aux1} ensures that $\lambda_{g,g'}$ is a morphism in $\Cat$. The factor $\xi^m$ compensates the difference of the right action $(\xi')^k \cdot (\id_{X_g} \ot  s^k)$ of $1_{kr}$ on the left hand side and $(\xi \xi')^k \cdot s^k$ on the right hand side, see \eqref{eq:Ar-action-X(b,x)}, so that $\lambda_{g,g'}$ is 
a morphism of $A_r|A_r$-bimodules. The factor $\xi^m$ is also required in order for $(\lambda_{g,g'},X_{gg'})$ to be a cokernel for 
the morphism $l-r : X_g \ot A_r \ot X_{g'} \to X_g \ot X_{g'}$.

For $a,b,c \in \C$, $n_i \in \Z$, $g_i = (\eta_i,\xi_i) \in G$, and writing 
$X_i := X_{g_i}$, $i=1,2,3$, we have 
\begin{align} 
	\lambda_{g_1g_2,g_3} \circ (\lambda_{g_1,g_2} &\ot \id_{X_{3}}): 
	\iota_{\mathcal{L}(\eta_1), n_1r}(a) \ot \iota_{\mathcal{L}(\eta_2), n_2r}(b) \ot 
	\iota_{\mathcal{L}(\eta_3), n_3r}(c) \\
	&\mapsto\ \iota_{\mathcal{L}(\eta_1\eta_2\eta_3),\ n_1+n_2+n_3 
		+ \sigma(\eta_1,\eta_2) + \sigma(\eta_1\eta_2,\eta_3)} 
		\left(\xi_1^{n_2}(\xi_1\xi_2)^{n_3}\cdot abc\right) \nonumber
\end{align}
and 
\begin{align} 
	\lambda_{g_1,g_2g_3} \circ (\id_{X_{1}} &\ot \lambda_{g_2,g_3}):  
	\iota_{\mathcal{L}(\eta_1), n_1r}(a) \ot \iota_{\mathcal{L}(\eta_2), n_2r}(b) \ot 
	\iota_{\mathcal{L}(\eta_3), n_3r}(c) \\
	&\mapsto\ \iota_{\mathcal{L}(\eta_1\eta_2\eta_3),\ n_1+n_2+n_3 
		+ \sigma(\eta_2,\eta_3) + \sigma(\eta_1,\eta_2\eta_3)} 
		\left(\xi_1^{n_2+n_3+\sigma(\eta_2,\eta_3)}\xi_2^{n_3}\cdot abc\right) \ . \nonumber
\end{align}
Hence,
\begin{align} 
\lambda_{g_1,g_2g_3} \circ (\id_{X_{1}} \ot \lambda_{g_2,g_3}) \circ \alpha_{X_{1},X_{2},X_{3}}
	= \psi(g_1,g_2,g_3) \cdot \lambda_{g_1g_2,g_3} \circ (\lambda_{g_1,g_2} \ot \id_{X_{3}}) \ , 
\end{align}
where $\alpha_{X_{1},X_{2},X_{3}}$ are the 
associativity constraints in $\Cat$ and $\psi$ is the 3-cocycle
\begin{align} \label{eq:3-cocycle}
	\psi(g_1,g_2,g_3) = \xi_1^{\sigma(\eta_2,\eta_3)} \ . 
\end{align}
By the definition of associativity constraints in the 
category of bimodules, 
this shows that the functor sending $x_g \in \Cat_G(\psi)$ to 
$X_g$ gives rise to a monoidal equivalence of $\Cat_G(\psi)$ 
and $(A_r\BMod[\Cat]A_r)_\mathrm{ss}$. 
Let us just note 

\begin{prop}
The 3-cocycle $\psi$ from \eqref{eq:3-cocycle} is not a coboundary.
\end{prop}
\begin{proof}
Suppose, to the contrary, that 
\begin{align} 
\psi(g_1,g_2,g_3) = d\varphi(g_1,g_2,g_3) = \varphi(g_2,g_3)\varphi(g_1g_2,g_3)^{-1}\varphi(g_1,g_2g_3)\varphi(g_1,g_2)^{-1} 
\end{align} 
for some map $\varphi: G \times G \to \C^\times$. 
Then for all $g,g'\in G$ we have (with $e = (1,1) \in G$ the neutral element) 
\begin{align} 
1 = \psi(g,e,g') = \varphi(e,g')\varphi(g,g')^{-1}\varphi(g,g')\varphi(g,e)^{-1} \ , 
\end{align}
hence, $\varphi(g,e) = \varphi(e,g')$. 
Now fix $g := (-1,-1) \in G$. Then we have $\psi(g,g,g) = -1 \neq 1$, 
but, with $g^2 = e$,
\begin{align} 
d\varphi(g,g,g) = \varphi(g,g)\varphi(g^2,g)^{-1}\varphi(g,g^2)\varphi(g,g)^{-1} = 1 \ , 
\end{align}
a contradiction.
\end{proof}

We are now in the position to prove Proposition \ref{prop:fcAB}. 

\begin{proof}[Proof of Proposition \ref{prop:fcAB}]
 Faithfulness is clear, since $\fcAB$ acts as the identity on hom-sets. 
 Exactness comes from $\fcAB$ taking direct sums to direct sums. 
 As $\mathcal{L}$ in \eqref{eq:comp_Psi} gives 
 rise to an injective map $\C^\times \to \C / r\Z$, 
 the functor $\fcAB$ maps non-isomorphic simple object to non-isomorphic 
 simple objects. Hence, $\fcAB$ is full. 
 
 Let $H \subseteq G$ be the subgroup $H = \{(\xi,1) \in G\}$. 
 To see monoidality of $\fcAB$, the key point is that $\psi$ restricted to $H$ is trivial, $\psi|_{H^{\times 3}} = 1$.
  Indeed, an exact functor sending 
 $\C_z \in \Repss(\algB)$ 
 ($\C_z$ as in Remark \ref{rem:image-of-bimod-functor}) 
 to $x_{(z,1)}$ gives rise 
 to a monoidal equivalence of $\Repss(\algB)$ and the 
 subcategory $\Cat_H(\psi|_{H^{\times 3}}) = \Cat_H(1)$. 
 Now any choices of isomorphisms 
 $\fcAB(\C_z) \cong X_{(z,1)}$ 
 may be used to obtain the structure maps needed 
 to turn $\fcAB$ into a monoidal functor. 
\end{proof}

\newcommand{\etalchar}[1]{$^{#1}$}
\renewcommand{\refname}{References}
\providecommand{\bysame}{\leavevmode\hbox to3em{\hrulefill}\thinspace}
\providecommand{\MR}{\relax\ifhmode\unskip\space\fi MR }
\providecommand{\MRhref}[2]{%
  \href{http://www.ams.org/mathscinet-getitem?mr=#1}{#2}
}
\providecommand{\href}[2]{#2}


\begin{thebibliography}{BGK{\etalchar{+}}12}

\bibitem[AS02]{AndSchn02}
N.~Andruskiewitsch and H.-J. Schneider, \emph{{Pointed Hopf algebras}}, New
  directions in Hopf algebras, MSRI series Cambridge Univ Press (2002), 1--68,
  \href{http://arxiv.org/abs/math/0110136}{[math.QA/0110136]}.

\bibitem[Bes97]{Be95}
Yu.N. Bespalov, \emph{Crossed modules and quantum groups in braided
  categories}, \href{http://dx.doi.org/10.1023/A:1008674524341}{Appl. Cat.
  Structures \textbf{5} (1997), 155--204},
  \href{http://arxiv.org/abs/q-alg/9510013}{[q-alg/9510013]}.

\bibitem[BGK{\etalchar{+}}13]{Boos12}
H.~Boos, F.~G\"ohmann, A.~Kl\"umper, Kh.S. Nirov, and A.V. Razumov,
  \emph{{Universal integrability objects}},
  Theor. Math. Phys. \textbf{174} (2013), 25--45, 
  \href{http://arxiv.org/abs/1205.4399}{[math-ph/1205.4399]}.

\bibitem[BKj00]{BaKi00}
B.~Bakalov and A.~Kirillov~jr., \emph{{Lectures on Tensor Categories and
  Modular Functors}}, University Lecture Series, vol.~21, American Mathematical
  Society, 2000.

\bibitem[BL93]{BL93}
D.~Bernard and A.~LeClair, \emph{{The quantum double in integrable quantum
  field theory}}, \href{http://dx.doi.org/10.1016/0550-3213(93)90515-Q}{Nucl.
  Phys. B \textbf{399} (1993), 709--748},
  \href{http://arxiv.org/abs/hep-th/9205064}{[hep-th/9205064]}.

\bibitem[BLZ96]{BLZ96}
V.V. Bazhanov, S.L. Lukyanov, and A.B. Zamolodchikov, \emph{{Integrable
  Structure of Conformal Field Theory, Quantum KdV Theory and Thermodynamic
  Bethe Ansatz}}, \href{http://dx.doi.org/10.1007/BF02101898}{Commun. Math.
  Phys. \textbf{177} (1996), 381--398},
  \href{http://arxiv.org/abs/hep-th/9412229}{[hep-th/9412229]}.

\bibitem[BLZ97a]{BLZ97b}
\bysame, \emph{{Integrable Quantum Field Theories in Finite Volume: Excited
  State Energies}},
  \href{http://dx.doi.org/10.1016/S0550-3213(97)00022-9}{Nucl. Phys.
  \textbf{B489} (1997), 487--531},
  \href{http://arxiv.org/abs/hep-th/9607099}{[hep-th/9607099]}.

\bibitem[BLZ97b]{BLZ97a}
\bysame, \emph{{Integrable structure of conformal field theory. II. Q-operator
  and DDV equation}}, \href{http://dx.doi.org/10.1007/s002200050240}{Commun.
  Math. Phys. \textbf{190} (1997), 247--278},
  \href{http://arxiv.org/abs/hep-th/9604044}{[hep-th/9604044]}.

\bibitem[BLZ99]{BLZ99}
\bysame, \emph{Integrable structure in conformal field theory {III}. {T}he
  {Y}ang-{B}axter {R}elation},
  \href{http://dx.doi.org/10.1007/s002200050531}{Commun. Math. Phys.
  \textbf{200} (1999), 297--324},
  \href{http://arxiv.org/abs/hep-th/9805008}{[hep-th/9805008]}.

\bibitem[BMR]{InProgr}
D.~B\"ucher, C.~Meneghelli, and I.~Runkel, 
	\emph{Functional analytic properties of non-local conserved charges 
	in quantum sine-Gordon theory}, 
  \emph{in preparation}.

\bibitem[CP91]{CP91}
V.~Chari and A.~Pressley, \emph{Quantum affine algebras},
  \href{http://dx.doi.org/10.1007/BF02102063}{Commun. Math. Phys. \textbf{142}
  (1991), 261--283}.

\bibitem[CP94]{CP94}
\bysame, \emph{{A guide to quantum groups}}, Cambridge University Press, 1994.

\bibitem[DKR11]{DKR11}
A.~Davydov, L.~Kong, and I.~Runkel, \emph{{Field theories with defects and the
  centre functor}}, Mathematical Foundations of Quantum Field Theory and
  Perturbative String Theory (H.~Sati and U.~Schreiber, eds.), Proc. Symp. Pure
  Math.\ {\bf 83} (2011), 71--128
  \href{http://arxiv.org/abs/1107.0495}{[math.QA/1107.0495]}.

\bibitem[Dri86]{Drin86}
V.G. Drinfel'd, \emph{{Quantum groups}}, Proceedings of the International
  Congress of Mathematicians (1986), 798--820,
  \href{http://www.mathunion.org/ICM/ICM1986.1/}{www.mathunion.org/ICM/ICM1986%
.1/}.

\bibitem[EFKj98]{EFK98}
P.I. Etingof, I.B. Frenkel, and A.~Kirillov~jr., \emph{{Lectures on
  Representation Theory and Knizhnik-Zamolodchikov Equations}}, Mathematical
  Surveys and Monographs, vol.~58, American Mathematical Society, 1998.

\bibitem[FFRS07]{FFRS06}
J.~Fr{\"o}hlich, J.~Fuchs, I.~Runkel, and C.~Schweigert, \emph{{Duality and
  defects in rational conformal field theory}},
  \href{http://dx.doi.org/10.1016/j.nuclphysb.2006.11.017}{Nucl. Phys.
  \textbf{B763} (2007), 354--430},
  \href{http://arxiv.org/abs/hep-th/0607247}{[hep-th/0607247]}.

\bibitem[FFS12]{FFS12}
J.~Fjelstad, J.~Fuchs, and C.~Stigner, \emph{{RCFT with defects: Factorization
  and fundamental world sheets}},
  \href{http://dx.doi.org/10.1016/j.nuclphysb.2012.05.011}{Nucl. Phys.
  \textbf{B863} (2012), 213--259},
  \href{http://arxiv.org/abs/1202.3929}{[hep-th/1202.3929]}.

\bibitem[FGRS07]{Fuchs07}
J.~Fuchs, M.R. Gaberdiel, I.~Runkel, and Christoph Schweigert,
  \emph{{Topological defects for the free boson CFT}},
  \href{http://dx.doi.org/10.1088/1751-8113/40/37/016}{J. Phys. \textbf{A40}
  (2007), 11403}, \href{http://arxiv.org/abs/0705.3129}{[hep-th/0705.3129]}.

\bibitem[FR99]{FR98}
E.~Frenkel and N.~Reshetikhin, \emph{The $q$-characters of representations of
  quantum affine algebras and deformations of {W}-algebras}, Contemp. Math.
  \textbf{248} (1999), 163--206,
  \href{http://arxiv.org/abs/math/9810055}{[math/9810055]}.

\bibitem[FRS02]{FRS02}
J.~Fuchs, I.~Runkel, and C.~Schweigert, \emph{{TFT construction of RCFT
  correlators I: partition functions}},
  \href{http://dx.doi.org/10.1016/S0550-3213(02)00744-7}{Nucl. Phys.
  \textbf{B646} (2002), 353--497},
  \href{http://arxiv.org/abs/hep-th/0204148}{[hep-th/0204148]}.

\bibitem[FRS05]{FRS05}
\bysame, \emph{{TFT construction of RCFT correlators IV: Structure constants
  and correlation functions}},
  \href{http://dx.doi.org/10.1016/j.nuclphysb.2005.03.018}{Nucl. Phys.
  \textbf{B715} (2005), 539--638},
  \href{http://arxiv.org/abs/hep-th/0412290}{[hep-th/0412290]}.

\bibitem[Hua05]{Huang:2003cq}
  Y.-Z.~Huang,
  \emph{Differential equations, duality and modular invariance},
  \href{http://dx.doi.org/10.1142/S021919970500191X}{Commun. Contemp. Math. {\bf 7} (2005) 649--706} 
   \href{http://arxiv.org/abs/math/0303049}{[math.QA/0303049]}.

\bibitem[Hua08]{Hu08}
Y.-Z. Huang, \emph{{Rigidity and modularity of vertex tensor categories}},
  \href{http://dx.doi.org/10.1142/S0219199708003083}{Commun. Contemp. Math.
  \textbf{10} (2008), 871--911},
  \href{http://arxiv.org/abs/math/0502533}{[math.QA/0502533]}.

\bibitem[JS91]{JS91}
A.~Joyal and R.~Street, \emph{{Tortile Yang-Baxter operators in tensor
  categories}}, J. Pure Appl. Algebra \textbf{71} (1991), 43--51.

\bibitem[JS93]{JS93}
\bysame, \emph{{Braided tensor categories}},
  \href{http://dx.doi.org/10.1006/aima.1993.1055}{Adv. Math. \textbf{102}
  (1993), 20--78}.

\bibitem[Kac98]{KacVOAs}
V.~Kac, \emph{Vertex algebras for beginners}, University Lecture Series,
  vol.~10, American Mathematical Society, 1998.

\bibitem[Kas95]{Kas95}
C.~Kassel, \emph{Quantum groups}, Springer-Verlag, 1995.

\bibitem[Kor03]{Kor03}
C.~Korff, \emph{{Auxiliary matrices for the six-vertex model at roots of 1 and
  a geometric interpretation of its symmetries}},
  \href{http://dx.doi.org/10.1088/0305-4470/36/19/305}{J Phys A: Math Gen
  \textbf{36} (2003), 5229--5266},
  \href{http://arxiv.org/abs/math-ph/0302002}{[math-ph/0302002]}.

\bibitem[Lus94]{Lus94}
G.~Lusztig, \emph{{Introduction to Quantum Groups}}, Birkh\"auser-Verlag, 1994.

\bibitem[Maj95]{Ma95}
S.~Majid, \emph{Algebras and {H}opf {A}lgebras in {Braided Categories}},
  \href{http://arxiv.org/abs/q-alg/9509023}{[q-alg/9509023]}, 1995.

\bibitem[Maj99]{Ma95b}
\bysame, \emph{{Double-Bosonisation and the Construction of {$U_q(g)$}}}, Math.
  Proc. Camb. Phil. Soc. \textbf{125} (1999), 151--192,
  \href{http://arxiv.org/abs/q-alg/9511001}{[q-alg/9511001]}.

\bibitem[ML98]{McL98}
S.~Mac~Lane, \emph{{Categories for the Working Mathematician}},
  Springer-Verlag, 1998.

\bibitem[MR09]{MR09}
D.~Manolopoulos and I.~Runkel, \emph{{A monoidal category for perturbed defects
  in conformal field theory}},
  \href{http://dx.doi.org/10.1007/s00220-009-0958-2}{Commun. Math. Phys.
  \textbf{295} (2009), 327--362},
  \href{http://arxiv.org/abs/0904.1122}{[hep-th/0904.1122]}.

\bibitem[MS89]{MooreSeib89}
G.~W. Moore and N.~Seiberg, \emph{{Classical and Quantum Conformal Field
  Theory}}, \href{http://dx.doi.org/10.1007/BF01238857}{Commun. Math. Phys.
  \textbf{123} (1989), 177--254}.

\bibitem[M{\"u}g03]{Mu03b}
M.~M{\"u}ger, \emph{{From Subfactors to Categories and Topology II. The quantum
  double of tensor categories and subfactors}},
  \href{http://dx.doi.org/10.1016/S0022-4049(02)00248-7}{J. Pure Appl. Alg.
  \textbf{180} (2003), 159--219},
  \href{http://arxiv.org/abs/math/0111205}{[math.CT/0111205]}.

\bibitem[NT10]{NT09}
G.~Niccoli and J.~Teschner, \emph{{The Sine-Gordon model revisited: I}},
  \href{http://dx.doi.org/10.1088/1742-5468/2010/09/P09014}{J. Stat. Mech.
  \textbf{1009} (2010), P09014},
  \href{http://arxiv.org/abs/0910.3173}{[hep-th/0910.3173]}.

\bibitem[RS09]{RuSu08}
I.~Runkel and R.R. Suszek, \emph{{Gerbe-holonomy for surfaces with defect
  networks}}, Adv. Theor. Math. Phys. \textbf{13} (2009), 1137--1219,
  \href{http://arxiv.org/abs/0808.1419}{[hep-th/0808.1419]}.

\bibitem[Run08]{Ru07}
I.~Runkel, \emph{{Perturbed Defects and T-Systems in Conformal Field Theory}},
  \href{http://dx.doi.org/10.1088/1751-8113/41/10/105401}{J. Phys.
  \textbf{A41} (2008), 105401},
  \href{http://arxiv.org/abs/0711.0102}{[hep-th/0711.0102]}.

\bibitem[Run10]{Ru10}
\bysame, \emph{{Non-local conserved charges from defects in perturbed conformal
  field theory}}, J. Phys. A: Math. Theor. \textbf{43} (2010), 365206,
  \href{http://arxiv.org/abs/1004.1909}{[hep-th/1004.1909]}.

\bibitem[Run14]{RuSympFerm}
\bysame, \emph{{A braided monoidal category for free super-bosons}},
  \href{http://dx.doi.org/10.1063/1.4868467}{J. Math. Phys. \textbf{55} (2014), 041702}, 
  \href{http://arxiv.org/abs/1209.5554}{[math.QA/1209.5554]}.

\bibitem[RW02]{RW02}
M.~Rossi and R.~Weston, \emph{{A generalized Q-operator for $U_q(\hat\lsl_2)$
  vertex models}}, \href{http://dx.doi.org/10.1088/0305-4470/35/47/304}{J.
  Phys. A: Math. Gen. \textbf{35} (2002), 10015--10032},
  \href{http://arxiv.org/abs/math-ph/0207004}{[math-ph/0207004]}.

\bibitem[Sch96]{Schau96}
P.~Schauenburg, \emph{A characterization of the borel-like subalgebras of
  quantum enveloping algebras},
  \href{http://dx.doi.org/10.1080/00927879608825714}{Comm. Alg. \textbf{24}
  (1996), 2811--2823}.

\bibitem[Sch01]{Sch01}
\bysame, \emph{{The monoidal center construction and bimodules}},
  \href{http://dx.doi.org/10.1016/S0022-4049(00)00040-2}{J. Pure Appl. Algebra
  \textbf{158} (2001), 325--346}.

\bibitem[Seg02]{SegalDef}
G.~Segal, \emph{{The definition of conformal field theory}}, Topology, geometry
  and quantum field theory (U.~Tillmann, ed.), London Math. Soc. Lect. Note
  Ser., vol. 308, 2004, pp.~421--577.

\bibitem[Sem12]{Sem11}
A.M. Semikhatov, \emph{{Fusion in the entwined category of Yetter--Drinfeld
  modules of a rank-1 Nichols algebra}},
  Theor. Math. Phys. \textbf{173} (2012), 1329--1358,
  \href{http://arxiv.org/abs/1109.5919}{[math.QA/1109.5919]}, 2011.

\bibitem[Sem14]{Sem11a}
\bysame, \emph{{Virasoro central charges for Nichols algebras}},
  Conformal Field Theories and Tensor Categories, 
  Springer, 2014, pp. 67--92,
  \href{http://arxiv.org/abs/1109.1767}{[math.QA/1109.1767]}.

\bibitem[ST12]{SemTip11}
A.M. Semikhatov and I.Yu. Tipunin, \emph{{The Nichols algebra of screenings}},
  \href{http://dx.doi.org/10.1142/S0219199712500290}{Commun. Contemp. Math.
  \textbf{14} (2012), 1250029},
  \href{http://arxiv.org/abs/1101.5810}{[math.QA/1101.5810]}.

\bibitem[Vaf87]{Vafa87}
C.~Vafa, \emph{{Conformal theories and punctured surfaces}},
  \href{http://dx.doi.org/10.1016/0370-2693(87)91358-X}{Phys. Lett.
  \textbf{B199} (1987), 195--202}.

\bibitem[Yet90]{Yet90}
D.N. Yetter, \emph{Quantum groups and representations of monoidal categories},
  \href{http://dx.doi.org/10.1017/S0305004100069139}{Math. Proc. Camb. Phil.
  Soc \textbf{108} (1990), 261--290}.

\end{thebibliography}
\end{document}